\documentclass[a4paper]{report}
\usepackage{mathrsfs}

\RequirePackage{etoolbox}
\newtoggle{DissertateSingleSpace}
\toggletrue{DissertateSingleSpace}
\relax

\iftoggle{DissertateSingleSpace}{
    \newcommand{\dnormalspacing}{1.2} %1.2
    \newcommand{\dcompressedspacing}{1.0} %1.0
}{
    \newcommand{\dnormalspacing}{2.0} %2.0
    \newcommand{\dcompressedspacing}{1.2} %1.2
}

\let\oldquote\quote
\let\endoldquote\endquote
\renewenvironment{quote}
    {\begin{spacing}{\dcompressedspacing}\oldquote}
    {\endoldquote\end{spacing}}

\let\olditemize\itemize
\let\endolditemize\enditemize
\renewenvironment{itemize}
    {\begin{spacing}{\dcompressedspacing}\olditemize}
    {\endolditemize\end{spacing}}

\let\oldenumerate\enumerate
\let\endoldenumerate\endenumerate
\renewenvironment{enumerate}
    {\begin{spacing}{\dcompressedspacing}\oldenumerate}
    {\endoldenumerate\end{spacing}}

\RequirePackage[outer=1.2in, inner=1.2in, twoside, a4paper]{geometry}
\RequirePackage{graphicx}
\RequirePackage{lettrine}
\RequirePackage{setspace}
\RequirePackage{verbatim}
\RequirePackage{amsmath}
\RequirePackage{mathtools}
\DeclarePairedDelimiter{\abs}{\lvert}{\rvert}
\RequirePackage{siunitx}

\RequirePackage{color}
\RequirePackage{hyperref}
\RequirePackage{url}
\RequirePackage{amssymb}

\newlength\fheight
\newlength\fwidth
\RequirePackage{fancyhdr}
\RequirePackage[medium, md, sc]{titlesec} % format for the title of the sections
\RequirePackage{titling}

\RequirePackage[comma,numbers]{natbib}

\RequirePackage[palatino]{quotchap}

\def\advisor#1{\gdef\@advisor{#1}}
\def\mastername#1{\gdef\@mastername{#1}}
\def\studentId#1{\gdef\@studentId{#1}}
\def\coadvisorOne#1{\gdef\@coadvisorOne{#1}}
\def\coadvisorsUniversity#1{\gdef\@coadvisorsUniversity{#1}}
\def\academicYear#1{\gdef\@academicYear{#1}}
\def\department#1{\gdef\@department{#1}}
\def\university#1{\gdef\@university{#1}}

\definecolor{SchoolColor}{rgb}{0.71, 0, 0.106} % UNIPD red
\definecolor{chaptergrey}{rgb}{0.61, 0, 0.09} % dialed back a little
\definecolor{midgrey}{rgb}{0.4, 0.4, 0.4}

\hypersetup{
    colorlinks,
    citecolor=black,
    filecolor=black,
    linkcolor=black,
    urlcolor=black,
}

\newtheorem{theorem}{Theorem}[section]
\newtheorem{lemma}[theorem]{Lemma}
\newtheorem{proposition}[theorem]{Proposition}
\newtheorem{corollary}[theorem]{Corollary}
\newtheorem{remark}[theorem]{Remark}
\newtheorem{definition}[theorem]{Definition}

\usepackage{braket}

\DeclarePairedDelimiter{\norma}{\lVert}{\rVert}

\newcommand{\Rn}{\mathbb R^n}
\newcommand{\Rnn}{\mathbb R^{n+1}}
\newcommand{\Sn}{\mathscr{S}(\Rn)}
\newcommand{\Snn}{\mathscr{S}(\Rnn)}
\newcommand{\dive}{\text{div}}

\usepackage{scalerel,stackengine}
\stackMath
\newcommand\reallywidehat[1]{%
\savestack{\tmpbox}{\stretchto{%
  \scaleto{%
    \scalerel*[\widthof{\ensuremath{#1}}]{\kern-.6pt\bigwedge\kern-.6pt}%
    {\rule[-\textheight/2]{1ex}{\textheight}}%WIDTH-LIMITED BIG WEDGE
  }{\textheight}% 
}{0.5ex}}%
\stackon[1pt]{#1}{\tmpbox}%
}
\usepackage{dsfont}

\newenvironment{proof}{\paragraph{Proof:}}{\hfill$\square$}

\begin{document}

%%%%%%%%%%%%%%%%%
%%%%%FRONTESPIZIO%%%%%
%%%%%%%%%%%%%%%%%

\thispagestyle{empty}
\setcounter{page}{0}

\begin{center}
\includegraphics[width=%
0.3\textwidth]{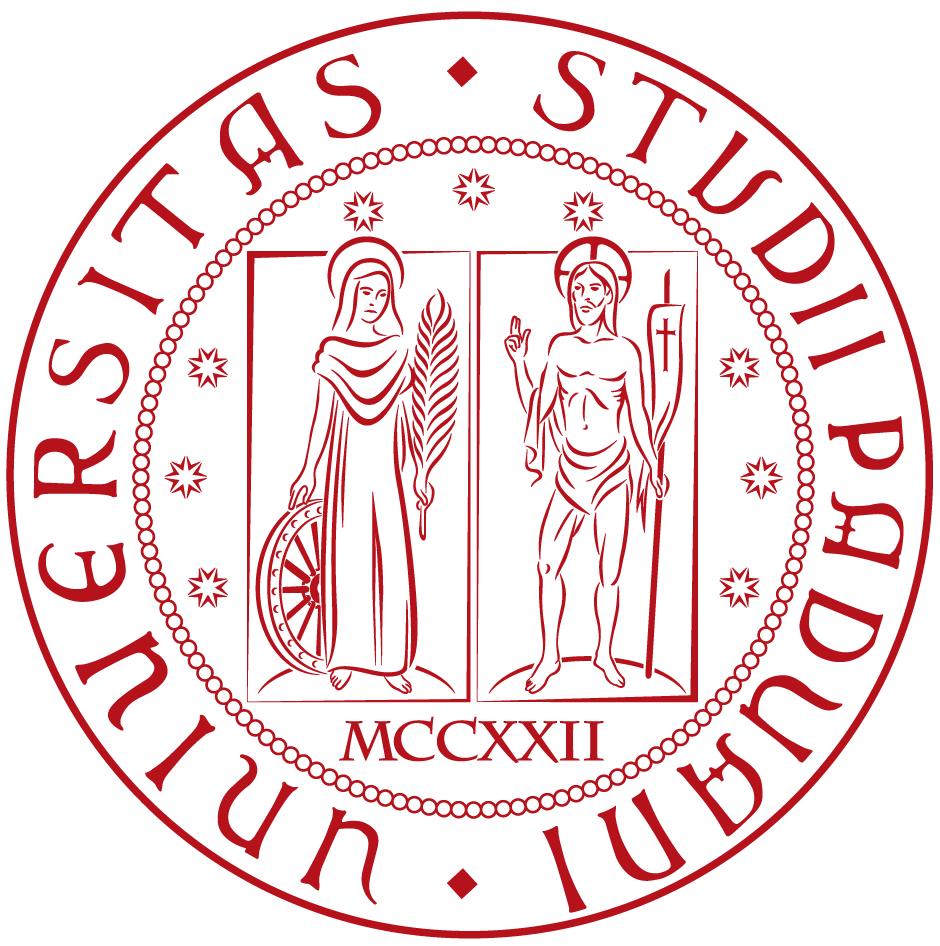}
\end{center}

\begin{center}
\Large
\textsc{University of Padova}\\
\vspace{-0.4cm}
\rule{\textwidth}{0.1mm}\\
\large
\textsc{Department of Civil, Environmental and Architectural Engineering}\\ 
\bigskip
Master Thesis in Mathematical Engineering \\
\vfill %%\vspace{3cm}
\Large
\textsc{The extension problem for fractional powers}\\
\textsc{of higher order of some evolutive operators}
\end{center}
\vfill
\begin{tabular}[t]{l}
Master Candidate:\\
Pietro Gallato
\end{tabular}
%%%%%%
\hfill 
\begin{tabular}[t]{l}
Supervisor: \\
Prof. Nicola Garofalo
\end{tabular}
\vfill
\begin{center}
\normalsize
\rule{8cm}{0.1mm}\\
\bigskip
Academic Year \\ 2021-2022 \\ 05/10/2022
\end{center}

%%%%%%%%%%%%%%%%
%%%%%DEDICATION%%%%%
%%%%%%%%%%%%%%%%

\phantomsection
\setcounter{page}{2}
\vspace*{\fill}
\rm
\cleardoublepage
\setstretch{\dnormalspacing}
“If people do not believe that mathematics is simple, it is only because they do not realize how complicated life is.” \\
\footnotesize
--- John Von Neumann

%%%%%%%%%%%%%%%
%%%%%ABSTRACT%%%%%
%%%%%%%%%%%%%%%

\phantomsection
\setcounter{page}{3}
\addcontentsline{toc}{chapter}{Abstract}
\chapter*{Abstract}
This thesis studies the extension problem for higher-order fractional powers of the heat operator $H=\Delta-\partial_t$ in $\mathbb{R}^{n+1}$. Specifically, given $s>0$ and indicating with $[s]$ its integral part, we study the following degenerate partial differential equation in the thick space $\mathbb{R}^{n+1}\times \mathbb{R}_y^+$,
\begin{equation}
    \label{a:1}
    \mathscr{H}^{[s]+1}U= \left( \partial_{yy} +\frac{a}{y}\partial_y +H \right)^{[s]+1}U=0.
\end{equation}
The connection between the Bessel parameter $a$ in \eqref{a:1} and the fractional parameter $s>0$ is given by the equation
\begin{equation*}
    a= 1-2(s-[s]).
\end{equation*}
When $s\in(0,1)$ this equation reduces to the well-known relation $a=1-2s$, and in such case \eqref{a:1} becomes the famous Caffarelli-Silvestre extension problem. Generalising their result, in this thesis we show that the nonlocal operator $(-H)^{\,s}$ can be realised as the Dirichlet-to-Neumann map associated with the solution $U$ of the extension equation \eqref{a:1}.

In this thesis we systematically exploit the evolutive semigroup $\{P_{\tau}^H \}_{\tau>0}$, associated with the Cauchy problem
\begin{equation*}
    \begin{cases}
    \partial_{\tau}u-Hu=0\\
    u((x,t),0)=f(x,t).
    \end{cases}
\end{equation*}
This approach provides a powerful tool in analysis, and it has the twofold advantage of allowing an independent treatment of several complex calculations involving the Fourier transform, while at same time extending to frameworks where the Fourier transform is not available.

\tableofcontents

\phantomsection
\clearpage
\setcounter{chapter}{0}

%%%%%%%%%%%%%%%%%
%%%%%INTRODUCTION%%%%%
%%%%%%%%%%%%%%%%%

\chapter{Introduction}
\label{chp:intro}

In his visionary papers \citep{R38} and \citep{R49} Marcel Riesz introduced the fractional powers of the Laplacean in Euclidean and Lorentzian space, developed the calculus of these nonlocal operators and studied the Dirichlet and Cauchy problems for respectively $(-\Delta)^s$ and $(\partial_{tt}-\Delta)^s$.

Pseudo-differential operators such as $(-\Delta)^s$, $(\partial_{tt}-\Delta)^s$, $(\partial_t -\Delta)^s$ play an important role in many branches of the applied sciences ranging from fluid dynamics, to elasticity and to quantum mechanics.

Our objective in this thesis is studying the extension problem for higher-order fractional powers of the heat operator $H=\Delta-\partial_t$ in $\Rnn$. Specifically, given $s>0$ and indicating with $[s]$ its integral part, we study the following degenerate partial differential equation in the thick space $\Rnn \times \mathbb{R}_y^+$,
\begin{equation}
    \label{int1}
    \mathscr{H}^{[s]+1}U= \left( \partial_{yy}+\frac{a}{y}\partial_y+H \right)^{[s]+1}U=0.
\end{equation}
The connection between the Bessel parameter $a$ in \eqref{int1} and the fractional non-integer parameter $s>0$ is given by the equation
\begin{equation*}
    a=1-2(s-[s]).
\end{equation*}
When $s\in(0,1)$ this equation reduces to the well-known relation $a=1-2s$, and in such case \eqref{int1} becomes the famous Caffarelli-Silvestre extension problem. Generalising their result, in this thesis we show that the nonlocal operator $(-H)^s$ can be realised as the Dirichlet-to-Neumann map associated with the solution $U$ of the extension equation \eqref{int1}.

A list of the topics covered by this thesis is provided by the table of contents. Diving deeper into the details:
\begin{enumerate}
    \item In Chapter \ref{chp:2} we are presenting some of the most basic aspects of the operator $(-\Delta)^s$, a complete introduction of which is available at \citep{Gft}. In particular we have
    \begin{itemize}
        \item In Section \ref{s:1.1} we introduce the main pointwise definition of the nonlocal operator $(-\Delta)^s$, see \eqref{fraclap} below. In Proposition \ref{p:decay} we show that the definition \eqref{fraclap} implies a decay at infinity of the fractional Laplacean that plays an important role in its analysis.
        \item Section \ref{s:1.2} constitutes a brief interlude on two important protagonist of classical analysis which also play a central role in this chapter: the Fourier transform and Bessel functions. These two classical subjects are inextricably connected. One the one hand, the Bessel functions are eigenfunctions of the Laplacean. On the other, they also appear as the Fourier transform of the measure carried by the unit sphere. In this connection, and since it is a recurrent ingredient in this note, we recall the classical Fourier-Bessel integral formula due to Bochner, see Theorem \ref{t:FBr} below.
        \item Section \ref{s:1.3} opens with the proof of Proposition \ref{p:pdnfp}, which describes the action of $(-\Delta)^s$ on the Fourier transform side. This result proves an important fact: the fractional Laplacean is a \emph{pseudo-differential operator}. A basic consequence of Proposition \ref{p:pdnfp} is the semigroup property in Corollary \ref{semigroup} and the "integration by parts" Lemma \ref{l:Symmetry}, which shows that $(-\Delta)^s$ is a symmetric operator. We close the section with the computation in Proposition \ref{p:gammacomp} of the normalization constant $\gamma(n,s)$ in the pointwise definition \eqref{fraclap}.
        \item In Section \ref{s:fs} we want to find the fundamental solution of $(-\Delta)^s$, i.e., proving Theorem \ref{t:84}. This can be done in several ways, but we choose to exploit the tools provided in \citep{Gft}.
        \item Section \ref{s:expb} presents in detail the central theme of the analysis of the fractional Laplacean: the extension problem of Caffarelli and Silvestre \eqref{101}. We construct the Poisson kernel for the extension operator, and provide two proofs of \eqref{102}, which characterizes $(-\Delta)^s$ as the weighted Dirichlet-to-Neumann map of the extension problem. The extension procedure is a very powerful tool which has been applied so far in many different directions, and it is hardly possible to accurately describe the impact of this paper in the field.
    \end{itemize}
    \item In Chapter \ref{chp:3} we study the fractional operators, in particular $(-\Delta)^s$ and $(\partial_t-\Delta)^s$, with the systematic use of the heat semigroup $\{P_t\}_{t\ge0}$. The semigroup approach provides a powerful tool in analysis and has a twofold advantage. On one hand it allows a treatment independent of several complex calculations involving the Fourier transform and, more importantly, it extends to frameworks in which the Fourier transform is not available. References on the methods and tools exploited in this chapter can be found at \citep{Gfc}. In particular we have
    \begin{itemize}
        \item In Section \ref{s:61} we define in \eqref{62} the heat semigroup and its main properties. The name is justified by the fact that the function $u(x,t)=P_tf(x)$ solves the Cauchy problem for the heat equation $\partial_tu-\Delta u=0$ in $\Rn \times \mathbb{R}^+$.
        \item Section \ref{s:62} opens with the Proposition \ref{p:610}, which basically shows the ultracontractivity property of the heat semigroup.
        \item In Section \ref{s:64} we define the fractional Laplacean according to the formulation of Balakrishnan in Definition \ref{d:615}. We then proceed with the proof of some properties of the fractional Laplacean, similarly to what is done in Chapter \ref{chp:2}, but now leveraging the advantages that the heat semigroup grants.
        \item In Section \ref{s:65} we show that Balakrishnan's definition of the nonlocal operator $(-\Delta)^{\frac{\alpha}{2}}$ coincides with that introduced by M. Riesz in \citep{R38}. Subsequently, we analyse the asymptotic behaviour of this operator as $\alpha \nearrow 2$ and we show that, unsurprisingly, in the limit we obtain the negative of the Laplace operator $\Delta$.
        \item In semigroup theory a procedure for forming a new semigroup from a given one is that of the evolution semigroup. In Section \ref{s:68} we exploit this idea to introduce a new semigroup that will be used as a building block for: (i) defining the fractional powers of the heat operator $H=\Delta -\partial_t$; (ii) solving the extension problem for such nonlocal operators.
        \item In Section \ref{s:69} we define the fractional heat operator $(\partial_t - \Delta)^{\frac{\alpha}{2}}$ through the evolutive heat semigroup and we show some of its basic properties, similarly to what we have done in Section \ref{s:64}. 
        \item In Section \ref{s:teb} we solve the extension problem for the fractional heat operator. In Definition \ref{d:649} we define the Poisson kernel for the extension problem, and this allows us to define the following function
        \begin{equation*}
            U(x,y,t):=\int_0^{\infty}\int_{\Rn}\mathscr{P}^{(s)}(x-z,y,\tau)f(z,t-\tau)\,dz\,d\tau.
        \end{equation*}
        To reach our goal, i.e. to solve the extension problem, we make the crucial observation that $U$ can be written in the following form using the evolutive heat semigroup $P_{\tau}^H$
        \begin{equation*}
            U(x,y,t):=\frac{1}{2^{1-a}\Gamma \left( \frac{1-a}{2} \right)}y^{1-a}\int_0^{\infty} \frac{1}{\tau^{\frac{3-a}{2}}}e^{-\frac{y^2}{4\tau}}P_{\tau}^H f(x,t)\,d\tau.
        \end{equation*}
    \end{itemize}
    \item In Chapter \ref{chp:4} we first want to define and then solve the extension problem for the  fractional powers of higher order of $\partial_t-\Delta$. In particular we have
    \begin{itemize}
        \item In Section \ref{s:foho} we want to introduce the fractional operators of higher order considered in the previous chapters. In order to do this, we use the Balakrishnan formulation, which permits to give a natural generalization to higher order in both cases.
        \item Our goal in Section \ref{s:epho} is to give the statement of the extension problem of higher order and prove it in Subsection \ref{su:1}
    \end{itemize}
    \item In the concluding remarks \ref{chp:c} we comment on how to generalize the results obtained in Chapter \ref{chp:4}, and specifically in a forthcoming work we are going to solve the extension problem of higher order for the following class of operators
    \begin{equation*}
        \mathscr{K}u:= \text{tr}\,(Q\nabla^2u)+ \braket{BX,\nabla u} -\partial_t u,
    \end{equation*}
    where $Q$ and $B$ are $N\times N$ matrices with real, constant coefficients, with $Q\ge 0, Q=Q^*$.
\end{enumerate}

%%%%%%%%%%%%%%%%%
%%%%%CHAPTER 2%%%%%%%
%%%%%%%%%%%%%%%%%

%!TEX root = ../dissertation.tex

\chapter{Fractional Laplacean}
\label{chp:2}

%prima sezione
\section{The fractional Laplacean} \label{s:1.1}
In this section we introduce the M. Riesz' fractional Laplacean $(-\Delta)^s$, with $0<s<1$. Our first goal is to give a definition of this nonlocal operator.

Our initial observation is the following simple calculus lemma which could be used to provide a probabilistic interpretation oh the classical Laplacean on the real line.
\begin{lemma}\label{l:d2line}
Let $f \in C^{\,2}(a,b)$, then for every $x\in (a,b)$ one has
\begin{equation*}
    -f\,''(x)=\lim_{y\to 0}\frac{2f(x)-f(x+y)-f(x-y)}{y^2}
\end{equation*}
\end{lemma}
The expression in the right-hand side in the equation in Lemma \ref{l:d2line} is known as the symmetric difference quotient of order two. If we introduce the "spherical" surface and "solid" averaging operators
\begin{equation*}
    \mathscr{M}_y\,f(x)=\frac{f(x+y)+f(x-y)}{2}, \quad \mathscr{A}_y\,f(x)=\frac{1}{2y}\int_{x-y}^{x+y}f(t)\,dt,
\end{equation*}
then we can reformulate the conclusion in Lemma \ref{l:d2line} as follows:
\begin{equation*}
    -f\,''(x)=2\lim_{y\to 0}\frac{f(x)-\mathscr{M}_y\,f(x)}{y^2}=6\lim_{y\to 0}\frac{f(x)-\mathscr{A}_y\,f(x)}{y^2}
\end{equation*}
where it is easily seen that the second equality follows from the first one and L'Hopital's rule. The result that follows generalizes this observation to $n\ge2$.
\begin{proposition}
Let $\Omega \in \mathbb{R}^n$ be an open set. For any $f\in C^2(\Omega)$ and $x\in \Omega$ we have
\begin{equation}\label{lapn}
    -\Delta f(x)=2n\lim_{r\to 0}\frac{f(x)-\mathscr{M}_r\,f(x)}{r^2}=2(n+2)\lim_{r\to 0}\frac{f(x)-\mathscr{A}_r\,f(x)}{r^2}
\end{equation}
where $\Delta f$ is the operator of Laplace.
\end{proposition}
In the equation \eqref{lapn} we have indicated with
\begin{equation} \label{means}
    \mathscr{M}_ru(x)=\frac{1}{\sigma_{n-1}r^{n-1}}\int_{S(x,r)}u(y)\,d\sigma(y), \quad \mathscr{A}_ru(x)=\frac{1}{\omega_{n}r^{n}}\int_{B(x,r)}u(y)\,dy,
\end{equation}
the spherical surface and solid mean-value operators. Here, $B(x,r)=\{y\in\mathbb{R}^n : \abs{y-x}<r\}, \, \, S(x,r)=\partial B(x,r)$, $d\sigma$ is the $(n-1)$-Lebesgue measure on $S(x,r)$, and the numbers $\sigma_{n-1}$ and $\omega_n$ respectively represent the measure of the unit sphere and that of the unit ball in $\mathbb{R}^n$.

Before proceeding, and in preparation for the central definition of this section, let us observe that it is easy to recognize that we can write the second identity in \eqref{lapn} in the more suggestive fashion:
\begin{equation}\label{fashLap}
    -\Delta u(x)= c(n)\lim_{r\to0^+}\int_{\mathbb{R}^n}\frac{2u(x)-u(x+y)-u(x-y)}{r^{n+2}}\mathds{1}_{B(0,r)}(y)\,dy,
\end{equation}
where we have denoted by $\mathds{1}_E$ the indicator function of a set $E\subset\mathbb{R}^n$.

In the applied sciences it is of great importance to be able to consider fractional derivatives of functions. There exist many different definitions of such operations, but perhaps the most prominent one is based on the notion of (Marcel) Riesz' potential of a function. To motivate such operation let us assume that $n\ge3$, and recall that in mathematical physics the Newtonian potential of a function $f\in\mathscr{S}(\mathbb{R}^n)$ is given by 
\begin{equation*}
    I_2(f)(x)=\frac{1}{4\pi^{\frac{n}{2}}}\Gamma \left(\frac{n-2}{2} \right) \int_{\mathbb{R}^n} \frac{f(y)}{\abs{x-y}^{n-2}}\,dy,
\end{equation*}
Now, one recognizes that the convolution kernel $\frac{1}{4\pi^{\frac{n}{2}}}\Gamma \left(\frac{n-2}{2} \right)\frac{1}{\abs{x}^{n-2}}$ in the definition of $I_2(f)$ is just the fundamental solution
\begin{equation*}
    E(x)=\frac{1}{(n-2)\sigma_{n-1}}\frac{1}{\abs{x}^{n-2}}
\end{equation*}
of $-\Delta$. With this observation in mind, we recall the well-known identity of Gauss-Green that says that for any $f\in\mathscr{S}(\mathbb{R}^n)$one has
\begin{equation*}
    I_2(-\Delta f)=f.
\end{equation*}
In other words, the Newtonian potential is the inverse of $-\Delta$. This important observation leads to the introduction of M. Riesz' generalization of the Newtonian potential.
\begin{definition}[Riesz' potentials]
\label{d:Rp}
For any $n\in\mathbb{N}$, let $0<\alpha<n$. The Riesz potential of order $\alpha$ is the operator whose action on a function $f\in\mathscr{S}(\mathbb{R}^n)$ is given by
\begin{equation*}
    I_{\alpha}(f)(x)=\frac{\Gamma \left( \frac{n-\alpha}{2}\right)}{\pi^{\frac{n}{2}}2^{\alpha}\Gamma \left(\frac{\alpha}{2}\right)}\int_{\mathbb{R}^n}\frac{f(y)}{\abs{x-y}^{n-\alpha}}\,dy.
\end{equation*}
\end{definition}
The important reason behind the normalization constant is that such constant is chosen to guarantee the validity of the following crucial result, a kind of fractional fundamental theorem of calculus, stating that for any $f\in\mathscr{S}(\mathbb{R}^n)$ one has in $\mathscr{S}\,'(\mathbb{R}^n)$
\begin{equation}
    \label{rpfo}
    I_{\alpha}(-\Delta)^{\frac{\alpha}{2}}f=(-\Delta)^{\frac{\alpha}{2}}I_{\alpha}f.
\end{equation}
Of course \eqref{rpfo} makes no sense unless we say what we mean by the fractional operator $(-\Delta)^{\frac{\alpha}{2}}$. The most natural way to introduce it is by defining the action of $(-\Delta)^{\frac{\alpha}{2}}$ on the Fourier transform side by the equation
\begin{equation}
    \label{fourierfrac}
    \mathscr{F}((-\Delta)^{\frac{\alpha}{2}}u)=(2\pi\abs{\cdot})^{\alpha}\mathscr{F}(u), \quad u\in\mathscr{S}\,'(\mathbb{R}^n).
\end{equation}
The equation \eqref{rpfo} shows that $I_{\alpha}$ inverts the fractional powers of the Laplacean, i.e.,
\begin{equation}\label{fracint}
    I_{\alpha}=(-\Delta)^{-\frac{\alpha}{2}}, \quad 0<\alpha<n.
\end{equation}
For this reason $I_{\alpha}$ is also called the \emph{fractional integration operator} of order $\alpha$.

Since our focus in this section is the fractional Laplacean $(-\Delta)^s$ in the range $0<s<1$, we will henceforth let $s=\alpha/2$ in the above formulas. Although we have formally introduced such operator in the equation \eqref{fourierfrac} above, such definition has a major drawback:it is not easy understand a given function (or a distribution) by prescribing its Fourier transform. It is for this reason that we begin our story introducing a different pointwise definition of the fractional Laplacean that is more directly connected to the symmetric difference quotient of order two in the opening calculus Lemma \ref{l:d2line}, and with \eqref{fashLap}.
\begin{definition}\label{d:fraclap}
    Let $0<s<1$. The fractional Laplacean of a function $u\in\mathscr{S}(\mathbb{R}^n)$ is the nonlocal operator in $\mathbb{R}^n$ defined by the expression
    \begin{equation}\label{fraclap}
        (-\Delta)^su(x)=\frac{\gamma(n,s)}{2}\int_{\mathbb{R}^n}\frac{2u(x)-u(x+y)-u(x-y)}{\abs{y}^{n+2s}}\,dy,
    \end{equation}
    where $\gamma(n,s)>0$ is a suitable normalization constant that will be given implicitly in the future.
\end{definition}
It is obvious that \eqref{fraclap} defines a linear operator since for any $u,v\in\mathscr{S}(\mathbb{R}^n)$ and $c\in\mathbb{R}$ one has
\begin{equation*}
    (-\Delta)^s(u+v)=(-\Delta)^su+(-\Delta)^sv, \quad (-\Delta)^s(cu)=c(-\Delta)^su.
\end{equation*}
It is also important to observe that the integral in the right-hand side of \eqref{fraclap} is convergent. To see this, it suffices to write
\begin{equation*}
    \begin{split}
        \int_{\mathbb{R}^n}\frac{2u(x)-u(x+y)-u(x-y)}{\abs{y}^{n+2s}}\,dy&=\int_{\abs{y}\le 1}\frac{2u(x)-u(x+y)-u(x-y)}{\abs{y}^{n+2s}}\,dy\\
        &+\int_{\abs{y}> 1}\frac{2u(x)-u(x+y)-u(x-y)}{\abs{y}^{n+2s}}\,dy.
    \end{split}
\end{equation*}
Taylor's formula for $C^2$ functions gives for $\abs{y}\le 1$
\begin{equation*}
    2u(x)-u(x+y)-u(x-y)=-\Braket{\nabla^2u(x)y,y}+o(\abs{y}^2),
\end{equation*}
where we have indicated with $\nabla^2u$ the Hessian matrix of $u$. Therefore,
\begin{equation*}
    \abs*{\int_{\abs{y}\le 1}\frac{2u(x)-u(x+y)-u(x-y)}{\abs{y}^{n+2s}}\,dy}\le C \int_{\abs{y}\le 1} \frac{dy}{\abs{y}^{n-2(1-s)}}<\infty,
\end{equation*}
since $0<s<1$. On the other hand, keeping in mind that $u\in\mathscr{S}(\mathbb{R}^n)$ implies in particular that $u\in L^{\infty}(\mathbb{R}^n)$, we have
\begin{equation*}
    \abs*{\int_{\abs{y}> 1}\frac{2u(x)-u(x+y)-u(x-y)}{\abs{y}^{n+2s}}\,dy}\le 4\norma{u}_{L^{\infty}(\mathbb{R}^n)}\int_{\abs{y}>1}\frac{dy}{\abs{y}^{n+2s}}<\infty.
\end{equation*}
We have seen that for every $u\in \mathscr{S}(\mathbb{R}^n)$ definition \ref{d:fraclap} provides a well-defined function on $\mathbb{R}^n$.

Two basic operations in analysis are the Euclidean translations and dilations
\begin{equation*}
    \tau_hf(x)=f(x+h), \quad h \in \mathbb{R}^n, \qquad \delta_{\lambda}f(x)=f(\lambda x), \quad \lambda>0.
\end{equation*}
The next result clarifies the interplay of $(-\Delta)^s$ with them.
\begin{lemma}
For every function $u\in \mathscr{S}(\mathbb{R}^n)$ we have for every $h\in\mathbb{R}^n$
\begin{equation*}
    (-\Delta)^s(\tau_hu)=\tau_h((-\Delta)^su),
\end{equation*}
and every $\lambda>0$
\begin{equation*}
    (-\Delta)^s(\delta_{\lambda}u)=\lambda^{2s}\delta_{\lambda}((-\Delta)^su).
\end{equation*}
\end{lemma}

A fundamental property of the Laplacean $\Delta$ is its invariance with respect to the action of the orthogonal group $\mathbb{O}(n)$ on $\mathbb{R}^n$. This means that if $u$ is a function in $\mathbb{R}^n$, then for every $T\in \mathbb{O}(n)$ one has $\Delta(u \circ T)=\Delta u \circ T$. The following lemma shows that $(-\Delta)^s$ enjoys the same property.
\begin{lemma}
Let $u(x)=f(\abs{x})$ be a function with spherical symmetry in $C^2(\mathbb{R}^n)\cap L^{\infty}(\mathbb{R}^n)$. Then, also $(-\Delta)^su$ has spherical symmetry.
\end{lemma}
\begin{proof}
This follows in a simple way from \eqref{fraclap}. In order to prove that $(-\Delta)^su$ is spherically symmetric we need to show that for every $T\in\mathbb{O}(n)$ and every $x\in\Rn$ one has
\begin{equation*}
    (-\Delta)^su(Tx)=(-\Delta)^su(x).
\end{equation*}
We have
\begin{equation*}
    \begin{split}
        (-\Delta)^su(Tx)&= \frac{\gamma(n,s)}{2}\int_{\Rn}\frac{2f(\abs{Tx})-f(\abs{Tx+y})-f(\abs{Tx-y})}{\abs{y}^{n+2s}}\,dy\\
        &=\frac{\gamma(n,s)}{2}\int_{\Rn}\frac{2f(\abs{x})-f(\abs{x+T^ty})-f(\abs{x-T^ty})}{\abs{y}^{n+2s}}\,dy.
    \end{split}
\end{equation*}
If we make the change of variable $z=T^ty$, we conclude
\begin{equation*}
    \begin{split}
        (-\Delta)^su(Tx)&=\frac{\gamma(n,s)}{2}\int_{\Rn}\frac{2f(\abs{x})-f(\abs{x+z})-f(\abs{x-z})}{\abs*{Tz}^{n+2s}}\,dz\\
        &=\frac{\gamma(n,s)}{2}\int_{\Rn}\frac{2f(\abs{x})-f(\abs{x+z})-f(\abs{x-z})}{\abs*{z}^{n+2s}}\,dz\\
        &=(-\Delta)^su(x),
    \end{split}
\end{equation*}
And we are done.
\end{proof}

Before proceeding we note the following alternative expression for $(-\Delta)^s$ that is at times quite useful in the computations.
\begin{proposition}
For any $u\in \Sn$ one has
\begin{equation}
    \label{fraclap2}
    (-\Delta)^su(x)=\gamma(n,s)\, \text{PV}\int_{\Rn}\frac{u(x)-u(y)}{\abs{x-y}^{n+2s}}\,dy,
\end{equation}
where now the integral is taken according to Cauchy's principal value sense
\begin{equation*}
    \text{PV}\int_{\Rn}\frac{u(x)-u(y)}{\abs{x-y}^{n+2s}}\,dy=\lim_{\varepsilon\to0^+}\int_{\abs{y-x}>\varepsilon}\frac{u(x)-u(y)}{\abs{x-y}^{n+2s}}\,dy
\end{equation*}
\end{proposition}
\begin{proof}
The expression \eqref{fraclap2} follows directly from \eqref{fraclap} above as follows
\begin{equation*}
    \begin{split}
        \frac{1}{2}&\int_{\Rn}\frac{2u(x)-u(x+y)-u(x-y)}{\abs{y}^{n+2s}}\,dy=\frac{1}{2}\lim_{\varepsilon\to 0}\int_{\abs{y}>\varepsilon}\frac{2u(x)-u(x+y)-u(x-y)}{\abs{y}^{n+2s}}\,dy\\
        &=\frac{1}{2}\lim_{\varepsilon\to 0} \int_{\abs{y}>\varepsilon}\frac{u(x)-u(x+y)}{\abs{y}^{n+2s}}\,dy+\frac{1}{2}\lim_{\varepsilon\to 0}\int_{\abs{y}>\varepsilon}\frac{u(x)-u(x-y)}{\abs{y}^{n+2s}}\,dy\\
        &=\lim_{\varepsilon\to 0}\int_{\abs{y}>\varepsilon}\frac{u(x)-u(y)}{\abs{x-y}^{n+2s}}\,dy.
    \end{split}
\end{equation*}
However, it is now necessary to take the principal value of the integral since we have eliminated the cancellation of the linear terms in the symmetric difference of order two, and $u(x)-u(y)$ in only $O(\abs{x-y})$. Thus, the smoothness of $u$ no longer guarantees the local integrability, unless we are in the regime $0<s<1/2$.
\end{proof}

Before proceeding we recall that $\Sn$ is the space $C^{\infty}(\Rn)$ endowed with the metric topology
\begin{equation*}
    d(f,g)=\sum_{p=0}^{\infty}2^{-p}\frac{\norma{f-g}_p}{1+\norma{f-g}_p},
\end{equation*}
generated by the countable family of norms
\begin{equation}
    \label{normaS}
    \norma{f}_p= \sup_{\abs{\alpha}\le p} \sup_{x\in\Rn}(1+\abs{x}^2)^{\frac{p}{2}}\abs{\partial^{\alpha}f(x)}, \quad p \in \mathbb{N}\cup\{0\}.
\end{equation}
Now, we can prove that $(-\Delta)^su$ suitably decays at infinity:
\begin{proposition}\label{p:decay}
Let $u\in \Sn$. Then, for every $x\in \Rn$ with $\abs{x}>1$, we have
\begin{equation*}
    \abs{(-\Delta)^su(x)}\le C_{u,n,s}\abs{x}^{-(x+2s)},
\end{equation*}
where with $\norma{x}_p$ as in \eqref{normaS}, we have let
\begin{equation*}
    C_{u,n,s}=C_{n,s}(\norma{u}_{n+2}+\norma{u}_n+\norma{u}_{L^1(\Rn)}).
\end{equation*}
\end{proposition}
\begin{proof}
To see this we write
\begin{equation*}
    \begin{split}
        (-\Delta)^su(x)&=\frac{\gamma(n,s)}{2}\int_{\abs{y}<\frac{\abs{x}}{2}}\frac{2u(x)-u(x+y)-u(x-y)}{\abs{y}^{n+2s}}\,dy\\
        &+\frac{\gamma(n,s)}{2}\int_{\abs{y}\ge\frac{\abs{x}}{2}}\frac{2u(x)-u(x+y)-u(x-y)}{\abs{y}^{n+2s}}\,dy.
    \end{split}
\end{equation*}
Taylor's formula gives
\begin{equation*}
    2u(x)-u(x+y)-u(x-y)=-\frac{1}{2}\Braket{\nabla^2u(y^*)y,y}-\frac{1}{2}\Braket{\nabla^2u(y^{**})y,y},
\end{equation*}
where $y^*=x+yt^*$, $y^{**}=x+yt^{**}$, for $t^*,t^{**} \in [0,1]$. We now observe that on the set where $\abs{y}<\abs{x}/2$ we have by the triangle inequality
\begin{equation}
    \label{triang}
    \abs{x}<2\abs{y^*} \qquad \abs{x}<2\abs{y^{**}}.
\end{equation}
Using \eqref{triang} and the definition \eqref{normaS} of the norm $\abs{u}_{n+2}$ in $\Sn$, we find
\begin{equation*}
    \begin{split}
        &\abs*{\int_{\abs{y}<\frac{\abs{x}}{2}}\frac{2u(x)-u(x+y)-u(x-y)}{\abs{y}^{n+2s}}\,dy}\le\frac{1}{2}\int_{\abs{y}<\frac{\abs{x}}{2}}\frac{\abs{\nabla^2u(y^*)}+\abs{\nabla^2u(y^{**})}}{\abs{y}^{n+2s}}\abs{y}^2\,dy\\
        &\le C\norma{u}_{n+2}\left(\int_{\abs{y}<\frac{\abs{x}}{2}}\frac{\abs{y}^2}{(1+\abs{y^*}^2)^{\frac{n+2}{2}}\abs{y}^{n+2s}},dy + \int_{\abs{y}<\frac{\abs{x}}{2}}\frac{\abs{y}^2}{(1+\abs{y^{**}}^2)^{\frac{n+2}{2}}\abs{y}^{n+2s}},dy\right)\\
        &\le C\abs{x}^{-n-2}\norma{u}_{n+2}\int_{\abs{y}<\frac{\abs{x}}{2}}\frac{dy}{\abs{y}^{n+2s-2}}=C\abs{x}^{-n-2}\norma{u}_{n+2}\abs{x}^{2-2s}=C\norma{u}_{n+2}\abs{x}^{-(x+2s)},
    \end{split}
\end{equation*}
where $C=C_{n,s}>0$.

Next, we estimate
\begin{equation*}
    \begin{split}
        &\abs*{\int_{\abs{y}\ge\frac{\abs{x}}{2}}\frac{2u(x)-u(x+y)-u(x-y)}{\abs{y}^{n+2s}}\,dy}\le 2 \int_{\abs{y}\ge\frac{\abs{x}}{2}}\frac{\abs{u(x+y)-u(x)}}{\abs{y}^{n+2s}}\,dy\\
        &\le2 \int_{\abs{y}\ge\frac{\abs{x}}{2}}\frac{\abs{u(x+y)+u(x)}}{\abs{y}^{n+2s}}\,dy.
    \end{split}
\end{equation*}
We have
\begin{equation*}
    \begin{split}
        &\int_{\abs{y}>\frac{\abs{x}}{2}}\frac{\abs{u(x)}}{\abs{y}^{n+2s}}\le \sup_{x\in\Rn}((1+\abs{x}^2)^{\frac{n}{2}}\abs{u(x)})\int_{\abs{y}\ge\frac{\abs{x}}{2}}\frac{dy}{(1+\abs{x}^2)^{\frac{n}{2}}\abs{y}^{n+2s}}\\
        &\le\sup_{x\in\Rn}((1+\abs{x}^2)^{\frac{n}{2}}\abs{x}^{-n}\int_{\abs{y}\ge\frac{\abs{x}}{2}}\frac{dy}{\abs{y}^{n+2s}}\le \frac{C\norma{u}_n}{\abs{x}^{n+2s}},
    \end{split}
\end{equation*}
where $C=C_{n,s}>0$. Finally, we have trivially
\begin{equation*}
    \int_{\abs{y}\ge \frac{\abs{x}}{2}}\frac{\abs{u(x+y)}}{\abs{y}^{n+2s}}\,dy\le\frac{2^{n+2s}}{\abs{x}^{n+2s}}\int_{\abs{y}\ge\frac{\abs{x}}{2}}\abs{u(x+y)}\,dy\le \frac{2^{n+2s}\norma{u}_{L^1(\Rn)}}{\abs{x}^{n+2s}}.
\end{equation*}
This completes the proof.
\end{proof}

Proposition \ref{p:decay} has the following non trivial consequence.
\begin{corollary}\label{L1nat}
Let $u\in \Sn$.Then, $(-\Delta)^s u\in C^{\infty}(\Rn)\cap L^1(\Rn)$.
\end{corollary}
The estimate in Proposition \ref{p:decay} can be written
\begin{equation*}
    -C_{u,n,s}\abs{x}^{-(n+2s)}\le -(-\Delta)^su(x)\le C_{u,n,s}\abs{x}^{-(n+2s)}.
\end{equation*}
Let us notice that on a nonnegative bump function the estimate from below can be made stronger, a fact that reflects the nonlocal character of $(-\Delta)^s$. Suppose for instance that $u\in C_0^{\infty}(\Rn)$, with $0\le u\le 1$, $u\equiv 1$ on $B(0,1)$ and $\text{supp } u \subset \bar{B}(0,2)$. Then, for $x\in\Rn \setminus B(0,3)$ one has from \eqref{fraclap}
\begin{equation*}
    \begin{split}
        -(-\Delta)^su(x)&=\frac{\gamma(n,s)}{2}\int_{\Rn}\frac{u(x+y)+u(x-y)}{\abs{y}^{n+2s}}\,dy\\
        &\ge \gamma(n,s)\int_{B(0,1)}\frac{dz}{\abs{x-z}^{n+2s}}\,dz.
    \end{split}
\end{equation*}
Since $\abs{x-z}\ge2$, for $\abs{z} \le 1$ we infer $\abs{x}\ge\abs{x-z}-\abs{z}\ge \abs{x-z}-1 \ge \abs{x-z}/2$. This gives some $C(n,s)>0$
\begin{equation*}
    -(-\Delta)^su\ge C(n,s)\abs{x}^{-(n+2s)}>0,
\end{equation*}
which shows that $(-\Delta)^su$ needs not to vanish even far away from the support of $u$. This is clearly impossible for local operators $P(x,\partial_x)$, for which one has the obvious property supp $P(x,\partial_x)u \subset \text{supp }u$.

%Seconda sezione
\section{A brief interlude about very classical stuff}
\label{s:1.2}
To proceed with the analysis of the nonlocal operator $(-\Delta)^s$ we will need some basic properties of two important, and deeply interconnected, protagonist of classical analysis: the Fourier transform and Bessel functions. Since they both play a pervasive role in these thesis, as a help to the reader in this section we recall their definition along with some elementary facts. Before we do that, however, we introduce the ever present Euler's gamma function (see e.g. chapter 1 in \citep{Le72}):
\begin{equation*}
    \Gamma(x)=\int_0^{\infty}t^{x-1}e^{-t}\,dt \quad x>0.
\end{equation*}
The well-known identity $\Gamma(1/2)=\sqrt{\pi}$ is simply a reformulation of the famous integral
\begin{equation*}
    \int_{\mathbb{R}}e^{-x^2}\,dx=\sqrt{\pi}.
\end{equation*}
Of course, $\Gamma(z)$ can be equally defined as holomorphic function for every $z\in\mathbb{C}$ with $\mathfrak{R}z>0$. It easy to check that for such $z$, one has
\begin{equation}
    \label{fact}
    \Gamma(z+1)=z\Gamma(z).
\end{equation}
This formula, and its iterations, can be used to meromorphically extend $\Gamma(z)$ to the whole complex plane having simple poles at $z=-k$, $k\in \mathbb{N}\cup\{0\}$, with residues $(-1)^k$. In particular, when $0<s<1$, one obtains from \eqref{fact}
\begin{equation}
    \label{fact2}
    \Gamma(1-s)=-s\Gamma(-s).
\end{equation}
Furthermore, one has the following basic relations:
\begin{equation}
    \label{fact3}
    \Gamma(z)\Gamma(1-z)=\frac{\pi}{\sin \pi z},
\end{equation}
and
\begin{equation}
    \label{fact4}
    2^{2z-1}\Gamma(z)\Gamma \left( z + \frac{1}{2} \right) = \sqrt{\pi} \Gamma(2z).
\end{equation}
\emph{Stirling's formula} provides the asymptotic behavior of the gamma function for large positive values of its argument
\begin{equation}
    \label{Stirling}
    \Gamma(x)=\sqrt{2\pi}x^{x-\frac{1}{2}}e^{-x}\left( 1 + O\left( \frac{1}{x}\right) \right), \quad \text{as }x\to +\infty.
\end{equation}
We close this brief prelude with a very classical formula which connects the gamma functions to the $(n-1)$-dimensional Hausdorff measure of the unit sphere $\mathbb{S}^{n-1}\subset \Rn$, and the $n$-dimensional volume of the unit ball
\begin{equation}
    \label{ball}
    \sigma_{n-1}=\frac{2\pi^{\frac{n}{2}}}{\Gamma \left( \frac{n}{2} \right)}, \qquad \omega_n=\frac{\sigma_{n-1}}{n}=\frac{\pi^{\frac{n}{2}}}{\Gamma \left( \frac{n}{2}+1 \right)}.
\end{equation}
One identity that we will use is the following
\begin{equation}\label{gammafrom}
    \int_{0}^{\infty} u^{-s-1}(1-e^{-u})\,du=\frac{1}{s}\int_0^{\infty}u^{-s}e^{-u}\,du=\frac{\Gamma(1-s)}{s}.
\end{equation}
Deeply connected with the gamma function is Euler's \emph{beta function} which for $x,y>0$ is defined as follows
\begin{equation}\label{beta}
    B(x,y)=2\int_0^{\frac{\pi}{2}}(\cos \theta)^{2x-1}(\sin \theta)^{2y-1}\,d\theta.
\end{equation}
It is an easy exercise to recognize that
\begin{equation}
    \label{beta2}
    B(x,y)=2\int_0^1(1-\tau^2)^{x-1}\tau^{2y-1}\,d\tau=\int_0^1(1-s)^{x-1}s^{y-1}\,ds.
\end{equation}
The link between the beta and the gamma function is expressed by the following equation
\begin{equation}
    \label{betagamma}
    B(x,y)=\frac{\Gamma(x)\Gamma(y)}{\Gamma(x+y)},
\end{equation}
see e.g. (1.5.6) on p. 14 in \citep{Le72}. A useful integral which is expressed in terms of the beta, or gamma function is contained in the following proposition.
\begin{proposition}\label{p:gammabeta}
Let $b>-n$ and $a>n+b$, then
\begin{equation}
    \label{useprop}
    \int_{\Rn} \frac{\abs{x}^b}{(1+\abs{x}^2)^{\frac{a}{2}}}\,dx= \frac{\pi^{\frac{n}{2}}}{\Gamma \left( \frac{n}{2}\right)} \frac{\Gamma \left(\frac{b+n}{2} \right) \Gamma \left( \frac{a-b-n}{2}\right)}{\Gamma \left(\frac{a+b}{2} \right)}.
\end{equation}
In particular, if $b=0$ and $a=n+1$, then
\begin{equation}
    \label{useprop2}
    \int_{\Rn}\frac{dx}{(1+\abs{x}^2)^{\frac{n+1}{2}}}=\frac{\pi^{\frac{n+1}{2}}}{\Gamma \left(\frac{n+1}{2} \right)}.
\end{equation}
\end{proposition}
\begin{proof}
Let us observe preliminarily that the assumption $b>-n$ serves to guarantee that the integrand belongs to $L_{\text{loc}}^1(\Rn)$, whereas it is in $L^1(\Rn)$ if and only if $a-b>n$. Under these hypothesis we have
\begin{equation*}
    \begin{split}
        &\int_{\Rn}\frac{\abs{x}^b}{(1+\abs{x}^2)^{\frac{a}{2}}}\,dx=\sigma_{n-1}\int_0^{\infty}\frac{r^{b+n-1}}{(1+r^2)^{\frac{a}{2}}}\\
        &=\int_0^{\frac{\pi}{2}}\frac{(\tan \xi)^{b+n-1}}{(1+\tan^2 \xi)^{\frac{a-2}{2}}}\,d\xi=\sigma_{n-1}\int_0^{\frac{\pi}{2}}(\sin \xi)^{b+n-1}(\cos \xi)^{a-b-n-1}\,d\xi\\
        &=\frac{\sigma_{n-1}}{2}B\left(\frac{b+n}{2},\frac{a-b-n}{2} \right),
    \end{split}
\end{equation*}
If we now apply formulas \eqref{ball} and \eqref{betagamma} we obtain \eqref{useprop}. To obtain \eqref{useprop2} it suffices to keep in mind that $\Gamma(1/2)=\sqrt{\pi}$.
\end{proof}

We are ready to introduce the queen of classical analysis: given a function $u\in L^1(\Rn)$, we define its Fourier transform as
\begin{equation*}
    \mathscr{F}(u)(\xi)=\reallywidehat{u}(\xi)=\int_{\Rn}e^{-2\pi\Braket{\xi,x}}u(x)\,dx.
\end{equation*}
We notice that the normalization that we have adopted in the above definition is the one which makes $\mathscr{F}$ an isometry of $L^2(\Rn)$ onto itself, see \citep{SW71}. We recall next some of the basic properties of $\mathscr{F}$. If $\tau_h u(x)=u(x+h)$ and $\delta_{\lambda}u(x)=u(\lambda x)$ are the translation and dilation operators in $\Rn$, then we have
\begin{equation}
    \label{translation}
    \reallywidehat{\tau_yu}(\xi)=e^{2\pi i \Braket{\xi,y}}\reallywidehat{u}(\xi),
\end{equation}
and
\begin{equation}
    \label{dilation}
    \reallywidehat{\delta_{\lambda}u}(\xi)=\lambda^{-n}\reallywidehat{u}\left( \frac{\xi}{\lambda} \right).
\end{equation}
The Fourier transform is also invariant under the action of the orthogonal group $\mathbb{O}(n)$. We have in fact for every $T\in \mathbb{O}(n)$
\begin{equation}
    \label{fourierInvariant}
    \reallywidehat{u\circ T}=\reallywidehat{u}\circ T.
\end{equation}
Formula \eqref{fourierInvariant} says that the Fourier transform of a spherically symmetric function is spherically symmetric as well.

Another crucial property is the Riemann-Lebesgue lemma:
\begin{equation}
    \label{R-L}
    u\in L^1(\Rn) \Longrightarrow \abs{\reallywidehat{u}(\xi)}\to 0 \text{ as } \abs{\xi}\to \infty.
\end{equation}
This result has important consequences when combined with the following two formulas. Let $u\in L^1(\Rn)$ be such that for $\alpha \in \mathbb{N}_0^n$ also $\partial^{\alpha}u\in L^1(\Rn)$. Then,
\begin{equation}
    \label{fourierder}
    \reallywidehat{(\partial^{\alpha}u)}(\xi)=(2\pi i)^{\abs{\alpha}}\xi^{\alpha}\reallywidehat{u}(\xi).
\end{equation}
In particular, \eqref{R-L} and \eqref{fourierder} give: $\abs{\xi^{\alpha}}\abs{\reallywidehat{u}(\xi)}\to 0$ as $\abs{\xi}\to \infty$. Furthermore, if $u\in L^1(\Rn)$ is such that for $\alpha \in \mathbb{N}_0^n$ one has $x \to x^{\alpha}u(x)\in L^1(\Rn)$, then,
\begin{equation}
    \label{fourierder2}
    \partial^{\alpha}\reallywidehat{u}(\xi)=(-2\pi i)^{\abs{\alpha}}\reallywidehat{(\cdot)^{\alpha}u}(\xi).
\end{equation}
In particular, \eqref{R-L} and \eqref{fourierder2} imply that: $\partial^{\alpha}\reallywidehat{u} \in C(\Rn)$ and $\abs{\partial^{\alpha} \reallywidehat{u}(\xi)}\to 0$ as $\abs{\xi}\to \infty$.

Combining these observations one derives one of the central properties of $\mathscr{F}$: it maps continuously $\Sn$ onto itself and is an isomorphism. Its inverse is also continuous, and is given by Fourier inversion formula
\begin{equation*}
    \mathscr{F}^{-1}(u)(x)= \int_{\Rn}e^{2\pi\Braket{\xi,x}}\reallywidehat{u}(\xi)\,d\xi.
\end{equation*}
We next introduce the second main character of this section: the Bessel functions. The book \citep{Le72} provides a rewarding account of this beautiful classical subject.
\begin{definition}
    \label{d:besselfunct}
    For every $v \in \mathbb{C}$ such that $\mathfrak{R}v>-\frac{1}{2}$ we define the Bessel function of the first kind and of complex order $v$ by the formula
    \begin{equation}
        \label{firstkind}
        J_v(z)=\frac{1}{\Gamma \left( \frac{1}{2}\right) \Gamma \left(v+\frac{1}{2} \right)}\left(\frac{z}{2} \right)^v\int_{-1}^1e^{izt}(1-t^2)^{\frac{2v-1}{2}}\,dt,
    \end{equation}
    where $\Gamma(x)$ denotes the Euler gamma function.
\end{definition}
The function $J_v(z)$ in \eqref{firstkind} derives its name from the fact that it solves the linear ordinary differential equation known as \emph{Bessel equation of order} $v$
\begin{equation}
    \label{Besseleq}
    z^2\frac{d^2J}{dz^2}+z\frac{dJ}{dz}+(z^2-v^2)J=0.
\end{equation}
An expression of $J_v$ as a power series for an arbitrary value of $v\in\mathbb{C}$ is provided by
\begin{equation}
    \label{besselsol}
    J_v(z)=\sum_{k=0}^{\infty}(-1)^k \frac{(z/2)^{v+2k}}{\Gamma(k+1)\Gamma(k+v+1)}, \quad \abs{z}<\infty,\,\, \abs{\arg(z)}<\pi,
\end{equation}
see e.g. (5.3.2) on p. 102 in \citep{Le72}. When $v\not \in \mathbb{Z}$, another linearly independent solution of \eqref{Besseleq} is provided by the function $J_{-v}(z)$. When $z\in\mathbb{Z}$ the two functions $J_v$ and $J_{-v}$ are linearly dependent, and in order to find a second solution linearly independent from $J_v$ one has to proceed differently.

The observation that follows is very important in most concrete applications of the theory. Suppose that $\varphi(z)$ be a solution to the Bessel equation \eqref{Besseleq}, and consider the function defined by the transformation
\begin{equation}
    \label{transformation}
    u(y)=y^{\alpha}\varphi(\beta y^{\gamma}).
\end{equation}
Then, one easily verifies that $u(y)$ satisfies the \emph{generalized Bessel equation}
\begin{equation}
    \label{genBesseleq}
    y^2u''(y) + (1-2\alpha)yu'(y)+ [\beta^2\gamma^2y^{2\gamma}+(\alpha^2-v^2\gamma^2)]u(y)=0.
\end{equation}

Returning to Definition \ref{d:besselfunct}, from \eqref{firstkind} and \eqref{beta2} we immediately find
\begin{equation*}
    z^{-v}J_v(z)  \xrightarrow[z \to 0]{}  \frac{2^{-v+1}}{\Gamma \left( \frac{1}{2} \right) \Gamma \left( v+\frac{1}{2} \right)}\int_0^1(1-s^2)^{\frac{2v-1}{2}}\,ds= \frac{2^{-v+1}}{\Gamma \left( \frac{1}{2} \right) \Gamma \left( v+\frac{1}{2} \right)} B\left( v+\frac{1}{2},\frac{1}{2}\right).
\end{equation*}
From this asymptotic relation and \eqref{betagamma} one obtains
\begin{equation}
    \label{asymp}
    J_v(z) \cong \frac{2^{-v}}{\Gamma(v+1)}z^v, \quad \text{ as } z \to 0.
\end{equation}
Unlike the simple expression of the asymptotic of $J_v(z)$ as $z\to0$, the behavior at infinity of $J_v(z)$ is more delicate to come by. We have the following result, see (5.11.6) on p. 122 in \citep{Le72}.
\begin{lemma}\label{l:asymp}
Let $\mathfrak{R}v>-\frac{1}{2}$. One has
\begin{equation}
    \label{asymp1}
    \begin{split}
    J_v(z)=&\sqrt{\frac{2}{\pi z}}\cos \left( z-\frac{\pi v}{2} - \frac{\pi}{4} \right) +O(z^{-\frac{3}{2}})\\
    &\text{as } \abs{z}\to \infty, \quad -\pi+\delta< \arg z <\pi -\delta. 
    \end{split}
\end{equation}
In particular,
\begin{equation}
    \label{asymp2}
    J_v(z)=O(z^{-\frac{1}{2}}), \quad \text{as } z\to \infty, \,\, z\ge 0.
\end{equation}
\end{lemma}
Along with the Bessel equation \eqref{Besseleq}, in Sections \ref{s:fs} and \ref{s:expb} below we will need the \emph{modified Bessel equation} of order $v\in\mathbb{C}$,
\begin{equation}
    \label{besseleq2}
    z^2\frac{d^2\varphi}{dz^2}+z\frac{d \varphi}{d z}-(z^2+v^2)\varphi=0.
\end{equation}
Two linearly independent solutions of \eqref{besseleq2} are the modified Bessel function of the first kind,
\begin{equation}
    \label{solbesseleq2}
    I_v(z)=\sum_{k=0}^{\infty}\frac{(z/2)^{v+2k}}{\Gamma(k+1)\Gamma(k+v+1)}, \quad \abs{z}<\infty, \,\, \abs{\arg (z)}<\pi,
\end{equation}
and the modified Bessel function of the third kind, or Macdonald function, which for order $v \not = 0,\pm1,\pm2,\dots$, is given by
\begin{equation}
    \label{thirdkind}
    K_v(z)=\frac{\pi}{2}\frac{I_{-v}(z)-I_v(z)}{\sin \pi v}, \quad \abs{\arg (z)}<\pi.
\end{equation}
Notice that $K_v(z)=K_{-v}(z)$.

It easy to verify that if $\varphi(z)$ is a solution to the modified Bessel equation \eqref{besseleq2}, then the function defined by the transformation \eqref{transformation} satisfies the \emph{generalized modified Bessel equation}
\begin{equation}
    \label{gmb}
    y^2u''(y)+(1-2\alpha)yu'(y)+[(\alpha^2-v^2\gamma^2)-\beta^2\gamma^2y^{2\gamma}]u(y)=0.
\end{equation}

As we have stated in the opening of this section the Fourier transform and the Bessel functions are deeply connected. One important instance of this link is the following result which provides a deeper meaning to the invariance of the Fourier transform with respect to the action of the orthogonal group $\mathbb{O}(n)$. We emphasize that the presence of Bessel functions in Theorem \ref{t:FBr} below underscores the interplay between curvature (that of the unit sphere $\mathbb{S}^{n-1}\subset \mathbb{R}^n$) and Fourier analysis. For the following result we refer to Theorem 40 on p. 69 in \citep{BC49}.
\begin{theorem}[Fourier-Bessel representation]
\label{t:FBr}
Let $u(x)=f(\abs{x})$, and suppose that
\begin{equation*}
    t \to t^{\frac{n}{2}}f(t)J_{\frac{n}{2}-1}(t) \in L^1(\mathbb{R}^+),
\end{equation*}
where we have denoted by $J_{\frac{n}{2}-1}$ the Bessel function of order $v=\frac{n}{2}-1$ defined by \eqref{firstkind}. Then,
\begin{equation*}
    \reallywidehat{u}(\xi)=2\pi \abs{\xi}^{-\frac{n}{2}+1}\int_0^{\infty}t^{\frac{n}{2}}f(t)J_{\frac{n}{2}-1}(2\pi \abs{\xi}t)\,dt.
\end{equation*}
\end{theorem}
To check the integrability assumption in Theorem \ref{t:FBr} we can use the above given asymptotic \eqref{asymp} and \eqref{asymp2} for the Bessel function $J_v$.

Another family of special functions that will be needed in this paper are the so-called \emph{hypergeometric functions}. In order to introduce them we recall the definition of the Pochammer's symbols
\begin{equation*}
    \alpha_0=1, \quad \alpha_k:= \frac{\Gamma(\alpha +k)}{\Gamma(\alpha)}=\alpha(\alpha+1)\dots(\alpha +k-1), \quad k \in \mathbb{N}.
\end{equation*}
Notice that since, as we have said,the gamma function has a pole in $z=0$, we have
\begin{equation*}
    0_k=
    \begin{cases}
    1 \quad \text{ if } \quad k=0\\
    0 \quad \text{ for } \,\, k\ge 0.
    \end{cases}
\end{equation*}
\begin{definition}\label{d:hg}
    Let $p,q\in\mathbb{N}_0$ be such that $p \le q+1$, and let $\alpha_1,\dots,\alpha_p$ and $\beta_1,\dots,\beta_q$ be give parameters such that $-\beta_j \not \in \mathbb{N}_0$ for $j=1,\dots,q$. Given a nummber $z\in\mathbb{C}$, the power series
    \begin{equation*}
        {}_{p}F_q(\alpha_1,\dots,\alpha_p;\beta_1,\dots,\beta_q;z)=\sum_{k=0}^{\infty}\frac{(\alpha_1)_k\dots(\alpha_p)_k}{(\beta_1)_k\dots(\beta_q)_k}\frac{z^k}{k!}
    \end{equation*}
    is called the generalized hypergeometric function. When $p=2$ and $q=1$, then the function ${}_{2}F_1(\alpha_1,\alpha_2;\beta_1;z)$ is the Gauss' hypergeometric function, and it is usually denoted by $F(\alpha_1,\alpha_2;\beta_1;z)$.
\end{definition}
We have
\begin{equation}
    \label{hg1}
    F(\alpha,0;\beta;z)=F(0,\alpha;\beta;z)=1,
\end{equation}
and (see also p. 275 in \citep{Le72})
\begin{equation}
    \label{hg2}
    F(\alpha,\beta;\beta;-z)={}_{1}F_0(\alpha;-z)=(1+z)^{-\alpha}.
\end{equation}
It also interesting to observe that the hypergeometric function ${}_{0}F_1$ is in essence a Bessel function, up to powers and rescaling. One has in fact form \eqref{solbesseleq2} and Definition \ref{d:hg},
\begin{equation}
    \label{bhg}
    I_{v}(z)=\frac{1}{\Gamma(v+1)}\left( \frac{z}{2}\right)^v {}_{0}F_1(v+1;(z/2)^2).
\end{equation}

%terza sezione
\section{Fourier, Bessel and fractional Laplacean}
\label{s:1.3}
After our brief interlude on the Fourier transform and Bessel functions, we now return to the main protagonist of this chapter.
\begin{proposition}[Pseudodifferential nature of $(-\Delta)^s$]
\label{p:pdnfp}
Let $\gamma(n,s)>0$ be the number identified by the following formula
\begin{equation}
    \label{gammans}
    \gamma(n,s)\int_{\Rn}\frac{1-\cos (z_n)}{\abs{z}^{n+2s}}\,dz=1.
\end{equation}
Then, for any $u\in \Sn$ we have
\begin{equation}
    \label{ftlp}
    \reallywidehat{(-\Delta)^su}(\xi)=(2\pi\abs{xi})^{2s}\reallywidehat{u}(\xi).
\end{equation}
\end{proposition}
\begin{proof}
Let us observe that in view of Corollary \eqref{L1nat} we know that $(-\Delta)^su\in L^1(\Rn)$ and thus we can take its Fourier transform in the sense of $L^1$. Having said this, if denote by $\tau_hu(x)=u(x+h)$ the translation operator in $\Rn$, we can rewrite \eqref{fraclap} in the following way
\begin{equation}
    \label{fraclap3}
    (-\Delta)^su(x)=\frac{\gamma(n,s)}{2}\int_{\Rn}\frac{2u(x)-\tau_y(x)-\tau_{-y}(x)}{\abs{y}^{n+2s}}\,dy.
\end{equation}
Using \eqref{translation} we easily find
\begin{equation}
    \label{ftlapfrac2}
    \reallywidehat{(-\Delta)^su}(\xi)=\gamma(n,s)\left( \int_{\Rn}\frac{1-\cos(2\pi \Braket{\xi,y})}{\abs{y}^{n+2s}}\,dy \right)\reallywidehat{u}(\xi)=J(\xi)\reallywidehat{u}(\xi),
\end{equation}
where we have let
\begin{equation*}
    J(\xi)=\gamma(n,s)\int_{\Rn}\frac{1-\cos(2\pi \Braket{\xi,y})}{\abs{y}^{n+2s}}\,dy.
\end{equation*}
We notice that the integral defining $J(\xi)$ only depends on $\abs{\xi}$. For every $T\in\mathbb{O}(n)$ one in fact easily verifies that $J(T\xi)=J(\xi)$. For $\xi\not =0$ we can thus write 
\begin{equation*}
    J(\xi)=\gamma(n,s)\int_{\Rn}\frac{1-\cos\left( \Braket{\frac{\xi}{\abs{\xi}},2\pi\abs{\xi}y}\right)}{\abs{y}^{n+2s}}\,dy.
\end{equation*}
The change of variable $z=2\pi\abs{\xi}y$ now gives
\begin{equation}
    \label{Jcomp}
    \begin{split}
        J(\xi)&=(2\pi\abs{\xi})^{2s}\gamma(n,s)\int_{\Rn}\frac{1-\cos\left( \Braket{\frac{\xi}{\abs{\xi}},z}\right)}{\abs{z}^{n+2s}}\,dz\\
        &=(2\pi\abs{\xi})^{2s}\gamma(n,s)\int_{\Rn}\frac{1-\cos\left( \Braket{e_n,z}\right)}{\abs{z}^{n+2s}}\,dz=(2\pi\abs{\xi})^{2s}\gamma(n,s)\int_{\Rn}\frac{1-\cos z_n}{\abs{z}^{n+2s}}\,dz
    \end{split}
\end{equation}
Notice that the integrand in the right-hand side of the latter equation is nonnegative, and that the integral is convergent. We have in fact
\begin{equation*}
    \begin{split}
        \int_{\Rn}\frac{1-\cos z_n}{\abs{z}^{n+2s}}\,dz&=\int_{\abs{z}\le 1}\frac{1-\cos z_n}{\abs{z}^{n+2s}}\,dz+\int_{\abs{z}>1}\frac{1-\cos z_n}{\abs{z}^{n+2s}}\,dz\\
        &\le C \int_{\abs{z}\le 1}\frac{dz}{\abs{z}^{n-2(1-s)}} + 2 \int_{\abs{z}>1}\frac{dz}{\abs{z}^{n+2s}}<\infty. 
    \end{split}
\end{equation*}
Finally, if we substitute in \eqref{ftlapfrac2} the expression given by \eqref{Jcomp}, it becomes clear that if we choose $\gamma(n,s)>0$ as in \eqref{gammans}, then \eqref{ftlp} holds.
\end{proof}

Equation \eqref{ftlp} in Proposition \ref{p:pdnfp} has the following immediate consequence.
\begin{corollary}[Semigroup property]
\label{semigroup}
Let $0<s,s'<1$, with $s+s'\le 1$. Then, for any $u\in \Sn$ we have
\begin{equation*}
    (-\Delta)^{s+s'}u=(-\Delta)^s(-\Delta)^{s'}u=(-\Delta)^{s'}(-\Delta)^{s}u.
\end{equation*}
\begin{proof}
It is enough to verify the desired equality on the on the Fourier transform side. Using \eqref{ftlp} we find
\begin{equation*}
    \begin{split}
        \mathscr{F}\left( (-\Delta)^{s+s'}u \right)&=(2\pi\abs{x})^{2(s+s')}\reallywidehat{u}=(2\pi\abs{x})^{2s}(2\pi\abs{x})^{2s'}\reallywidehat{u}\\
        &=\mathscr{F}((-\Delta)^s(-\Delta)^{s'}u)=\mathscr{F}((-\Delta)^{s'}(-\Delta)^{s}u).
    \end{split}
\end{equation*}
\end{proof}

\end{corollary}
With Proposition \ref{p:pdnfp} in hands we can now prove the following important "integration by parts" formula.
\begin{lemma}
\label{l:Symmetry}
Let $0<s\le1$. Then, for any $u,v\in\Sn$ we have
\begin{equation}
    \label{symmetry}
    \int_{\Rn}u(x)(-\Delta)^sv(x)\,dx=\int_{\Rn}(-\Delta)^su(x)v(x)\,dx.
\end{equation}
\end{lemma}
\begin{proof}
The case $s=1$ is well-known, and it is just integration by parts, so let us focus on $0<s<1$. Since by Corollary \eqref{L1nat} we know $\reallywidehat{(-\Delta)^su},\reallywidehat{(-\Delta)^sv}\in L^{1}(\Rn)$, we can use the following formula, valid for any $f,g\in L^{1}(\Rn)$,
\begin{equation}
    \label{ibpft}
    \int_{\Rn}\reallywidehat{f}(\xi)g(\xi)\,d\xi=\int_{\Rn}f(\xi)\reallywidehat{g}(\xi)\,d\xi.
\end{equation}
Applying \eqref{ibpft} and \eqref{ftlp} in Proposition \ref{p:pdnfp}, we find
\begin{equation*}
    \begin{split}
        \int_{\Rn}&(-\Delta)^su(x)v(x)\,dx=\int_{\Rn}(-\Delta)^su(x)\mathscr{F}(\mathscr{F}^{-1}v)(x)\,dx= \int_{\Rn}\mathscr{F}((-\Delta)^su)(\xi)\mathscr{F}^{-1}v(\xi)\,d\xi\\
        &=\int_{\Rn}(2\pi\abs{\xi})^{2s}\reallywidehat{u}(\xi)\mathscr{F}^{-1}v(\xi)\,d\xi=\int_{\Rn}\reallywidehat{u}(\xi)(2\pi\abs{\xi})^{2s}\mathscr{F}^{-1}v(\xi)\,d\xi.
    \end{split}
\end{equation*}
Using \eqref{ftlp} again we have
\begin{equation}
    \label{symf3}
    \mathscr{F}^{-1}((-\Delta)^sv)(\xi)=(2\pi\abs{\xi})^{2s}\mathscr{F}^{-1}v(\xi).
\end{equation}
Inserting this information in the above equation, and applying \eqref{ftlp} again, we find
\begin{equation*}
    \begin{split}
        \int_{\Rn}&(-\Delta)^su(x)v(x)\,dx= \int_{\Rn}\reallywidehat{u}(\xi) \mathscr{F}^{-1}((-\Delta)^sv)(\xi) \, d\xi\\
        &=\int_{\Rn}\mathscr{F}^{-1}(\reallywidehat{u})(x)(-\Delta)^sv(x)\,dx= \int_{\Rn}u(x)(-\Delta)^sv(x)\,dx.
    \end{split}
\end{equation*}
\end{proof}

We next turn to computing explicitly the constant $\gamma(n,s)$ in \eqref{gammans}.
\begin{proposition}
\label{p:gammacomp}
Let $0<s<1$. Then, we have
\begin{equation}
    \label{gammacomp}
    \gamma(n,s)=\frac{s2^{2s}\Gamma \left( \frac{n+2s}{2} \right)}{\pi^{\frac{n}{2}}\Gamma(1-s)}.
\end{equation}
\end{proposition}
\begin{proof}
If we denote by $\theta\in [0,\pi]$ the angle that the vector $z\in \mathbb{R}^n \setminus \{ 0\}$ forms with the positive direction of the $z_n$-axis, then Cavalieri's principle, and Fubini's theorem, give
\begin{equation*}
    \begin{split}
        \int_{\Rn}\frac{1-\cos z_n}{\abs{z}^{n+2s}}\,dz&= \int_0^{\infty}\int_{\mathbb{S}^{n-1}}\frac{1-\cos (r\cos \theta)}{r^{n+2s}}d\,\sigma r^{n-1}\,dr\\
        &=\int_0^{\infty}\frac{1}{r^{1+2s}}\int_0^{\pi}[1-\cos(r\cos \theta)]\int_{L_{\theta}}\,d\sigma'\,d\theta\,dr,
    \end{split}
\end{equation*}
where we have indicated by $L_{\theta}= \{ y\in\mathbb{S}^{n-1}| \Braket{y,e_n}=\cos \theta \}$ the $(n-2)$-dimensional sphere in $\Rn$ with radius $\sin \theta$ obtained by intersecting $\mathbb{S}^{n-1}$ with the hyperplane $y_n=\cos \theta$. Since with $\sigma_{n-2}$ given by \eqref{ball} above we have
\begin{equation*}
    \int_{L_{\theta}}\,d\sigma'=\sigma_{n-2}(\sin \theta)^{n-2},
\end{equation*}
we obtain
\begin{equation}
    \label{calcgamma}
    \begin{split}
        \int_{\Rn}& \frac{1-\cos z_n}{\abs{z}^{n+2s}}\,dz=\sigma_{n-2}\int_{0}^{\infty}\frac{1}{r^{1+2s}}\int_0^{\pi}[1-\cos(r \cos \theta)](\sin \theta)^{n-2}\,d\theta \, dr\\
        &=\sigma_{n-2}\int_0^{\infty}\frac{1}{r^{1+2s}}\int_0^{\pi}[1-\cos(r \cos \theta)](1-\cos^2\theta)^{\frac{n-3}{2}}\sin \theta \, d\theta \, dr \quad (\text{set } \, u=\cos \theta)\\
        &=\sigma_{n-2}\int_0^{\infty}\frac{1}{r^{1+2s}}\int_{-1}^1 [1-\cos (ru)](1-u^2)^{\frac{n-3}{2}}\,du\,dr\\
        &=\sigma_{n-2}\int_0^{\infty}\frac{1}{r^{1+2s}}\left[ \int_{-1}^1(1-u^2)^{\frac{n-3}{2}}\,du - \int_{-1}^1 \cos (ru)(1-u^2)^{\frac{n-3}{2}}\,du \right]\,dr.
    \end{split}
\end{equation}
From \eqref{beta} and \eqref{beta2} we thus find
\begin{equation*}
    \int_{-1}^1(1-s^2)^{\frac{2v-1}{2}}\,ds=2\int_0^1(\cos \theta)^{2v}\,d\theta=B\left( v+\frac{1}{2},\frac{1}{2} \right)=\frac{\Gamma \left(v+\frac{1}{2} \right)\Gamma \left(\frac{1}{2} \right)}{\Gamma(v+1)}.
\end{equation*}
This gives
\begin{equation*}
    \int_{-1}^1(1-u^2)^{\frac{n-3}{2}}\,du=\frac{\Gamma \left(\frac{n-1}{2} \right)\Gamma \left(\frac{1}{2} \right)}{\Gamma \left( \frac{n}{2} \right)}.
\end{equation*}
On the other hand, we have
\begin{equation*}
    \int_{-1}^1 \cos (ru)(1-u^2)^{\frac{n-3}{2}}\,du=\int_{-1}^1 e^{iru}(1-u^2)^{\frac{n-3}{2}}\,du.
\end{equation*}
From this equation and \eqref{firstkind} in Definition \ref{d:besselfunct} we obtain with $v=\frac{n-2}{2}$ and $z=r$,
\begin{equation*}
    \int_{-1}^1 \cos(ru)(1-u^2)^{\frac{n-3}{2}}\,du=\Gamma \left( \frac{n-1}{2} \right) \Gamma\left(\frac{1}{2} \right)\left( \frac{2}{r} \right)^{\frac{n-2}{2}}J_{\frac{n-2}{2}}(r).
\end{equation*}
Substituting in \eqref{calcgamma} above, we find
\begin{equation*}
    \int_{\Rn}\frac{1-\cos z_n}{\abs{z}^{n+2s}}\,dz=\sigma_{n-2}\frac{\Gamma \left(\frac{n-1}{2} \right)\Gamma \left(\frac{1}{2} \right)}{\Gamma \left( \frac{n}{2} \right)} \int_0^{\infty}\frac{1}{r^{1+2s}}\left[ 1- \Gamma \left( \frac{n}{2} \right) \left( \frac{2}{r}\right)^{\frac{n-2}{2}}J_{\frac{n-2}{2}}(r) \right]\,dr.
\end{equation*}
Keeping \eqref{ball} in mind, which gives
\begin{equation*}
    \sigma_{n-2}=\frac{2\pi^{\frac{n-1}{2}}}{\Gamma \left( \frac{n-1}{2} \right)},
\end{equation*}
and $\sqrt{\pi}=\Gamma(1/2)$, we conclude that
\begin{equation*}
    \int_{\Rn}\frac{1-\cos z_n}{\abs{z}^{n+2s}}\,dz=\sigma_{n-1} \int_0^{\infty}\frac{1}{r^{1+2s}}\left[ 1- \Gamma \left( \frac{n}{2} \right) \left( \frac{2}{r}\right)^{\frac{n-2}{2}}J_{\frac{n-2}{2}}(r) \right]\,dr.
\end{equation*}
From this equation and \eqref{gammans} above, it is clear that the constant $\gamma(n,s)$ must be chosen so that
\begin{equation}
    \label{512}
    \gamma(n,s)\sigma_{n-1} \int_0^{\infty}\frac{1}{r^{1+2s}}\left[ 1- \Gamma \left( \frac{n}{2} \right) \left( \frac{2}{r}\right)^{\frac{n-2}{2}}J_{\frac{n-2}{2}}(r) \right]\,dr=1.
\end{equation}
In order to complete the proof, we are thus left with computing explicitly the integral in the right-hand side of \eqref{512}.

With $v=\frac{n}{2}-1$, consider now the function
\begin{equation*}
    \Psi_v(r)=1-\Gamma(v+1)\left( \frac{2}{r} \right)^v J_v(r).
\end{equation*}
From the series expansion of $J_v(r)$, see \eqref{besselsol} above, we have
\begin{equation*}
    J_v(r)=\frac{\left(\frac{r}{2} \right)^v}{\Gamma(v+1)}-\frac{\left(\frac{r}{2} \right)^{v+2}}{\Gamma(v+2)} +\frac{\left(\frac{r}{2} \right)^{v+4}}{\Gamma(v+3)}-\dots
\end{equation*}
This expansion gives for some function $h(r)=O(r^2)$ as $r\to 0$,
\begin{equation}
    \label{Psi0}
    \Psi_v(r)=(1+h(r))\left(\frac{r}{2} \right)^2.
\end{equation}
On the other hand, \eqref{asymp2} implies that as $r \to \infty$
\begin{equation}
    \label{Psiinf}
    \Psi_v(r)=1+O(r^{-(v+\frac{1}{2})}),
\end{equation}
and thus, in particular, $\Psi_v\in L^{\infty}[0,+\infty)$. We thus find
\begin{equation*}
    \begin{split}
    \int_0^{\infty}& \left[1-\Gamma \left( \frac{n}{2}\right)\left( \frac{2}{r}\right)^{\frac{n-2}{2}}J_{\frac{n-2}{2}}(r) \right]\,dr=\int_0^{\infty}\left( \frac{r^{-2s}}{-2s} \right)' \Psi_v(r)\,dr\\
    &=\lim_{R\to\infty}\lim_{\varepsilon\to0^+} \int_{\varepsilon}^{R}\left( \frac{r^{-2s}}{-2s} \right)' \Psi_v(r)\,dr\\
    &=-\lim_{R\to\infty}\frac{R^{-2s}}{2s}\Psi_v(R)+ \lim_{\varepsilon\to 0^+}\frac{\varepsilon^{-2s}}{2s}\Psi_v(\varepsilon) + \int_0^{\infty}\frac{r^{-2s}}{2s}\Psi_v'(r)\,dr.
    \end{split}
\end{equation*}
Since as we have observed $\Psi_v \in L^{\infty}[0,\infty)$, we clearly have
\begin{equation*}
    \lim_{R\to\infty}\frac{R^{-2s}}{2s}\Psi_v(R)=0.
\end{equation*}
From \eqref{Psi0} and the fact that $0<s<1$, we obtain
\begin{equation*}
    \lim_{\varepsilon\to0^+}\frac{\varepsilon^{-2s}}{2s}\Psi_v(\varepsilon)=0.
\end{equation*}
We thus infer that
\begin{equation*}
    \int_0^{\infty} \left[1-\Gamma \left( \frac{n}{2}\right)\left( \frac{2}{r}\right)^{\frac{n-2}{2}}J_{\frac{n-2}{2}}(r) \right]\,dr=\int_0^{\infty}\frac{r^{-2s}}{2s}\Psi_v'(r)\,dr.
\end{equation*}
On the other hand, the recursion formula for $J_v$, see e.g. (5.3.5) on p. 103 in \citep{Le72},
\begin{equation*}
    (z^{-v}J_v(z))'=-z^{-v}J_{v+1}(z),
\end{equation*}
gives
\begin{equation*}
    \Psi'(r)=-2^v\Gamma(v+1)(r^{-v}J_v(r))'=2^v\Gamma(v+1)r^{-v}J_{v+1}(r).
\end{equation*}
We thus find
\begin{equation*}
    \int_0^{\infty} \left[1-\Gamma \left( \frac{n}{2}\right)\left( \frac{2}{r}\right)^{\frac{n-2}{2}}J_{\frac{n-2}{2}}(r) \right]\,dr=\frac{2^v\Gamma(v+1)}{2s}\int_{0}^{\infty}\frac{1}{r^{\frac{n}{2}-1+2s}}J_{\frac{n}{2}}(r)\,dr.
\end{equation*}
Recalling that $v=\frac{n}{2}-1$ we can write the right-hand side as follows
\begin{equation*}
    \frac{2^v\Gamma(v+1)}{2s}\int_0^{\infty}\frac{1}{r^{v+2s}}J_{v+1}(r)\,dr=\frac{2^v\Gamma(v+1)}{2s}\int_0^{\infty}\frac{1}{r^{\mu -q}}J_{\mu}(r)\,dr,
\end{equation*}
where $\mu=v+1=\frac{n}{2}$, and $q=1-2s$. We now invoke the following result, which is formula (17) on p. 684 in \citep{GR80}:
\begin{equation}
    \label{GR80}
    \int_0^{\infty}\frac{1}{r^{\mu -q}}J_{\mu}(ar)\,dr=\frac{\Gamma\left( \frac{q+1}{2} \right)}{2^{\mu-q}a^{q-\mu+1}\Gamma \left( \mu-\frac{q}{2} + \frac{1}{2} \right)},
\end{equation}
provided that
\begin{equation*}
    -1<\mathfrak{R}q<\mathfrak{R}\mu - \frac{1}{2}.
\end{equation*}
With the above values of the parameters $\mu$ and $q$ this condition becomes
\begin{equation*}
    -1<1-2s<\frac{n}{2}-\frac{1}{2}.
\end{equation*}
Now, the former inequality is satisfied since it is equivalent to $s<1$, and the second is also also satisfied since it is equivalent to $s>\frac{1-n}{4}$, which is of course true since $s>0$, whereas $\frac{1-n}{4}\le0$ In conclusion, we obtain from \eqref{GR80}
\begin{equation*}
    \frac{2^v\Gamma(v+1)}{2s}\int_0^{\infty}\frac{1}{r^{v+2s}}J_{v+1}(r)\,dr=\frac{\Gamma \left( \frac{n}{2} \right)}{2s}\frac{\Gamma(1-s)}{2^{2s}\Gamma \left( \frac{n}{2} +s \right)}.
\end{equation*}
Returning to \eqref{512}, and keeping the first identity in \eqref{ball} in mind, we reach the conclusion that the constant $\gamma(n,s)$ is given by the equation
\begin{equation*}
    \gamma(n,s)\frac{2\pi^{\frac{n}{2}}}{\Gamma \left( \frac{n}{2} \right)} \frac{\Gamma \left( \frac{n}{2} \right)}{2s}\frac{\Gamma (1-s)}{2^{2s}\Gamma \left( \frac{n}{2}+s\right)}=1,
\end{equation*}
which finally gives
\begin{equation*}
    \gamma(n,s)=\frac{s2^{2s}\Gamma \left(\frac{n}{2}+s \right)}{\pi^{\frac{n}{2}}\Gamma (1-s)}.
\end{equation*}
This proves \eqref{gammacomp}, thus completing the proposition.
\end{proof}

%Terza sezione
\section{Fundamental solution}\label{s:fs}
In this section we compute the fundamental solution of the fractional Laplacean operator. 

Before we turn to the proof of the main results we pause for a moment to recall that there exist spaces larger than $\Sn$, or $L^{\infty}(\mathbb{R}^n) \cap C^2(\Rn))$, in which it is still possible to define the nonlocal Laplacean either pointwise or as a tempered distribution. Following Definition 2.3 in \citep{Si07}, given $0<s<1$ we can also consider the linear space of the functions $u\in C^{\infty}(\Rn)$ such that for every multi-index $\alpha \in \mathbb{N}_0^n$
\begin{equation*}
    [u]_{\alpha}=\sup_{x\in \Rn}(1+\abs{x}^{n+2s})\abs{\partial^{\alpha}u(x)}<\infty.
\end{equation*}
We denote by $\mathscr{L}_s(\Rn)$ the space $C^{\infty}(\Rn)$ endowed with the countable family of seminorms $[\cdot]_{\alpha}$, and by $\mathscr{S}_s'(\Rn)$ its topological dual. We clearly have the inclusions
\begin{equation}
    \label{81}
    C_0^{\infty}(\Rn) \hookrightarrow \Sn \hookrightarrow \mathscr{S}_s(\Rn)\hookrightarrow C^{\infty}(\Rn),
\end{equation}
with the dual inclusions give by
\begin{equation}
    \label{82}
    \mathscr{E}'(\Rn) \hookrightarrow \mathscr{S}_s'(\Rn) \hookrightarrow \mathscr{S}'(\Rn)\hookrightarrow \mathscr{D}'(\Rn),
\end{equation}
where we recall that $\mathscr{E}'(\Rn)$ indicates the space of distributions with compact support. The next lemma justifies the introduction of the space $\mathscr{S}_s(\Rn)$.
\begin{lemma}
\label{l:81}
Let $u\in \Sn$. Then, $(-\Delta)^s u \in \mathscr{S}_s(\Rn)$.
\end{lemma}
\begin{proof}
From Proposition \ref{p:decay} we know that
\begin{equation*}
    [(-\Delta)^su]_0=\sup_{x \in \Rn}(1+\abs{x}^{n+2s})\abs{(-\Delta)^su(x)}<\infty.
\end{equation*}
Suppose now that $\alpha\in\mathbb{N}_0^n$ and $\abs{\alpha}=1$. We can write $\alpha=e_k$, where $e_k$ indicate one the vectors of the standard basis of $\Rn$. Applying \eqref{ftlp} in Proposition \ref{p:pdnfp} and \eqref{fourierder2}, we have
\begin{equation*}
    \begin{split}
        \partial^{\alpha}(-\Delta)^s u(x) &= \partial_k \mathscr{F}^{-1}\reallywidehat{(-\Delta)^su}(x)=(-2\pi i )\mathscr{F}^{-1}\left(\xi_k \reallywidehat{(-\Delta)^s u} \right)(x)\\
        &=(-2\pi i)\mathscr{F}^{-1}(\xi_k (2\pi \abs{\xi})^{2s}\reallywidehat{u}(\xi)) \quad \text{by \eqref{fourierder}}\\
        &= \mathscr{F}^{-1}\left( (2\pi \abs{\xi})^{2s}\reallywidehat{\partial_ku}(\xi)\right) \quad \text{by \eqref{ftlp} again}\\
        &= \mathscr{F}^{-1}\mathscr{F}((-\Delta)^s\partial_ku)=(-\Delta)^s\partial_ku.
    \end{split}
\end{equation*}
Since $\partial_k u \in \Sn$, again by Proposition \ref{p:decay} we conclude that
\begin{equation*}
    [u]_{e_k}= \sup_{x\in\Rn}(1+\abs{x}^{n+2s})\abs{\partial_ku(x)}<\infty.
\end{equation*}
Proceeding by induction on $\abs{\alpha}$, for all $\alpha\in\mathbb{N}_0^n$, we reach the desired conclusion.
\end{proof}

With Lemma \ref{l:81} in hands we can now extend the notion of solution to distributional ones.
\begin{definition}
    \label{d:82}
    Let $T\in \mathscr{S}'(\Rn)$. We say that a distribution $u\in\mathscr{S}_s'(\Rn)$ solves $(-\Delta)^su=T$ if for every test function $\varphi \in \Sn$ one has 
    \begin{equation*}
        \Braket{u,(-\Delta)^s \varphi}=\Braket{T,\varphi}.
    \end{equation*}
\end{definition}
In the special case in which $T=\delta$, the Dirac delta, then Definition \ref{d:82} leads to the following.
\begin{definition}[Fundamental solution]
    \label{d:83}
    We say that a distribution $E_s\in\mathscr{S}_s'(\Rn)$ is a fundamental solution of $(-\Delta)^s$ if $(-\Delta)^sE_s=\delta$. This means that for every $\varphi\in\Sn$ one has 
    \begin{equation*}
        \Braket{E_s,(-\Delta)^s\varphi}=\varphi(0).
    \end{equation*}
\end{definition}
It is clear from Definition \ref{d:83} that if $E_s\in\mathscr{S}_s'(\Rn)$ is a fundamental solution of $(-\Delta)^s$, then one has $(-\Delta)^sE_s=0$ in $\mathscr{D}'(\Rn \setminus \{0\})$. The following result establishes the existence of an explicit fundamental solution $E_s\in C^{\infty}(\Rn \setminus \{0\})$ of $(-\Delta)^s$.
\begin{theorem}
\label{t:84}
Let $n\ge2$ and $0<s<1$. Denote by
\begin{equation}
    \label{83}
    E_s(x)=\alpha(n,s)\abs{x}^{-(n-2s)},
\end{equation}
where the normalizing constant in \eqref{83} is given by
\begin{equation}
    \label{84}
    \alpha(n,s)=\frac{\Gamma \left(\frac{n}{2}-s \right)}{2^{2s}\pi^{\frac{n}{2}}\Gamma(s)}.
\end{equation}
Then, $E_s$ is a fundamental solution of $(-\Delta)^s$.
\end{theorem}
The proof of Theorem \ref{t:84} will be given after Lemma \ref{l:86} below.
\begin{lemma}
\label{l:85}
Suppose that either $n\ge2$, or $n=1$ and $0<s<1/2$. For every $y>0$ consider the regularized fundamental solution
\begin{equation}\label{86}
    E_{s,y}(x)=\alpha(n,s)(y^2+\abs{x}^2)^{-\frac{n-2s}{2}}.
\end{equation}
Then,
\begin{equation}
    \label{87}
    \reallywidehat{E_{s,y}}(\xi)=\frac{y^2}{2^{2s-1}\pi^s\Gamma(s)}\abs{\xi}^{-s}K_s(2\pi y\abs{\xi}),
\end{equation}
where we have denoted by $K_v$ the modified Bessel function of the third kind, see \eqref{thirdkind} above. From \eqref{87} we obtain for every $\xi\not=0$
\begin{equation}
    \label{88}
    \reallywidehat{E_s}(\xi)= \lim_{y\to 0^+}\reallywidehat{E_{s,y}}(\xi)=(2\pi\abs{xi})^{-2s}.
\end{equation}
\end{lemma}
\begin{proof}
To prove \eqref{87} it suffices to show that for every $f\in\Sn$ we have
\begin{equation}
    \label{89}
    \Braket{\reallywidehat{E_{s,y}},f}=\frac{y^s}{2^{2s-1}\pi^s\Gamma(s)}\int_{\Rn}\abs{\xi}^{-s}K_s(2\pi y \abs{\xi})f(\xi)\,d\xi.
\end{equation}
To establish \eqref{89} we use the heat semigroup and Bochner's subordination. The idea is to start from the observation that for every $L>0$ and $\alpha>0$ one has
\begin{equation}
    \label{810}
    \int_0^{\infty}e^{-tL}t^{\alpha}\frac{dt}{t}=\frac{\Gamma(\alpha)}{L^{\alpha}}.
\end{equation}
Using Fubini and \eqref{810} with $L=\abs{\xi}^2+y^2$, we obtain for any $\alpha>0$
\begin{equation*}
    \begin{split}
        \int_0^{\infty}& t^{\alpha} \left( \int_{\Rn} e^{-t(\abs{\xi}^2+y^2)}\reallywidehat{f}(\xi)\,d\xi \right)\frac{dt}{t}\\
        &= \int_{\Rn}\reallywidehat{f}(\xi)\left( \int_0^{\infty}t^{\alpha}e^{-t(\abs{\xi}^2+y^2)}\frac{dt}{t} \right)\,d\xi\\
        &=\Gamma(\alpha)\int_{\Rn} \reallywidehat{f}(\xi)(\abs{\xi}^2+y^2)^{-\alpha}\,d\xi.
    \end{split}
\end{equation*}
The above assumptions $n\ge2$, or $n=1$ and $0<s<1/2$, imply that $\alpha=\frac{n}{2}-s>0$. If we thus let $\alpha=\frac{n}{2}-s$ in the latter formula we find
\begin{equation}
    \label{811}
    \int_0^{\infty}t^{\frac{n}{2}-s}\left( \int_{\Rn} e^{-t(\abs{\xi}^2+y^2)}\reallywidehat{f}(\xi)\,d\xi \right)\frac{dt}{t}=\Gamma \left( \frac{n-2s}{2} \right) \int_{\Rn} \reallywidehat{f}(\xi)(\abs{\xi}^2+y^2)^{-(\frac{n-2s}{2})}\,d\xi.
\end{equation}
On the other hand, \eqref{ibpft} above gives for any $f\in\Sn$ and $y>0$
\begin{equation*}
    \int_{\Rn} \mathscr{F}_{x\to\xi}\left( e^{-t(\abs{x}^s+y^2)}\right)\reallywidehat{f}(\xi)\,d\xi=\int_{\Rn} \reallywidehat{f}(\xi)e^{-t(\abs{\xi}^2+y^2)}\,d\xi.
\end{equation*}
Multiplying both sides of this equation by $t^{\frac{n}{2}-s}$ and integrating between $0$ and $\infty$ with respect to the dilation invariant measure $\frac{dt}{t}$ we obtain
\begin{equation*}
    \int_0^{\infty}t^{\frac{n}{2}-s}\int_{\Rn}\mathscr{F}_{x\to\xi}\left( e^{-t(\abs{x}^2+y^2)} \right)f(\xi)\,d\xi\,\frac{dt}{t}=\int_0^{\infty}t^{\frac{n}{2}-s}e^{-y^2t}\int_{\Rn}\reallywidehat{e^{-t\abs{\cdot}^2}}(\xi)f(\xi)\,d\xi\,\frac{dt}{t}.
\end{equation*}
We next recall the following notable Fourier transform in $\Rn$: for every $t>0$, and every $\xi\in\mathbb{R}^n$, one has
\begin{equation}
    \label{812}
    \reallywidehat{e^{-t\abs{\cdot}^2}}(\xi)=\frac{\pi^{\frac{n}{2}}}{t^{\frac{n}{2}}}\exp \left( -\pi^2 \frac{\abs{\xi}^2}{t} \right).
\end{equation}
Substituting \eqref{812} in the preceding formula, we find
\begin{equation*}
    \begin{split}
        \int_0^{\infty}& t^{\frac{n}{2}-s}\int_{\Rn} \mathscr{F}_{x\to \xi} \left( e^{-t(\abs{x}^2+y^2)} \right) f(\xi)\,d\xi\,\frac{dt}{t}\\
        &=\pi^{\frac{n}{2}}\int_0^{\infty}t^{-s}e^{-y^2t}\int_{\Rn} \exp \left( -\pi^2 \frac{\abs{\xi}^2}{t} \right)f(\xi)\,d\xi\,\frac{dt}{t}\\
        &=\pi^{\frac{n}{2}}\int_{\Rn}f(\xi)\left( \int_0^{\infty} t^{-s}e^{-y^2t}\exp \left( -\pi^2 \frac{\abs{\xi}^2}{t} \right)\frac{dt}{t} \right)\,d\xi.
    \end{split}
\end{equation*}
We now use the following formula that can be found in 9. on p. 340 of \citep{GR80}
\begin{equation}
    \label{813}
    \int_0^{\infty} t^{v-1}e^{-(\frac{\beta}{t}+\gamma t)}\,dt=2\left( \frac{\beta}{\gamma} \right)^{\frac{v}{2}}K_v(2\sqrt{\beta \gamma}),
\end{equation}
provided $\mathfrak{R}\beta,\mathfrak{R}\gamma>0$. Applying \eqref{813} with
\begin{equation*}
    v=-s, \quad \beta=\pi^2 \abs{\xi}^2, \quad \gamma=y^2,
\end{equation*}
and keeping in mind that, as we have already observed, $K_v=K_{-v}$ (see 5.7.10 in \citep{Le72}), we find
\begin{equation}
    \label{814}
    \int_0^{\infty}t^{-s}e^{-y^2t} \exp \left( -\pi^2 \frac{\abs{\xi}^2}{t} \right)\frac{dt}{t}=2\left( \frac{y}{\pi \abs{\xi}} \right)^s K_s(2\pi y \abs{\xi}).
\end{equation}
Substituting \eqref{814} in the above integral, we conclude 
\begin{equation}
    \label{815}
    \int_0^{\infty} t^{\frac{n}{2}-s}\int_{\Rn}\reallywidehat{e^{-t(\abs{\cdot}^2+y^2)}}(\xi)f(\xi)\,d\xi\, \frac{dt}{t}=2 \pi^{\frac{n}{2}-s}y^s \int_{\Rn}\abs{\xi}^{-s}K_s(2\pi y \abs{\xi})f(\xi)\,d\xi.
\end{equation}
Since the integral in the left-hand side of \eqref{815} equals that in the left-hand side of \eqref{811}, we finally have
\begin{equation}
    \label{816}
    \alpha(n,s) \int_{\Rn} \reallywidehat{f}(\xi)(\abs{\xi}^2+y^2)^{-\left( \frac{n-2s}{2} \right)}\,d\xi=
    \alpha(n,s)\frac{2\pi^{\frac{n}{2}-s}y^s}{\Gamma \left( \frac{n-2s}{2} \right)}\int_{\Rn}\abs{\xi}^{-s} K_s(2\pi y \abs{\xi}) f(\xi)\,d\xi.
\end{equation}
Recalling \eqref{84}, which gives $\alpha(n,s)=\frac{\Gamma \left(\frac{n}{2}-s \right)}{2^{2s}\pi^{\frac{n}{2}}\Gamma(s)}$, we infer from \eqref{816} that
\begin{equation}
    \label{817}
    \alpha(n,s) \int_{\Rn} \reallywidehat{f}(\xi)(\abs{\xi}^2+y^2)^{-\left( \frac{n-2s}{2} \right)}\,d\xi=
    \frac{y^s}{2^{2s-1}\pi^s\Gamma(s)}\int_{\Rn}\abs{\xi}^{-s} K_s(2\pi y \abs{\xi}) f(\xi)\,d\xi.
\end{equation}
Keeping \eqref{86} in mind, we can rewrite \eqref{817} as follows
\begin{equation*}
    \Braket{E_{s,y},\reallywidehat{f}}=\frac{y^s}{2^{2s-1}\pi^s\Gamma(s)}\int_{\Rn}\abs{\xi}^{-s} K_s(2\pi y \abs{\xi}) f(\xi)\,d\xi.
\end{equation*}
Since by definition $\Braket{\reallywidehat{E_{s,y}},f}=\Braket{E_{s,y},\reallywidehat{f}}$, we conclude that \eqref{89} holds, thus completing the proof.
\end{proof}

We next prove a remarkable result concerning the function $E_{s,y}$ defined by \eqref{83} and \eqref{84} above.
\begin{lemma}
\label{l:86}
For every $y>0$ the function $E_{s,y}$ satisfies the equation
\begin{equation}
    \label{818}
    (-\Delta)^sE_{s,y}(x)=y^{2s}\frac{\Gamma \left( \frac{n}{2}+s \right)}{\pi^{\frac{n}{2}} \Gamma(s)}(y^2+\abs{x}^2)^{-(\frac{n}{2}+s)}.
\end{equation}
\end{lemma}
\begin{proof}
In order to establish \eqref{818} we begin by computing the function
\begin{equation*}
    F_{s,y}(x):=(-\Delta)^sE_{s,y}(x).
\end{equation*}
With this objective in mind we appeal to \eqref{ftlp}, which gives
\begin{equation}
    \label{819}
    \reallywidehat{F_{s,y}}(\xi)=\reallywidehat{(-\Delta)^sE_{s,y}}(\xi)=(2\pi\abs{\xi})^{2s}\reallywidehat{E_{s,y}}(\xi).
\end{equation}
We now use \eqref{87} in Lemma \ref{l:85}. Inserting such equation in \eqref{819} we obtain
\begin{equation}
    \label{820}
    \reallywidehat{F_{s,y}}(\xi)=(2 \pi \abs{\xi})^{2s}\frac{y^s}{2^{2s-1}\pi^s\Gamma(s)}\abs{\xi}^{-s}K_s(2\pi y \abs{\xi})=\frac{2y^s\pi^s}{\Gamma(s)}\abs{\xi}^sK_s(2\pi y\abs{\xi}).
\end{equation}
Using Theorem \ref{t:FBr} we find from \eqref{820}
\begin{equation}
    \label{821}
    F_{s,y}(x)=\frac{4y^s\pi^{s+1}}{\Gamma(s)}\frac{1}{\abs{x}^{\frac{n}{2}-1}}\int_0^{\infty}t^{\frac{n}{2}+s}K_s(2\pi y t)J_{\frac{n}{2}-1}(2\pi\abs{x}t)\,dt.
\end{equation}
If we now let
\begin{equation*}
    \lambda=-\frac{n}{2}-s, \quad \mu=s, \quad v=\frac{n}{2}-1,
\end{equation*}
then we can wirte the integral in the right-hand side of \eqref{821} in the form 
\begin{equation*}
    \int_0^{\infty} t^{-\lambda}K_{\mu}(at)J_v(bt)\,dt,
\end{equation*}
with
\begin{equation*}
    a=2\pi y, \quad b=2\pi \abs{x}.
\end{equation*}
Under the assumption $v-\lambda+1>\abs{\mu}$, that is presently equivalent to $n+s>s$, which is obviously true, we can appeal to formula 3. in 6.576 on p. 693 in \citep{GR80}. Such formula states that 
\begin{equation}
    \label{822}
    \begin{split}
        \int_0^{\infty}& t^{-\lambda}K_{\mu}(at)J_v(bt)\,dt=\\
        &=\frac{b^v\Gamma \left( \frac{v-\lambda +\mu +1}{2} \right) \Gamma \left( \frac{v-\lambda -\mu +1}{2} \right)}{2^{\lambda+1}a^{v-\lambda +1}\Gamma(1+v)}\, F\left( \frac{v-\lambda +\mu +1}{2},\frac{v-\lambda -\mu +1}{2}; v+1;-\frac{b^2}{a^2} \right),
    \end{split}
\end{equation}
where, we recall, $F(\alpha,\beta;\gamma;z)$ indicates the hypergeometric function ${}_{2}F_{1}(\alpha,\beta;\gamma;z)$, see Definition \ref{d:hg} above. Since
\begin{equation*}
    \frac{v-\lambda+\mu+1}{2}=\frac{n}{2}+s, \quad \frac{v-\lambda-\mu+1}{2}=\frac{n}{2},
\end{equation*}
from \eqref{821} and \eqref{822} we obtain
\begin{equation}
    \label{823}
    \begin{split}
        \int_0^{\infty}& t^{\frac{n}{2}+s}K_{s}(2\pi y t)J_{\frac{n}{2}-1}(2\pi \abs{x}t)\,dt=\\
        &=\frac{(2\pi\abs{x})^{\frac{n}{2}-1}\Gamma(\frac{n}{2}+s)}{2^{-\frac{n}{2}-s+1}(2\pi y)^{n+s}}\,F\left( \frac{n}{2}+s,\frac{n}{2};\frac{n}{2};-\frac{\abs{x}^2}{y^2} \right).
    \end{split}
\end{equation}
We now apply \eqref{hg2} to find
\begin{equation*}
    F\left( \frac{n}{2}+s,\frac{n}{2};\frac{n}{2};-\frac{\abs{x}^2}{y^2} \right)= \left( 1+\frac{\abs{x}^2}{y^2} \right)^{-\left( \frac{n}{2}+s \right)}.
\end{equation*}
Inserting this information into \eqref{823} we have
\begin{equation}
    \label{824}
    \int_0^{\infty} t^{\frac{n}{2}+s}K_{s}(2\pi y t)J_{\frac{n}{2}-1}(2\pi \abs{x}t)\,dt=\frac{(2\pi\abs{x})^{\frac{n}{2}-1}\Gamma(\frac{n}{2}+s)}{2^{-\frac{n}{2}-s+1}(2\pi y)^{n+s}}\left( 1+\frac{\abs{x}^2}{y^2} \right)^{-\left( \frac{n}{2}+s \right)}.
\end{equation}
From \eqref{821} and \eqref{824} we finally conclude
\begin{equation*}
    F_{s,y}(x)=\frac{\Gamma \left( \frac{n}{2}+s \right)}{y^n\pi^{\frac{n}{2}}\Gamma(s)}\left( 1+\frac{\abs{x}^2}{y^2} \right)^{-\left( \frac{n}{2}+s \right)}= \frac{y^{2s}\Gamma\left( \frac{n}{2}+s \right)}{\pi^{\frac{n}{2}}\Gamma(s)}(y^2+\abs{x}^2)^{-\left( \frac{n}{2}+s \right)}.
\end{equation*}
This establishes \eqref{818}, thus completing the proof.
\end{proof}

We are now ready to provide the 
\begin{proof}[Proof of Theorem \ref{t:84}]
Our objective is establishing
\begin{equation}
    \label{825}
    \int_{\Rn} E_s(x)(-\Delta)^s \phi(x)\,dx=\phi(0),
\end{equation}
for every test function $\phi\in \Sn$. We begin by observing that, since we are assuming that $n\ge2$, we automatically have that $0<s<\frac{n}{2}$. For $y>0$ we now consider the regularization $E_{s,y}$ of the distribution $E_s$ defined by \eqref{83} and \eqref{84} above. Notice that $E_{s,y}\in C^{\infty}(\Rn)$ and decays at $\infty$ like $\abs{x}^{-(n-2s)}$. Since for $\phi \in \Sn$ we know from Lemma \ref{l:81} that $(-\Delta)^s\phi\in \mathscr{S}_s(\Rn)$, it should be  clear that Lebesgue dominated convergence theorem gives
\begin{equation*}
    \int_{\Rn}E_{s,y}(x)(-\Delta)^s\phi(x)\,dx \longrightarrow \int_{\Rn}E_s(x)(-\Delta)^s\phi(x)\,dx
\end{equation*}
as $y\to 0^+$. On the other hand, Lemma \ref{l:Symmetry} (which continues to be valid in the present situation) gives
\begin{equation}
    \label{826}
    \int_{\Rn}E_{s,y}(x)(-\Delta)^s\phi(x)\,dx=\int_{\Rn}(-\Delta)^sE_{s,y}(x)\phi(x)\,dx.
\end{equation}
Therefore, in view of \eqref{826}, in order to complete the proof it will suffice to show that as $y\to0^+$
\begin{equation}
    \label{827}
    \int_{\Rn}(-\Delta)^sE_{s,y}(x)\phi(x)\,dx\longrightarrow \phi (0).
\end{equation}
To establish \eqref{827} we use \eqref{818} in Lemma \ref{l:86} which gives 
\begin{equation*}
    \begin{split}
        \int_{\Rn}& (-\Delta)^sE_{s,y}(x)\phi(x)\,dx=\frac{\Gamma \left( \frac{n}{2}+s \right)}{y^n\pi^{\frac{n}{2}}\Gamma(s)}\int_{\Rn}\left( 1+\frac{\abs{x}^2}{y^2} \right)^{-\left( \frac{n}{2}+s\right)}\phi(x)\,dx\\
        &= \frac{\Gamma \left( \frac{n}{2}+s \right)}{\pi^{\frac{n}{2}}\Gamma(s)}\int_{\Rn}(1+\abs{x'}^2)^{-\left( \frac{n}{2}+s \right)}\phi(yx')\,dx'\\
        &\longrightarrow \phi(0) \frac{\Gamma \left( \frac{n}{2}+s \right)}{\pi^{\frac{n}{2}}\Gamma(s)}\int_{\Rn}(1+\abs{x'}^2)^{-\left( \frac{n}{2}+s \right)}\,dx',
    \end{split}
\end{equation*}
where in the last equality we have used Lebesgue dominated convergence theorem. To complete the proof of \eqref{827} it would be sufficient to prove that
\begin{equation}
    \label{828}
    \frac{\Gamma \left( \frac{n}{2}+s \right)}{\pi^{\frac{n}{2}}\Gamma(s)}\int_{\Rn}(1+\abs{x'}^2)^{-\left( \frac{n}{2}+s \right)}\,dx'=1.
\end{equation} 
Now, the validity of \eqref{828} follows from a straightforward application of Proposition \ref{p:gammabeta} with the choice $a=n+2s$, $b=0$.
\end{proof}

%sezione 5
\section{Traces of Bessel processes}\label{s:expb}

When dealing with nonlocal operators such as $(-\Delta)^s$ a major difficulty is represented by the fact the they do not act on functions like differential operators do, but instead through nonlocal integral formulas such as \eqref{fraclap}. As a consequence, the rules of differentiation are not readily available. In this perspective it would be highly desirable to have some kind of procedure that allows to connect nonlocal problems to ones for which the rules of differential calculus are available. Exploring this connection is the principal objective of this section.

During the past decade there has been an explosion of interest in the analysis of nonlocal operators such as \eqref{fraclap} in connection with various problems from the applied sciences, analysis and geometry. The majority of these developments has been motivated by the remarkable 2007 "extension paper" \citep{CS07} by Caffarelli and Silvestre. In that paper the authors introduced a method that allows to convert nonlocal problems in $\Rn$ into ones that involve a certain (degenerate) differential operator in $\mathbb{R}_+^{n+1}$. Precisely, it was shown in \citep{CS07} that if for a given $0<s<1$ and $u\in\Sn$ one considers the function $U(x,y)$ that solves the following Dirichlet problem in the half-space $\mathbb{R}_+^{n+1}$:
\begin{equation}
    \label{101}
    \begin{cases}
    L_aU(x,y)=\dive_{x,y}(y^a\nabla_{x,y}U)=0 \quad x\in\Rn,y>0,\\
    U(x,0)=u(x),
    \end{cases}
\end{equation}
where now $a=1-2s$, then one can recover $(-\Delta)^su(x)$ by the following "trace" relation
\begin{equation}
    \label{102}
    -\frac{2^{2s-1}\Gamma(s)}{\Gamma(1-s)}\lim_{y\to0^+}y^{1-2s}\frac{\partial U}{\partial y}(x,y)=(-\Delta)^su(x).
\end{equation}
Thus, remarkably, \eqref{102} provides yet another way of characterizing $(-\Delta)^su(x)$ as thw weighted \emph{Dirichlet-to-Neumann} map of the extension problem \eqref{101}.

One key observation is that the second order degenerate elliptic equation in \eqref{101} can also be written in nondivergence form in the following way
\begin{equation}
    \label{105}
    \begin{cases}
    -\Delta_x U= \mathscr{B}_a U, \qquad \qquad \,\, (x,y)\in \mathbb{R}_+^{n+1}\\
    U(x,0)=u(x), \qquad \qquad \,\,\, x\in \Rn,\\
    U(x,y)\to 0, \text{ as } y\to \infty, \quad \, x\in \Rn,
    \end{cases}
\end{equation}
where we have denoted by
\begin{equation}
    \label{106}
    \mathscr{B}_a=\frac{\partial^2}{\partial y^2} + \frac{a}{y}\frac{\partial}{\partial y}
\end{equation}
the generator of the Bessel semigroup on $(\mathbb{R}^+,y^a\,dy)$.
\begin{theorem}
\label{t:101}
Let $u \in \Sn$. Then, the solution $U$ to the extension problem \eqref{101} is given by
\begin{equation}
    \label{107}
    U(x,y)=P_s(\cdot,y)\star u(x)=\int_{\Rn}P_s(x-z,y)u(z)\,dz,
\end{equation}
where
\begin{equation}
    \label{108}
    P_s(x,y)=\frac{\Gamma \left( \frac{n}{2}+s \right)}{\pi^{\frac{n}{2}}\Gamma(s)}\frac{y^{2s}}{(y^2+\abs{x}^2)^{\frac{n+2s}{2}}}
\end{equation}
is the Poisson kernel for the extension problem in the half-space $\Rnn_+$. For $U$ as in \eqref{107} one has
\begin{equation}
    \label{109}
    (-\Delta)^su(x)=-\frac{2^{2s-1}\Gamma(s)}{\Gamma(1-s)}\lim_{y\to0^+}y^a \frac{\partial U}{\partial y}(x,y).
\end{equation}
\end{theorem}
\begin{proof}
Consider the extension problem \eqref{101}, written in the form \eqref{105}. If we take a partial Fourier transform of the latter with respect to the variable $x\in\Rn$, we find
\begin{equation}
    \label{1010}
    \begin{cases}
    \frac{\partial^2\reallywidehat{U}}{\partial y^2}(\xi,y)+\frac{a}{y}\frac{\partial \reallywidehat{U}}{\partial y}(\xi,y)- 4\pi^2 \abs{\xi}^2\reallywidehat{U}(\xi,y)=0 \quad \text{in }\Rnn_+,\\
    \reallywidehat{U}(\xi,0)=\reallywidehat{u}(\xi), \quad \reallywidehat{U}(\xi,y) \to 0, \textbf{ as }y \to \infty, \quad x \in \, \Rn,
    \end{cases}
\end{equation}
where we have denoted 
\begin{equation*}
    \reallywidehat{U}(\xi,y)=\int_{\Rn}e^{-2\pi i \Braket{\xi,x}}U(x,y)\,dx.
\end{equation*}
In order to solve \eqref{1010} we fix $\xi \in \Rn \setminus \{ 0\}$, and with $Y(y)=Y_{\xi}(y)=\reallywidehat{U}(\xi,y)$, we write \eqref{1010} as 
\begin{equation}
    \label{1011}
    \begin{cases}
    y^2 Y''(y)+ayY'(y)-4\pi^2\abs{\xi}^2 y^2 Y(y) =0,\\
    Y(0)= \reallywidehat{u}(\xi,\sigma),\\
    Y(y)\to 0, \text{ as } y \to \infty.
    \end{cases}
\end{equation}
Comparing \eqref{1011} with the generalized modified Bessel equation in \eqref{gmb} above we see that the former fits into the general form of the latter provided that
\begin{equation*}
    \alpha=s, \quad \gamma=1, \quad v=s, \quad \beta=2\pi\abs{\xi}.
\end{equation*}
Thus, according to \eqref{gmb}, two linearly independent solutions of \eqref{1011} are given by
\begin{equation*}
    u_1(y)=y^s I_s(2 \pi \abs{\xi}y), \qquad u_2(y)=y^sK_s(2\pi\abs{\xi}y).
\end{equation*}
It ensues that, for every $\xi \not =0$, the general solution of \eqref{1010} is given by
\begin{equation*}
    \reallywidehat{U}(\xi,y)=Ay^sI_s(2\pi \abs{\xi}y) + By^sK_s(2\pi\abs{\xi}y).
\end{equation*}
The condition $\reallywidehat{U}(\xi,y)\to0$ as $y\to\infty$ forces $A=0$ (see e.g. formulas (5.11.9) and (5.11.10) on p. 123 of \citep{Le72} for the asymptotic behavior at $\infty$ of $K_s$ and $I_s$), and thus
\begin{equation}
\label{1012}
    \reallywidehat{U}(\xi,y)=By^sK_s(2\pi\abs{\xi}y).
\end{equation}
Next, we use the condition $\reallywidehat{U}(\xi,0)=\reallywidehat{u}(\xi)$ to fix the constant $B$. When $y\to0^+$ we have
\begin{equation*}
    \reallywidehat{U}(\xi,y)=By^sK_s(2\pi\abs{\xi}y)=B\frac{\pi y^s}{2}\frac{I_{-s}(2\pi \abs{\xi} y)-I_s(2\pi \abs{\xi}y)}{\sin \pi s }\to \frac{B\pi 2^{s-1}}{\Gamma(1-s)\sin \pi s}(2\pi \abs{\xi})^{-s},
\end{equation*}
Now from formula (5.7.1) on p. 108 of \citep{Le72}, we have as $z\to 0$
\begin{equation*}
    I_s(z)\cong \frac{1}{\Gamma(s+1)}\left( \frac{z}{2} \right)^s, \qquad I_{-s}(z)\cong \frac{1}{\Gamma(1-s)}\left( \frac{z}{2} \right)^{-s}.
\end{equation*}
Using this asymptotic, along with the formula \eqref{fact3} above, we find that as $y\to 0^+$,
\begin{equation*}
    \reallywidehat{U}(\xi,y)=By^sK_s(2\pi\abs{\xi}y)\to \frac{B\pi2^{s-1}}{\Gamma(1-s)\sin \pi s}(2\pi \abs{\xi})^{-s}=B2^{s-1}\Gamma(s)(2\pi \abs{\xi})^{-s}.
\end{equation*}
In order to fulfill the condition $\reallywidehat{U}(\xi,0)=\reallywidehat{u}(\xi)$ we impose that the right-hand side of the latter equation equal $\reallywidehat{u}(\xi)$. For this to happen we must have 
\begin{equation*}
    B=\frac{(2\pi\abs{\xi})^s\reallywidehat{u}(\xi)}{2^{s-1}\Gamma(s)}.
\end{equation*}
Substituting such value of $B$ in \eqref{1012}, we finally obtain
\begin{equation}
    \label{1013}
    \reallywidehat{U}(\xi,y)= \frac{(2\pi\abs{\xi})^s\reallywidehat{u}(\xi)}{2^{s-1}\Gamma(s)} y^s K_s(2\pi\abs{\xi}y).
\end{equation}
At this point we want to invert the Fourier transform in \eqref{1013}. In fact, it is clear from the latter equation that the function $U(x,y)$ will be given by \eqref{107}, with $P_s(x,y)$ as in\eqref{108}, if we can show that
\begin{equation}
    \label{1014}
    \mathscr{F}_{\xi \to x}^{-1} \left( \frac{(2\pi \abs{\xi})^s}{2^{s-1}\Gamma(s)}y^s K_s(2\pi\abs{\xi}y) \right)=\frac{\Gamma \left( \frac{n}{2}+s \right)}{\pi^{\frac{n}{2}}\Gamma(s)}\frac{y^{2s}}{(y^2+\abs{\xi}^2)^{\frac{n+2s}{2}}}.
\end{equation}
In view of Theorem \ref{t:FBr}, the latter identity is equivalent to
\begin{equation}
    \label{1015}
    \frac{2^2\pi^{s+1}y^s}{\abs{x}^{\frac{n}{2}-1}}\int_0^{\infty} t^{\frac{n}{2}+s}K_s(2\pi y t)J_{\frac{n}{2}-1}(2\pi \abs{x} t)\,dt= \frac{\Gamma \left( \frac{n}{2}+s \right)}{\pi^{\frac{n}{2}}\Gamma(s)}\frac{y^{2s}}{(y^2+\abs{\xi}^2)^{\frac{n+2s}{2}}}.
\end{equation}
We are thus left with proving \eqref{1015}. Remarkably, this identity has been already been established in \eqref{824} above. Therefore, \eqref{1015} does hold and, with it, \eqref{107} and \eqref{108} as well.

In order to complete the proof of the theorem we are thus left with establishing \eqref{109}. With this objective in mind we note that in view of \eqref{ftlp} in Proposition \ref{p:pdnfp}, proving \eqref{109} is equivalent to showing
\begin{equation}
    \label{1016}
    (2\pi\abs{\xi})^{2s}\reallywidehat{u}(\xi)=-\frac{2^{2s-1}\Gamma(s)}{\Gamma(1-s)}\lim_{y\to0^+}y^a\frac{\partial \reallywidehat{U}}{\partial y}(\xi,y).
\end{equation}
Keeping in mind that $a=1-2s$, and usind the formula
\begin{equation*}
    K_s'(z)=\frac{s}{z}K_s(z)-K_{s+1}(z)
\end{equation*}
(see (5.7.9) on p. 110 of \citep{Le72}), we obtain
\begin{equation*}
    y^a\frac{\partial \reallywidehat{U}}{\partial y}(\xi,y)=\frac{(2\pi \abs{\xi})^{s+1}\reallywidehat{u}(\xi)}{2^{s-1}\Gamma(s)}y^{1-s} \left[ \frac{2s}{(2\pi\abs{\xi})y}K_s(2\pi\abs{\xi}y)-K_{s+1}(2\pi\abs{\xi}y) \right].
\end{equation*}
Since
\begin{equation*}
    \frac{2s}{z}K_s(z)-K_{s+1}(z)=-K_{s-1}(z)=-K_{1-s}(z)
\end{equation*}
(again, by (5.7.9) on p. 110 of \citep{Le72}), we finally have
\begin{equation*}
    y^a\frac{\partial \reallywidehat{U}}{\partial y}(\xi,y)=-\frac{(2\pi \abs{\xi})^{s+1}\reallywidehat{u}(\xi)}{2^{s-1}\Gamma(s)}y^{1-s}K_{1-s}(2\pi \abs{\xi}y).
\end{equation*}
Now, as before, we have as $y\to0^+$,
\begin{equation*}
    y^{1-s}K_{1-s}(2\pi\abs{\xi}y) \longrightarrow 2^{-s}\Gamma(1-s)(2\pi\abs{\xi})^{s-1}.
\end{equation*}
We finally reach the conclusion that, as $y\to 0^+$,
\begin{equation*}
    y^a\frac{\partial \reallywidehat{U}}{\partial y}(\xi,y) \longrightarrow - \frac{\Gamma(1-s)}{2^{2s-1}\Gamma(s)}(2\pi\abs{\xi})^{2s}\reallywidehat{u}(\xi).
\end{equation*}
This proves \eqref{1016}, thus completing the proof.
\end{proof}

\begin{remark}\label{r:102}
Using Proposition \ref{p:gammabeta} with the choice $b=0$ and $a=n+2s$, it is easy to recognize from \eqref{108} that
\begin{equation}
    \label{1017}
    \norma{P_s(\cdot,y)}_{L^1(\Rn)}=\int_{\Rn}P_s(x,y)\,dx=1, \qquad \text{for every }y>0.
\end{equation}
\end{remark}
\begin{remark}\label{r:103}
Notice that when $s=1/2$ we have $a=1-2s=0$, and the extension operator $L_a$ becomes the standard Laplacean $L_a=\Delta_x+ \partial_y^2$ in $\Rnn$. From formula \eqref{108} we obtain in such case
\begin{equation*}
    P_{\frac{1}{2}}(x,y)=\frac{\Gamma \left( \frac{n+1}{2} \right)}{\pi^{\frac{n+1}{2}}}\frac{y}{(y^2+\abs{x}^2)^{\frac{n+1}{2}}},
\end{equation*}
which is in fact the standard Poisson kernel for the upper half-space $\Rnn_+$.
\end{remark}
\begin{remark}
\label{r:104}
If we compare the expression of the Poisson kernel in \eqref{108} with \eqref{818} in Lemma \ref{l:86}, we conclude that, remarkably, we have shown that
\begin{equation}
    \label{1018}
    P_s(x,y)=(-\Delta)^sE_{s,y}(x),
\end{equation}
where for $y>0$ the function $E_{s,y}=c(n,s)(y^2+\abs{x}^2)^{-\frac{n-2s}{2}}$ is the $y$-regularization of the fundamental solution of $(-\Delta)^s$. If we combine \eqref{1018} with \eqref{827} above, we see that we can reformulate \eqref{827} as follows
\begin{equation*}
    \lim_{y\to0^+}P_s(\cdot,y)=\delta \qquad \text{in }\, \mathscr{S}'(\Rn),
\end{equation*}
or, equivalently, for any $\phi \in \Sn$
\begin{equation*}
    \lim_{y\to0^+}\int_{\Rn}P_s(x,y)\phi(x)\,dx=\phi(0).
\end{equation*}
If we let $\tilde{\phi} (x)=\phi(-x)$, then we obtain from the latter limit relation
\begin{equation}
    \label{1019}
    P_s(\cdot,y)\star \phi(x)= \int_{\Rn}P_s(z,y)\tau_{-x}\tilde{\phi}(z)\,dz \longrightarrow \tau_{-x}\tilde{\phi}(0)=\phi(x).
\end{equation}
\end{remark}
\begin{remark}[Alternative proof of \eqref{109}]
\label{r:105}
Usiin the property \eqref{1019} of the Poisson kernel $P_s(x,y)$ we can provide another "short" proof of \eqref{109} along the following lines, see Section 3.1 in \citep{CS07}. Let $u\in \Sn$ and consider the solution $U(x,y)=P_s(\cdot,y)\star u(x)$ to the extension problem \eqref{101}, see \eqref{107}. Using \eqref{1017} we can write
\begin{equation*}
    U(x,y)=\frac{\Gamma \left( \frac{n}{2}+s \right)}{\pi^{\frac{n}{2}}\Gamma(s)}\int_{\Rn}\frac{y^{2s}}{(y^2+\abs{x-z}^2)^{\frac{n+2s}{2}}}\,dz+u(x).
\end{equation*}
Differentiating both sides of this formula with respect to $y$ and keeping in mind that $a=1-2s$, we obtain that as $y \to 0^+$
\begin{equation*}
    y^a\frac{\partial U}{\partial y}(x,y)= 2s \frac{\Gamma \left( \frac{n}{2}+s \right)}{\pi^{\frac{n}{2}}\Gamma(s)}\int_{\Rn}\frac{u(z)-u(x)}{(y^2+\abs{z-x}^2)^{\frac{n+2s}{2}}}\,dz + O(y^2).
\end{equation*}
Letting $y\to 0^+$ and using Lebesgue dominated convergence theorem, we thus find
\begin{equation*}
    \begin{split}
        \lim_{y\to 0^+} y^a\frac{\partial U}{\partial y}(x,y)&=2s \frac{\Gamma \left( \frac{n}{2}+s \right)}{\pi^{\frac{n}{2}}\Gamma(s)}\text{ PV }\int_{\Rn}\frac{u(z)-u(x)}{\abs{z-x}^{n+2s}}\,dz\\
        &=-2s \frac{\Gamma \left( \frac{n}{2}+s \right)}{\pi^{\frac{n}{2}}\Gamma(s)}\gamma(n,s)^{-1}(-\Delta)^su(x),
    \end{split}
\end{equation*}
where in the second equality we have used \eqref{fraclap2} above. If in the latter equation we now replace the expression \eqref{gammacomp} of the constant $\gamma(n,s)$, we reach the conclusion that \eqref{109} is valid.
\begin{flushright}
$\square$
\end{flushright}
\end{remark}
The Poisson kernel $P_s(x,y)$ is of course a solution of $L_aP_s=0$ in $\Rnn_+$. What is instead not obvious is that the $y$-regularization $E_{s,y}$ of the fundamental solution $E_s$ of $(-\Delta)^s$ introduced in \eqref{86} in Lemma \ref{l:85} is also a solution of the extension operator $L_a$. It was shown in \citep{CS07} that, up to a constant, such function is in fact the fundamental solution of $L_a$. The heuristic motivation behind this is that, with $x\in \Rn$, and $\eta \in \mathbb{R}^{a+1}$, if $y=\abs{\eta}$ then the operator
\begin{equation}
    \label{1020}
    y^{-a}L_a= \Delta_x + \frac{\partial^2}{\partial y^2}+\frac{a}{y}\frac{\partial}{\partial y}
\end{equation}
ca be thought of as the Laplacean in the fractional dimension $N=n+a+1$ acting on functions $U(x,\abs{\eta})$. Such heuristic is confirmed by the following result.
\begin{proposition}
\label{p:106}
For $y\in\mathbb{R}$ consider the function $G(x,y)=(\abs{x}^2+y^2)^{-\frac{n-2s}{2}}$, see \eqref{86}. Then, for every $(x,y)\in \Rnn_+$, with $a=1-2s$ we have
\begin{equation*}
    L_aG(x,y)=0.
\end{equation*}
\end{proposition}
\begin{proof}
It is convenient to use the expression of \eqref{1020} on functions depending on $r=\abs{x}$ and $y$
\begin{equation*}
    y^{-a}L_a=\frac{\partial^2}{\partial r^2} + \frac{n-1}{r}\frac{\partial}{\partial r} + \frac{\partial^2}{\partial y^2} + \frac{a}{y}\frac{\partial}{\partial y}.
\end{equation*}
Then, the proof becomes a simple computation. Abusing the notation we write 
\begin{equation*}
    G(x,y)=G(r,y)=(r^2+y^2)^{-\frac{n-2s}{2}}.
\end{equation*} 
We have
\begin{equation*}
    G_r=-(n+a-1)(r^2+y^2)^{-\frac{n+a-1}{2}-1}r,
\end{equation*}
\begin{equation*}
    G_{rr}=(n+a-1)(r^2+y^2)^{-\frac{n+a-1}{2}-2}((n+a)r^2-y^2).
\end{equation*}
This gives
\begin{equation*}
    G_{rr}+\frac{n-1}{r}G_r=(n+a-1)(r^2+y^2)^{-\frac{n+a-1}{2}-2}((1+a)r^2-ny^2).
\end{equation*}
On the other hand, a similar computation gives
\begin{equation*}
    G_{yy}+\frac{a}{y}G_y=-(n+a-1)(r^2+y^2)^{-\frac{n+a-1}{2}-2}((1+a)r^2-ny^2).
\end{equation*}
Adding the latter two equations gives the desired conclusion $L_a G=0$.
\end{proof}

%%%%%%%%%%%%%%%%%
%%%%%CHAPTER 3%%%%%
%%%%%%%%%%%%%%%%%

%!TEX root = ../dissertation.tex

\chapter{Fractional calculus}
\label{chp:3}

\section{The heat semigroup}\label{s:61}
Given a set $X$, a \emph{dynamical system} is a family $\{T(t) \}_{t\ge 0}$ of mappings $T(t):X \to X$ such that
\begin{itemize}
    \item $T(t+s)=T(t)T(s)$ for all $t,s \ge 0$,
    \item $T(0)=I_X$.
\end{itemize}
One can interpret $X$ as the set of all states of a system, $t \in [0,\infty)$ as time, and $T(t)$ as the map describing the change of a state $x\in X$ at time $t=0$ into the state $T(t)x$ at time $t>0$. When the state space $X$ is a vector space and each $T(t)$ is a linear operator on $X$, then $\{T(t) \}_{t\ge 0}$ is called a (one-parameter) \emph{semigroup} of operators. When $X$ is a normed space, we say that it is a semigroup of \emph{contractions} on $X$ if for every $t\ge0$
\begin{equation*}
    \norma{T(t)x}\le \norma{x}, \qquad x \in X.
\end{equation*}
When the normed space $X$ is a Banach space we say that a semigroup of bounded linear operators $\{ T(t)\}_{t\ge0}$ on $X$ is \emph{strongly continuous} if for every $x\in X$ its orbit map
\begin{equation*}
    t \longrightarrow T(t)x
\end{equation*}
is continuous from $[0,\infty)$ into X. Strongly continuous semigroups are important because they represent a generalisation of the exponential function $t \to e^{tA}$ of a matrix $A\in M_{n \times n}(\mathbb{C})$. Just as exponential functions provide a solutions of a scalar linear constant coefficient ordinary differential equations, strongly continuous semigroups provide solutions of linear constant coefficient ordinary differential equations in Banach spaces. Typically, such differential equations in Banach spaces arise from PDEs. 

%We will need the following Proposition
%\begin{proposition}
%\label{p:61}
%For a semigroup $\{T(t)\}_{t\ge0}$ on a Banach space $X$, the following assertions are equivalent.
%\begin{enumerate}
%    \item $\{T(t)\}_{t \ge 0}$ is strongly continuous.
%    \item $\lim_{t \to 0^+}T(t)x=x$ for all $x\in X$.
%    \item There exist $\delta >0 $, $M>1$, and a dense subset $D\subset X$ such that 
%    \begin{enumerate}
%        \item $\norma{T(t)}\le M$ for all $t\in[0,\delta]$,
%        \item $\lim_{t\to 0^+}T(t)x=x$ for all $x\in D$. 
%    \end{enumerate}
%\end{enumerate}
%\end{proposition}

We begin our discussion by considering the ubiquitous Gaussian
\begin{equation*}
    K(x)=(4\pi)^{-\frac{n}{2}}e^{-\frac{\abs{x}^2}{4}}.
\end{equation*}
Obviously, $K\in L^1(\Rn)$ and we easily have
\begin{equation*}
    \int_{\Rn}K(x)\,dx=1.
\end{equation*}
We next consider the following approximate identity associated with such kernel $K$
\begin{equation}
    \label{61}
    G(x,t):= t^{-\frac{n}{2}}K \left( \frac{x}{\sqrt{t}} \right)=(4\pi t)^{-\frac{n}{2}}e^{-\frac{\abs{x}^2}{4t}} \qquad t>0.
\end{equation}
The inquisitive reader might winder why we have scaled by $\sqrt{t}$ and not just $t$. This is due to the fact that the function $G(x,t)$ introduced in \eqref{61} is the fundamental solution of the heat operator $\partial_t-\Delta_x$, and for such operator the correct scaling is provided by the non-homogeneous (parabolic) dilations $\lambda \to (\lambda x, \lambda^2 t)$.

The next proposition contains an elementary but very important property of the function defined by \eqref{61}.
\begin{proposition}[Chapman-Kolmogorov equation]
\label{p:62}
For every $s,t >0$, $x,y\in\Rn$ one has
\begin{equation*}
    G(x-y,t+s)= \int_{\Rn} G(x-z,t)G(z-y,s)\,dz.
\end{equation*}
\end{proposition}
\begin{proof}
We note that, by translation, it suffices to prove such identity when $y=0$. We thus need to show that for every $x\in\Rn$ and $t,s>0$ we have
\begin{equation*}
     G(x,t+s)= \int_{\Rn} G(x-z,t)G(z,s)\,dz,
\end{equation*}
or equivalently 
\begin{equation*}
    (4 \pi (t+s))^{-\frac{n}{2}}e^{-\frac{\abs{x}^2}{4(t+s)}} = (4 \pi t)^{-\frac{n}{2}}(4 \pi s)^{-\frac{n}{2}}\int_{\Rn} e^{-\left( \frac{\abs{x-z}^2}{4t}+ \frac{\abs{z}^2}{4s} \right)}\,dz.
\end{equation*}
We now perform some elementary manipulations in the exponential in the integral in the right-hand side to find
\begin{equation*}
    \begin{split}
    \frac{\abs{x-z}^2}{4t}+\frac{\abs{z}^2}{4s}&= \frac{\abs{x}^2}{4(t+s)} + \frac{1}{4t}\left( \frac{4s}{4(t+s)}\abs{x}^2 + \abs{z}^2 -2\braket{x,z}+ \frac{\abs{z}^2}{4s}4t\right)\\
    &=\frac{\abs{x}^2}{4(t+s)} + \abs*{ \left( \frac{4s}{4t4(t+s)} \right)^{\frac{1}{2}}x - \left( \frac{4(t+s)}{4t4s} \right)^{\frac{1}{2}}z }^2.
    \end{split}
\end{equation*}
This gives
\begin{equation*}
    \begin{split}
        (4\pi t )^{-\frac{n}{2}}&(4\pi s)^{-\frac{n}{2}} \int_{\Rn} e^{-\left( \frac{\abs{x-z}^2}{4t}+ \frac{\abs{z}^2}{4s} \right)}\,dz\\
        &=(4\pi t )^{-\frac{n}{2}}(4\pi s)^{-\frac{n}{2}}e^{-\frac{\abs{x}^2}{4(t+s)}} \int_{\Rn} e^{- \abs*{ \left( \frac{4s}{4t 4(t+s)} \right)^{\frac{1}{2}}x - \left( \frac{4(t+s)}{4t 4s} \right)^{\frac{1}{2}}z}^2}\,dz.
    \end{split}
\end{equation*}
The change of variable
\begin{equation*}
    z \longrightarrow \xi = \left( \frac{4s}{4t 4(t+s)} \right)^{\frac{1}{2}}x - \left( \frac{4(t+s)}{4t 4s} \right)^{\frac{1}{2}}z,
\end{equation*}
for which we have
\begin{equation*}
    d\xi=\left( \frac{4(t+s)}{4t 4s} \right)^{\frac{n}{2}}\,dz,
\end{equation*}
now gives
\begin{equation*}
    \begin{split}
        (4\pi t )^{-\frac{n}{2}}&(4\pi s)^{-\frac{n}{2}} \int_{\Rn} e^{-\left( \frac{\abs{x-z}^2}{4t}+ \frac{\abs{z}^2}{4s} \right)}\,dz\\
        &=\frac{\pi^{-\frac{n}{2}}}{(4(t+s))^{\frac{n}{2}}}\int_{\Rn}e^{-\abs{\xi}^2}\,d\xi=(4\pi (t+s))^{-\frac{n}{2}}e^{-\frac{\abs{x}^2}{4(t+s)}},
    \end{split}
\end{equation*}
which finally proves the desired conclusion.
\end{proof}

The next result expresses a fundamental property of the function $G(x,t)$.
\begin{lemma}
\label{l:63}
For every $x \in Rn$ and $t>0$ one has
\begin{equation*}
    \partial_t G(x,t) - \Delta G(x,t)=0.
\end{equation*}
\end{lemma}
\begin{proof}
From the definition \eqref{61} it is immediate too verify that
\begin{equation*}
    \nabla G(x,t)= -\frac{x}{2t}G(x,t).
\end{equation*}
Then we find
\begin{equation*}
    \Delta G(x,t)= -\frac{1}{2t} \text{div }(G(\cdot,t)x)= -\frac{n}{2t}G(x,t)+\frac{\abs{x}^2}{4t}G(x,t).
\end{equation*}
On the other hand, differentiating in $t$ we easily find
\begin{equation*}
    \partial_t G(x,t)= -\frac{n}{2t}G(x,t) + \frac{\abs{x}^2}{4t}G(x,t). 
\end{equation*}
The desired conclusion follows.
\end{proof}

For reasons that will become clear subsequently we now introduce a special notation for the convolution with $G(\cdot,t)$
\begin{equation}
    \label{62}
    P_t f(x)= G(\cdot,t) \star f(x)= \int_{\Rn}G(x-y,t)f(y)\,dy,
\end{equation}
where $f$ is a measurable function on $\Rn$ for which the integral \eqref{62} makes sense. We can see that, as a linear operator, $P_t : L^p(\Rn) \to L^p(\Rn)$ for $1 \le p \le \infty$ and moreover
\begin{equation}
    \label{63}
    \norma{P_t f}_{L^p(\Rn)}\le \norma{f}_{L^p(\Rn)}.
\end{equation}
Furthermore, we have the following.
\begin{proposition}
\label{p:64}
Let $t,s>0$. For every $f\in L^p(\Rn)$ we have
\begin{equation*}
    P_{t+s}f=P_t(P_sf).
\end{equation*}
\end{proposition}
\begin{proof}
In view of \eqref{63} and Proposition \ref{p:62} we have
\begin{equation*}
    \begin{split}
        P_{t+s}f(x) &= \int_{\Rn} G(x-y,t+s)f(y)\, dy = \int_{\Rn} \int_{\Rn} G(x-z,t)G(z-y,s)\,dzf(y)\,dy\\
        &=\int_{\Rn}G(x-z,t)P_sf(z)\,dz=P_t(P_sf)(x).
    \end{split}
\end{equation*}
\end{proof}

\begin{theorem}
\label{t:65}
Let $1\le p < \infty$. If $f \in L^p(\Rn)$ we have
\begin{equation}
    \label{64}
    \lim_{t \to 0^+}\norma{P_t f -f}_{L^p(\Rn)}=0.
\end{equation}
If instead $f\in L^{\infty}(\Rn)$, then we have
\begin{equation}
    \label{65}
    \lim_{t \to 0^+}P_tf(x)=f(x)
\end{equation}
at every point $x \in \Rn$ of continuity for $f$.
\end{theorem}
As a consequence of Proposition \ref{p:64}, \eqref{63} and Theorem \ref{t:65} we obtain the following basic result.
\begin{proposition}
\label{p:66}
For every $1\le p\le \infty$ the one-parameter family $\{ P_t \}_{t \ge 0}$ is a semigroup of contractions on $L^p(\Rn)$. The semigroup is strongly continuous when $1\le p < \infty$.
\end{proposition}
The family of operators $\{ P_t \}_{t \ge 0}$ is called the \emph{heat semigroup} in $\Rn$. The name is justified by the fact that the function $u(x,t)=P_t f(x)$ solves the Cauchy problem for the heat equation $\partial_t u - \Delta u=0$ in $\Rn \times \mathbb{R}^+$. Denote the semigroup $P_t$ with the symbol $e^{t\Delta}$. Clearly, if we formally let $t=0$ we find $u(x,0)=f(x)$. Furthermore, by differentiating formally in $t$ one has
\begin{equation}
    \label{66}
    u_t(x,t)=\Delta e^{t \Delta}f(x)= \Delta u(x,t),
\end{equation}
therefore $u(x,t)=P_tf(x)$ satisfies the heat equation $\partial_t u - \Delta u=0$ in $\Rn \times \mathbb{R}^+$.
\begin{proposition}
\label{p:67}
Given $f \in C(\Rn) \cap L^{\infty}(\Rn)$, the function $u(x,t)=P_t f(x)$ is a solution of the Cauchy problem for the heat equation
\begin{equation*}
    \begin{cases}
    \partial_t u - \Delta u=0 \quad \text{ in } \Rn \times \mathbb{R}^+,\\
    u(x,0)=f(x).
    \end{cases}
\end{equation*}
\end{proposition}
Another fact suggested by \eqref{66} is the commutation property $P_t \Delta=\Delta P_t$. The next proposition establishes this.
\begin{proposition}
\label{p:68}
For every $f\in C_0^{\infty}(\Rn)$ one has for any $x\in \Rn$ and $t>0$
\begin{equation*}
    P_t(\Delta f)(x)=\Delta (P_t f)(x).
\end{equation*}
Moreover in general, if $f \in \Sn$ and $\beta \in \mathbb{N}_0^n$ one has
\begin{equation*}
    P_t(\partial^{\beta}f)(x)=\partial^{\beta}(P_tf)(x).
\end{equation*}
\end{proposition}
\begin{proof}
Let $f \in C_0^2(\Rn)$ and fix $R>0$ so large that $\text{supp } f \subset B(0,R)$. Then
\begin{equation*}
    \begin{split}
        P_t(\Delta f)(x) &= (4 \pi t )^{-\frac{n}{2}}\int_{B(0,R)}e^{-\frac{\abs{y-x}^2}{4t}}\Delta_y f(y)\,dy\\
        &=(4 \pi t )^{-\frac{n}{2}}\int_{B(0,R)}\Delta_y \left( e^{-\frac{\abs{y-x}^2}{4t}}\right)f(y)\,dy\\
        &=(4 \pi t )^{-\frac{n}{2}}\int_{B(0,R)}\Delta_x \left( e^{-\frac{\abs{y-x}^2}{4t}}\right)f(y)\,dy\\
        &= (4 \pi t )^{-\frac{n}{2}} \Delta_x\int_{B(0,R)}e^{-\frac{\abs{y-x}^2}{4t}}f(y)\,dy\\
        &= \Delta (P_t f)(x).
    \end{split}
\end{equation*}
\end{proof}

The following result is a simple but useful consequence of the commutation property in Proposition \ref{p:68}.
\begin{lemma}
\label{l:69}
Let $1 \le p \le \infty$. Given any $f \in \Sn$ for any $t \in [0,1]$ we have
\begin{equation*}
    \norma{P_t f - f}_p \le \norma{\Delta f}_p t.
\end{equation*}
\end{lemma}
\begin{proof}
By proposition \ref{p:67} and \ref{p:68} we have for any $f \in \Sn$,
\begin{equation*}
    P_t f(x) - f(x)= \int_0^t \frac{d}{d \tau} P_{\tau}f(x)\,d\tau=\int_0^t \Delta P_{\tau}f(x)\, d\tau= \int_0^t P_{\tau}\Delta f (x)\,d\tau.
\end{equation*}
This gives for any $0 \le t \le 1$,
\begin{equation*}
    \norma{P_t f - f}_p \le \int_0^t \norma{P_{\tau}\Delta f}_p \, d\tau \le \norma{\Delta f}_p \int_0^t\,d\tau= \norma{\Delta f}_p t.
\end{equation*}
where in the second inequality we have used \eqref{63}.
\end{proof}

\section{Ultracontractivity}\label{s:62}
We next establish a basic property of the semigroup $\{ P_t \}_{t \ge 0}$.
\begin{proposition}[Ultracontractivity]
\label{p:610}
Let $1\le p < \infty$ and $f \in L^p(\Rn)$. For every $x\in \Rn$ and $t>0$ we have
\begin{equation}
    \label{67}
    \abs{P_t f(x)} \le \frac{c(n,p)}{t^{\frac{n}{2p}}}\norma{f}_p.
\end{equation}
for a certain constant $c(n,p)>0$ (when $p=1$ one has $c(n,1)=(4\pi)^{-\frac{n}{2}}$). In particular, for any $f \in L^p(\Rn)$ and $x \in Rn$ one has
\begin{equation}
    \label{68}
    \lim_{t \to \infty}\abs{P_t f (x)}=0.
\end{equation}
\end{proposition}
\begin{proof}
Applying H\"{o}lder's inequality to \eqref{62} we find
\begin{equation*}
    \abs{P_t f(x)} \le \norma{f}_p \left( \int_{\Rn}G(x-y,t)^{p'}\,dy \right)^{\frac{1}{p'}}= \norma{f}_p (4 \pi t)^{-\frac{n}{2}}\left( \int_{\Rn} e^{-\frac{p'\abs{x-y}^2}{4t}}\,dy \right)^{\frac{1}{p'}},
\end{equation*}
with $1/p+1/p'=1$. By the change of variable $z=\sqrt{\frac{p'}{4t}}(y-x)$, for which $dz=\left(\frac{p'}{4t}\right)^{\frac{n}{2}}\,dy$, the desired conclusion immediately follows with $c(n,p)=\left( \frac{1}{p'}\right)^{\frac{n}{2p'}}(4\pi)^{-\frac{n}{2p}}$.
\end{proof}

Combining Proposition \ref{p:68} and \ref{p:610} we now establish the following remarkable instance of the \emph{subordination principle} of Bochner.

\begin{proposition}
\label{p:611}
Let $n\ge3$. Then the following equation holds in $\mathscr{D}'(\Rn)$
\begin{equation*}
    \Delta_y \int_0^{\infty}G(x-y,t)\,dt=-\delta_x,
\end{equation*}
where $\delta_x$ indicates the Dirac delta at $x$.
\end{proposition}
\begin{proof}
To establish the proposition we need to show that:
\begin{enumerate}
    \item for every $x \in \Rn$ the function $y\to \int_0^{\infty}G(x-y,t)\,dt$ defines an element of $\mathscr{D}'(\Rn)$;
    \item for every $f\in C_0^{\infty}(\Rn)$ one has 
    \begin{equation*}
        \braket{\int_0^{\infty}G(x-y,t)\,dt,\Delta_y f}=-f(x).
    \end{equation*}
\end{enumerate}
Concerning (1), we notice the following remarkable fact
\begin{equation}
    \label{69}
    \int_0^{\infty} G(x-y,t)\,dt= \frac{\Gamma \left( \frac{n-2}{2} \right)}{4\pi^{\frac{n}{2}}}\frac{1}{\abs{x-y}^{n-2}}.
\end{equation}
Equation \eqref{69} follows simply by the change of variable $\tau=\frac{\abs{x-y}^2}{4t}$ in the integral
\begin{equation*}
    \int_0^{\infty} G(x-y,t)\,dt=(4\pi)^{-\frac{n}{2}}\int_0^{\infty}\frac{1}{t^{\frac{n-2}{2}}}e^{-\frac{\abs{x-y}^2}{4t}}\frac{dt}{t}.
\end{equation*}
Since the function $y\to \abs{x-y}^{2-n}$ belongs to $L_{\text{loc}}^1(\Rn)$, conclusion (1) follows from \eqref{69}. To prove (2) let $f\in C_0^{\infty}(\Rn)$. We have by Fubini's theorem
\begin{equation*}
    \begin{split}
        &\braket{\Delta_y \int_0^{\infty} G(x-y,t)\,dt,f}=\braket{\int_0^{\infty} G(x-y,t)\,dt,\Delta_yf}\\
        &=\int_{\Rn}\int_0^{\infty}G(x-y,t)\,dt \Delta_y f(y)\,dy=\int_0^{\infty}\int_{\Rn}G(x-y,t) \Delta_y f(y)\,dy\,dt\\
        &=\int_0^{\infty}P_t(\Delta f)(x)\,dt=\int_0^{\infty}\Delta P_t f(x)\,dt,
    \end{split}
\end{equation*}
where in the last equality we have used \eqref{68}. Since Proposition \ref{p:67} gives $\Delta P_tf(x)=\partial_t P_t f(x)$, by the above computation we find 
\begin{equation*}
    \begin{split}
        &\braket{\Delta_y  \int_0^{\infty} G(x-y,t)\,dt,f}= \int_0^{\infty}\partial_{t}P_tf(x)\,dt=\lim_{T\to\infty,\, \epsilon \to 0^+}\int_{\epsilon}^{T}\partial_tP_tf(x)\,dt\\
        &=\lim_{T\to\infty,\, \epsilon \to 0^+}[P_T f(x)-P_{\epsilon}f(x)]=-f(x),
    \end{split}
\end{equation*}
where in the last equality we have used \eqref{65} and \eqref{640}.
\end{proof}

\section{Fractional powers of the Laplacian}\label{s:64}
In the Chapter \ref{chp:2} we introduced the fractional Laplacian, see \eqref{fraclap}. In this section, we will not follow the original presentation of the subject, but instead use the semigroup $\{ P_t\}_{t\ge0}$ and only subsequently recognise the equivalence of the two notions. The next definition was originally set forth by Balakrishnan and can be be deduced from his seminal works \citep{Bal59}, \citep{Bal60}.
\begin{definition}[Balakrishnan]
\label{d:615}
Let $0<\alpha<2$. The fractional Laplacian of order $\frac{\alpha}{2}$ is defined on a function $f \in \Sn$ by the formula
\begin{equation*}
    (-\Delta)^{\frac{\alpha}{2}}f(x)=-\frac{\frac{\alpha}{2}}{\Gamma \left( 1- \frac{\alpha}{2}\right)} \int_0^{\infty} \frac{1}{t^{1+\frac{\alpha}{2}}}(P_tf(x)-f(x))\,dt.
\end{equation*}
\end{definition}
Our first observation is that Definition \ref{d:615} makes sense, i.e., that the integral in the right-hand side is finite. This will be a consequence of the following result.
\begin{lemma}
\label{l:616}
There exists a constant $C(n)>0$ such that for every $f \in C^2(\Rn)$, with second derivatives in $L^{\infty}(\Rn)$, we have
\begin{equation*}
    \abs{P_tf(x)-f(x)}\le C(n)\norma{\nabla^2f}_{L^{\infty}(\Rn)}t.
\end{equation*}
\end{lemma}
\begin{proof}
As a first observation we claim that 
\begin{equation}
    \label{612}
    P_t f(x)-f(x)= \frac{1}{2} \int_{\Rn} G(y,t) [f(x+y)+f(x-y)-2f(x)]\,dy.
\end{equation}
Next, we observe that the $C^2$ Taylor formula gives
\begin{equation*}
    \begin{split}
        &f(x+y)=f(x)+ \braket{\nabla f(x),y} + \frac{1}{2}\braket{\nabla^2 f(y^{\star})y,y},\\
        &f(x-y)=f(x)- \braket{\nabla f(x),y} + \frac{1}{2}\braket{\nabla^2 f(y^{\star\star})y,y}.
    \end{split}
\end{equation*}
This implies
\begin{equation*}
    \abs{f(x+y) + f(x-y) -2f(x)}\le \norma{\nabla^2 f}_{L^{\infty}(\Rn)}\abs{y}^2.
\end{equation*}
Substituting this estimate in \eqref{612} we obtain
\begin{equation*}
    \abs{P_tf(x)-f(x)}\le \frac{1}{2}\norma{\nabla^2f}_{L^{\infty}(\Rn)}\int_{\Rn}G(y,t)\abs{y}^2\,dy\le C(n)\norma{\nabla^2f}_{L^{\infty}(\Rn)}t,
\end{equation*}
This proves the desired result.
\end{proof}

With Lemma \ref{l:616} we can now show that the integral defining $(-\Delta)^{\frac{\alpha}{2}}f(x)$ is finite for every $x\in Rn$. We have in fact
\begin{equation*}
    \begin{split}
        \int_0^{\infty}& \frac{1}{t^{1+\frac{\alpha}{2}}}(P_tf(x)-f(x))\,dt\\
        &=\int_0^{1} \frac{1}{t^{1+\frac{\alpha}{2}}}(P_tf(x)-f(x))\,dt+\int_1^{\infty} \frac{1}{t^{1+\frac{\alpha}{2}}}(P_tf(x)-f(x))\,dt.
    \end{split}
\end{equation*}
The integral on $(0,1)$ is finite thanks to Lemma \ref{l:616}. The integral on $(1,\infty)$ is trivially estimated as follows
\begin{equation*}
    \abs*{\int_1^{\infty} \frac{1}{t^{1+\frac{\alpha}{2}}}(P_tf(x)-f(x))\,dt}\le 2 \norma{f}_{L^{\infty}(\Rn)}\int_1^{\infty}\frac{1}{t^{1+\frac{\alpha}{2}}}\,dt < \infty,
\end{equation*}
where we have used \eqref{63} to infer $\norma{P_t f}_{L^{\infty}(\Rn)}\le \norma{f}_{L^{\infty}(\Rn)}$.

We close this section by establishing an "integration by parts" formula for the operator $(-\Delta)^{\frac{\alpha}{2}}$.
\begin{proposition}
\label{p:621}
Let $0<\alpha<2$. For any $f,g\in \Sn$ one has
\begin{equation*}
    \int_{\Rn}\left[ g(-\Delta)^{\frac{\alpha}{2}}f-f(-\Delta)^{\frac{\alpha}{2}}g \right]\,dx=0.
\end{equation*}
\end{proposition}
\begin{proof}
We have
\begin{equation*}
    \begin{split}
        \int_{\Rn}& g(x) (-\Delta)^{\frac{\alpha}{2}}f(x)\,dx= - \frac{\frac{\alpha}{2}}{\Gamma \left( 1-\frac{\alpha}{2}\right)} \int_{\Rn}g(x)\int_0^{\infty}\frac{1}{t^{1+\frac{\alpha}{2}}}(P_tf(x)-f(x))\,dt\,dx\\
        &=- \frac{\frac{\alpha}{2}}{\Gamma \left( 1-\frac{\alpha}{2}\right)} \int_0^{\infty}\frac{1}{t^{1+\frac{\alpha}{2}}}\int_{\Rn}g(x)(P_tf(x)-f(x))\,dx\,dt\\
        &=- \frac{\frac{\alpha}{2}}{\Gamma \left( 1-\frac{\alpha}{2}\right)} \int_0^{\infty}\frac{1}{t^{1+\frac{\alpha}{2}}}\int_{\Rn}f(x)(P_tg(x)-g(x))\,dx\,dt\\
        &=\int_{\Rn} f(x) (-\Delta)^{\frac{\alpha}{2}}g(x)\,dx.
    \end{split}
\end{equation*}
\end{proof}

We note that Proposition \ref{p:621} continues to be true if we replace the hypothesis $g\in\Sn$ with $g\in C^2(\Rn)$ with bounded second derivatives. With this observation we obtain the following.
\begin{corollary}
\label{c:622}
For any $f \in \Sn$ one has 
\begin{equation*}
    \int_{\Rn} (-\Delta)^{\frac{\alpha}{2}}f\,dx=0.
\end{equation*}
\end{corollary}

\section{Balakrishnan met M. Riesz}\label{s:65}
In this section we show that Balakrishnan's definition of the nonlocal operator $(-\Delta)^{\frac{\alpha}{2}}$ coincides with that introduced by M. Riesz in \citep{R38}. Subsequently, we analyse the asymptotic behaviour of this operator as $\alpha \nearrow 2$ and we show that, unsurprisingly, in the limit we obtain the negative of the Laplace operator $\Delta$.
\begin{proposition}
\label{p:623}
Let $0<\alpha<2$. For every $f \in \Sn$ one has
\begin{equation*}
    (-\Delta)^{\frac{\alpha}{2}}f(x)=\frac{\alpha 2^{\alpha-2}\Gamma \left( \frac{n+\alpha}{2} \right)}{\pi^{\frac{n}{2}}\Gamma \left( 1-\frac{\alpha}{2} \right)}\int_{\Rn}\frac{2f(x)-f(x+y)-f(x-y)}{\abs{y}^{n+\alpha}}\,dy.
\end{equation*}
\end{proposition}
\begin{proof}
Using equation \eqref{612} and Fubini's theorem we find
\begin{equation*}
    \begin{split}
        (-\Delta)^{\frac{\alpha}{2}}&f(x)=- \frac{\frac{\alpha}{2}}{2 \Gamma \left( 1-\frac{\alpha}{2} \right)}\int_0^{\infty}\frac{1}{t^{1+\frac{\alpha}{2}}}\int_{\Rn}G(y,t)[f(x+y)+f(x-y)-2f(x)]\,dy\,dt\\
        &=\frac{\frac{\alpha}{2}}{2 \Gamma \left( 1-\frac{\alpha}{2} \right)}\int_{\Rn}[f(x+y)+f(x-y)-2f(x)]\int_0^{\infty}\frac{1}{t^{1+\frac{\alpha}{2}}}G(y,t)\,dt\,dy.
    \end{split}
\end{equation*}
To complete the proof all is needed at this point is the following elementary computation
\begin{equation}
    \label{616}
    \begin{split}
        \int_0^{\infty}& \frac{1}{t^{1+\frac{\alpha}{2}}}G(y,t)\,dt=(4\pi)^{-\frac{n}{2}}\int_0^{\infty}\frac{1}{t^{\frac{n+\alpha}{2}}}e^{-\frac{\abs{y}^2}{4t}}\,\frac{dt}{t}\\
        &=(4\pi)^{-\frac{n}{2}}2^{n+\alpha}\Gamma \left( \frac{n+\alpha}{2} \right)\abs{y}^{-(n+\alpha)}.
    \end{split}
\end{equation}
\end{proof}

We next analyse the limit of $(-\Delta)^{\frac{\alpha}{2}}$ as $\alpha \nearrow 2$. 
\begin{proposition}
\label{p:624}
Let $f\in \Sn$. Then for any $x\in \Rn$ one has
\begin{equation*}
    \lim_{\alpha \nearrow 2}(-\Delta)^{\frac{\alpha}{2}}f(x)=-\Delta f(x).
\end{equation*}
\end{proposition}
\begin{proof}
We write
\begin{equation*}
    \begin{split}
        (-\Delta)^{\frac{\alpha}{2}}f(x)&= - \frac{\frac{\alpha}{2}}{ \Gamma \left( 1-\frac{\alpha}{2} \right)}\int_0^{\infty} \frac{1}{t^{1+\frac{\alpha}{2}}}(P_tf(x)-f(x))\,dt\\
        &=- \frac{\frac{\alpha}{2}}{ \Gamma \left( 1-\frac{\alpha}{2} \right)}\int_0^{1} \frac{1}{t^{1+\frac{\alpha}{2}}}(P_tf(x)-f(x))\,dt - \frac{\frac{\alpha}{2}}{ \Gamma \left( 1-\frac{\alpha}{2} \right)}\int_1^{\infty} \frac{1}{t^{1+\frac{\alpha}{2}}}(P_tf(x)-f(x))\,dt.\\
        &= \mathbb{I}(\alpha) + \mathbb{II}(\alpha).
    \end{split}
\end{equation*}
The second term is easily estimated as follows
\begin{equation*}
    \abs{\mathbb{II}(\alpha)}\le 2 \norma{f}_{L^{\infty}(\Rn)}\frac{\frac{\alpha}{2}}{ \Gamma \left( 1-\frac{\alpha}{2} \right)}\int_1^{\infty}\frac{1}{t^{1+\frac{\alpha}{2}}}\,dt \to 0, \quad \text{as } \alpha \to 2,
\end{equation*}
since $\left( 1-\frac{\alpha}{2}\right)\Gamma \left( 1-\frac{\alpha}{2} \right)=1+o(1)$ as $\alpha \nearrow 2$. For the first term we argue in the following way. By Lemma \ref{l:616} we can integrate by parts obtaining
\begin{equation*}
    \begin{split}
        \mathbb{I}(\alpha)&=-\frac{\frac{\alpha}{2}}{ \Gamma \left( 1-\frac{\alpha}{2} \right)} \int_0^1\left( \frac{t^{-\frac{\alpha}{2}}}{-\frac{\alpha}{2}} \right)' (P_t f(x) - f(x))\,dt\\
        &=\frac{1}{ \Gamma \left( 1-\frac{\alpha}{2} \right)}(P_1f(x)-f(x)) -\frac{1}{ \Gamma \left( 1-\frac{\alpha}{2} \right)} \int_0^1 t^{-\frac{\alpha}{2}}\frac{d}{dt}P_t f(x)\,dt\\
        &=\frac{1}{ \Gamma \left( 1-\frac{\alpha}{2} \right)}(P_1f(x)-f(x)) -\frac{1}{ \Gamma \left( 1-\frac{\alpha}{2} \right)} \int_0^1 t^{-\frac{\alpha}{2}}\Delta P_t f(x)\,dt\\
        &=\frac{1}{ \Gamma \left( 1-\frac{\alpha}{2} \right)}(P_1f(x)-f(x)) -\frac{1}{ \Gamma \left( 1-\frac{\alpha}{2} \right)} \int_0^1 t^{-\frac{\alpha}{2}} P_t \Delta f(x)\,dt.
    \end{split}
\end{equation*}
where in the last two equalities we have used Proposition \ref{p:67} ans \ref{p:68}. Since $\Delta f \in \Sn$, by Lemma \ref{l:616} again we can write for $t\in[0,1]$
\begin{equation*}
    P_t \Delta f(x)= \Delta f(x) + O(t).
\end{equation*}
Therefore
\begin{equation*}
    \frac{1}{ \Gamma \left( 1-\frac{\alpha}{2} \right)} \int_0^1 t^{-\frac{\alpha}{2}} P_t \Delta f(x)\,dt = \frac{1}{ \left( 1-\frac{\alpha}{2} \right) \Gamma \left( 1-\frac{\alpha}{2} \right)} \Delta f(x) + o(1)
\end{equation*}
as $\alpha \nearrow 2$. Substituting in the above expression of $\mathbb{I}(\alpha)$ we conclude that 
\begin{equation*}
    \mathbb{I}(\alpha) \to -\Delta f(x), \qquad \text{as } \alpha \nearrow 2,
\end{equation*}
thus completing the proof.
\end{proof}

%We close this section by also analysing the behaviour of $(-\Delta)^{\frac{\alpha}{2}}$ at the other end-point $\alpha \searrow 0$. As one expects, one should approach in the limit the identity operator. 
%\begin{proposition}
%\label{p:625}
%Let $f\in \Sn$. One has for every $x \in Rn$
%\begin{equation*}
%    \lim_{\alpha \to 0^+}(-\Delta)^{\frac{\alpha}{2}}f(x)=f(x).
%\end{equation*}
%\end{proposition}

\section{The evolutive semigroup}
\label{s:68}
In semigroup theory a procedure for forming a new semigroup from a given one is that of evolution semigroup. In this section we exploit this idea to introduce a new semigroup that will be used as a building block for: (i) defining the fractional powers of the heat operator $H=\Delta-\partial_t$; (ii) solve the extension problem for such nonlocal operators. 

Let us to introduce the following operator on functions
\begin{equation}
    \label{632}
    P_{\tau}^Hf(x,t)=\int_{\Rn}G(x-y,\tau)f(y,t-\tau)\,dy=P_{\tau}(\Lambda_{-\tau}f(\cdot,t))(x),
\end{equation}
and call it the \emph{evolutive semigroup}. The motivation for such name is in the fact that $\{P_{\tau}^H\}_{\tau >0}$ does in fact define a semigroup of contractions on $L^p(\Rnn)$, for $1\le p \le \infty$, where hereafter we use the notation $\Rnn$ to indicate the space $\Rn \times \mathbb{R}$ with respect the variables $(x,t)$. One has in fact from $\eqref{632}$ for $\tau,\sigma >0$
\begin{equation}
    \label{633}
    \begin{split}
        P_{\tau +\sigma}^H f(x,t)&= P_{\tau+\sigma}(\Lambda_{-\tau-\sigma}f(\cdot,t))(x)\\
        &=P_{\tau}(P_{\sigma}(\Lambda_{-\tau-\sigma}f(\cdot,t)))(x)\\
        &=P_{\tau}(\Lambda_{-\tau}(P_{\sigma}(\Lambda_{-\sigma}f(\cdot,t))))(x)\\
        &=P_{\tau}^H(P_{\sigma}^Hf)(x,t).
    \end{split}
\end{equation}
Furthermore, one has
\begin{equation}
    \label{634}
    \lim_{\tau \to 0^+}P_{\tau}^Hf(x,t)=f(x,t).
\end{equation}
The following two lemmas summarise the semigroup $P_{\tau}^H$ properties that have already been proved for $P_t$.
\begin{lemma}
\label{l:636}
For any $t>0$ we have:
\begin{enumerate}
    \item $H(\Snn)\subset\Snn$ and $P_{\tau}^H(\Snn)\subset\Snn$;
    \item For any $f \in \Snn$ and $(x,t)\in \Rnn$ one has $\partial_{\tau}P_{\tau}^Hf(x,t)=HP_{\tau}^Hf(x,t)$;
    \item For every $f \in \Snn$ and $(x,t)\in \Rnn$ the commutation property is true
        \begin{equation*}
            HP_{\tau}^Hf(x,t)=P_{\tau}^HHf(x,t).
        \end{equation*}
\end{enumerate}
\end{lemma}
\begin{proof}
(1) The first part is obvious. For the second part it suffices to show that $\reallywidehat{P_{\tau}^Hf}\in \Snn$ if $f \in \Snn$, and this follows from the following formula,
\begin{equation}
    \label{635}
    \reallywidehat{P_{\tau}^Hf}(\xi,\sigma)=e^{-\tau(4\pi^2\abs{\xi}^2 + 2\pi i \sigma)}\reallywidehat{f}(\xi,\sigma).
\end{equation}
(2) Is a consequence of the definition of $P_{\tau}^H$\\
(3) Follows immediately from \eqref{635} and the fact that
\begin{equation*}
    \reallywidehat{Hf}(\xi,\sigma)=-(4\pi^2\abs{\xi}^2 + 2\pi i \sigma)\reallywidehat{f}(\xi,\sigma),
\end{equation*}
or from the commutation property $\Delta P_t=P_t \Delta$ in Proposition \ref{p:68}, and from the relations $P_{\tau}^Hf=P_{\tau}(\Lambda_{-\tau}f)$, $H\Lambda_{-\tau}=\Lambda_{-\tau}H$.
\end{proof}

Henceforth, we will often use some mixed Lebesgue spaces which represent the appropriate substitute for the standard $L^p$ spaces when dealing with anisotropic partial differential operators such as the heat operator. Given a measurable function $f(x,t)$ on $\Rnn$, and exponents $1\le p$, $q \le \infty$, we will write $f \in L^{p}(\mathbb{R},L^q(\Rn))$ to indicate the fact that
\begin{equation*}
    \norma{f}_{L^p(\mathbb{R},L^q(\Rn))}=\left( \int_{\mathbb{R}} \norma{f(\cdot,t)}_{L^q(\Rn)}^p \,dt \right)^{\frac{1}{p}} < \infty,
\end{equation*}
with obvious changes when $p=\infty$. It is clear that $L^p(\mathbb{R},L^q(\Rn))=L^p(\Rnn)$.
\begin{lemma}
\label{l:637}
The following properties hold:
\begin{enumerate}
    \item For every $(x,t) \in \Rnn$ and $\tau>0$ we have $P_{\tau}^H 1(x,t)=1$;
    \item We have $P_{\tau+\sigma}^H=P_{\tau}^H \circ P_{\sigma}^H$ for every $\sigma,\tau >0$.
    \item Let $1\le p \le \infty$, then $P_{\tau}^H:L^p(\Rnn) \to L^p(\Rnn)$ with $\norma{P_{\tau}^H}_{L^p \to L^p}\le 1$. Therefore, $\{ P_{\tau}^H \}_{\tau >0}$ is a semigroup of contractions on $L^p(\Rnn)$ when $1\le p \le \infty$.
\end{enumerate}
\end{lemma}
\begin{proof}
The proof (1) and (2) have already been given. We only provide the details of (3). If $p=\infty$, then it is immediate to see
\begin{equation}
    \label{636}
    \norma{P_{\tau}^Hf}_{L^{\infty}(\Rnn)}\le \norma{f}_{L^{\infty}(\Rnn)},
\end{equation}
thus we assume that $1\le p < \infty$. Using the second equality in \eqref{632}, \eqref{63} and Tonelli's theorem we have for any $f \in L^p(\Rnn)$
\begin{equation*}
    \begin{split}
        \norma{P_{\tau}^H f}_{L^q (\Rnn)} &= \left( \int_{\mathbb{R}} \norma{P_{\tau}(\Lambda_{-\tau}f(\cdot,t))}_{L^p(\Rn)}^p\,dt \right)^{\frac{1}{p}} \le \left( \int_{\mathbb{R}} \norma{\Lambda_{-\tau}f(\cdot,t)}_{L^p(\Rn)}^p\,dt \right)^{\frac{1}{p}}\\
        &= \left( \int_{\mathbb{R}} \norma{f(\cdot,t)}_{L^p(\Rn)}^p\,dt \right)^{\frac{1}{p}} =\norma{f}_{L^p(\Rnn)}.
    \end{split}
\end{equation*}
\end{proof}

We conclude this section with the analogue of Lemma \ref{l:69} for the semigroup $\{ P_{\tau}^H\}_{\tau >0}$ and an important consequence of it.
\begin{lemma}
\label{l:638}
For every $f\in \Snn$ and $(x,t) \in \Rnn$ we have
\begin{equation*}
    \abs{P_{\tau}^Hf(x,t)-f(x,t)} \le \norma{Hf}_{L^{\infty}(\Rnn)}\tau.
\end{equation*}
\end{lemma}
\begin{proof}
We have
\begin{equation*}
    P_{\tau}^H f(x,t)-f(x,t)= \int_0^{\tau}\frac{d}{d \sigma}P_{\sigma}^Hf(x,t)\,d\sigma=\int_0^{\tau}HP_{\sigma}^Hf(x,t)\,d\sigma=\int_0^{\tau}P_{\sigma}^HHf(x,t)\,d\sigma,
\end{equation*}
where in the last equality we have used the commutation property (3) in Lemma \ref{l:636}. This gives, using \eqref{636},
\begin{equation*}
    \begin{split}
        \abs{P_{\tau}f(x,t)-f(x,t)}&\le \int_0^{\tau} \abs{P_{\sigma}^HHf(x,t)}\,d\sigma \le \int_0^{\tau} \norma{P_{\sigma}^HHf(x,t)}_{L^{\infty}(\Rnn)}\,d\sigma\\
        &\le \int_0^{\tau} \norma{Hf(x,t)}_{L^{\infty}(\Rnn)}\,d\sigma=\norma{Hf(x,t)}_{L^{\infty}(\Rnn)}\tau.
    \end{split}
\end{equation*}
\end{proof}

Arguing in a similar way one proves the following.
\begin{lemma}
\label{l:639}
Let $1\le p \le \infty$. Given any $f \in \Snn$ for any $\tau \in [0,1]$ we have
\begin{equation*}
    \norma{P_{\tau}^Hf-f}_{L^p(\Rnn)}\le \norma{H f}_{L^p(\Rnn)}\tau.
\end{equation*}
\end{lemma}
\begin{corollary}
\label{c:640}
Let $1\le p < \infty$. For any $f\in L^p(\Rnn)$ one has 
\begin{equation*}
    \lim_{\tau \to 0^+}\norma{P_{\tau}^H f -f }_{L^p(\Rnn)}=0.
\end{equation*}
As a consequence, $\{ P_{\tau}^H \}_{\tau>0}$ is a strongly continuous semigroup of contractions on $L^p(\Rnn)$.
\end{corollary}
\begin{proof}
Since $\Snn$ is dense in $L^p(\Rnn)$, for every $\epsilon>0$ there exists $\psi \in \Snn$ such that 
\begin{equation*}
    \norma{f-\psi}_{L^p(\Rnn)} < \frac{\epsilon}{3}.
\end{equation*}
Having fixed $\psi$ in this way, in view of Lemma \ref{l:639} there exists $\tau_0=\tau_0(\epsilon)>0$ such that for every $0<\tau<\tau_0$ we have
\begin{equation*}
    \norma{P_{\tau}^H \psi-\psi}_{L^p(\Rnn)} < \frac{\epsilon}{3}.
\end{equation*}
This gives for every $0<\tau<\tau_0$
\begin{equation*}
    \begin{split}
        \norma{P_{\tau}^H f -f }_{L^{p}(\Rnn)}&\le \norma{P_{\tau}^H (f -\psi) }_{L^{p}(\Rnn)} + \norma{P_{\tau}^H \psi -\psi }_{L^{p}(\Rnn)} +\norma{\psi -f }_{L^{p}(\Rnn)}\\
        &\le \norma{\psi -f }_{L^{p}(\Rnn)} +\frac{\epsilon}{3} +\frac{\epsilon}{3}<\epsilon,
    \end{split}
\end{equation*}
where we have used (3) of Lemma \ref{l:637}.
\end{proof}

\section{The fractional heat operator}\label{s:69}
With the results of Section \ref{s:68} in hand we are now ready to introduce the fractional powers $(\partial_t-\Delta)^{\frac{\alpha}{2}}$.
\begin{definition}
\label{d:641}
Let $0<\alpha<2$. The fractional heat operator of order $\frac{\alpha}{2}$ is defined on a function $f \in \Snn$ by the formula
\begin{equation*}
    (\partial_t-\Delta)^{\frac{\alpha}{2}}f(x,t)=-\frac{\frac{\alpha}{2}}{\Gamma \left(1-\frac{\alpha}{2} \right)}\int_0^{\infty}\frac{1}{\tau^{1+\frac{\alpha}{2}}}(P_{\tau}^Hf(x,t)-f(x,t))\,d\tau.
\end{equation*}
\end{definition}
We observe right-away that, thanks to \eqref{636} and Lemma \ref{l:638} the integral in the right-hand side of \eqref{d:641} is finite. It is also worth observing that $f(x,t)=f(x)$, then
\begin{equation*}
    (\partial_t-\Delta)^{\frac{\alpha}{2}}f(x,t)=(-\Delta)^{\frac{\alpha}{2}}f(x).
\end{equation*}
Next, we observe that if we presently define the parabolic dilations
\begin{equation*}
    \delta_{\lambda}f(x,t)=f(\lambda x, \lambda^2 t),
\end{equation*}
the simple manipulations show that
\begin{equation}
    \label{637}
    P_{\tau}^H(\delta_{\lambda}f)(x,t)=P_{\lambda^2\tau}^Hf(\lambda x, \lambda^2 t).
\end{equation}
One easily obtains from \eqref{637}
\begin{equation}
    \label{638}
    (\partial_t-\Delta)^{\frac{\alpha}{2}}(\delta_{\lambda}f)(x,t)=\lambda^{\alpha}(\partial_t-\Delta)^{\frac{\alpha}{2}}f(\lambda x,\lambda^2t),
\end{equation}
which shows that the fractional heat is an operator of order $\alpha$ with respect to the anisotropic parabolic dilations.

We conclude this section with the next result, which is an useful version of Proposition \ref{p:610}
\begin{proposition}[weak Ultracontractivity]
\label{p:643}
Let $1\le q< \infty$ and $f \in L^{\infty}(\mathbb{R},L^q(\Rn))$. For every $(x,t)\in \Rnn$ and $\tau>0$ we have
\begin{equation}
    \label{639}
    \abs{P_{\tau}^Hf(x,t)}\le \frac{c(n,q)}{\tau^{\frac{n}{2q}}}\norma{f}_{L^{\infty}(\mathbb{R},L^q(\Rn))},
\end{equation}
for a certain constant $c(n,q)>0$. In particular, for any $f \in L^{\infty}(\mathbb{R},L^q(\Rn))$ and $(x,t)\in \Rnn$ one has
\begin{equation}
    \label{640}
    \lim_{\tau \to \infty}\abs{P_{\tau}^Hf(x,t)}=0.
\end{equation}
\end{proposition}
\begin{proof}
Applying H\"{o}lder's inequality to \eqref{632} we find
\begin{equation*}
    \begin{split}
        \abs{P_{\tau}^H f(x,t)} & \le \int_{\Rn}G(x-y,\tau)\abs{f(y,t-\tau)}\,dy\\
        &\le \norma{f(\cdot,t-\tau)}_{L^q(\Rn)}\left( \int_{\Rn} G(x-y,\tau)^{q'}\,dy \right)^{\frac{1}{q'}}\\
        &\le c(n,q) \tau^{-\frac{n}{2q}}\norma{f}_{L^{\infty}(\mathbb{R},L^q(\Rn))},
    \end{split}
\end{equation*}
where $\frac{1}{q}+\frac{1}{q'}=1$.
\end{proof}

\section{The extension problem} \label{s:teb}
In their seminal 2007 paper \citep{CS07} Caffarelli and Silvestre introduced the \emph{extension problem} for the fractional powers of the Laplacian, that we have already proved in section \ref{s:expb}. In line with the spirit of the present chapter we are going to use the heat semigroup to solve, again, the \emph{extension problem}.

In what follows we consider the upper half-space $\Rnn_+ = \Rn_x \times \mathbb{R}_y^+ $, with variable $X=(x,y)$, where $x\in \Rn$ and $y>0$. Let $0< \alpha <2$ and introduce what we call the \emph{Bessel parameter} $a=1 -\alpha \in (-1,1)$. Given a function $f \in \Snn$ we want to find a function $U \in C^{\infty}(\Rnn \times \mathbb{R}^+)$ such that
\begin{equation}
    \label{651}
    \begin{cases}
    y^a\partial_t U - \dive_X(y^a \nabla_XU)=0 \quad \text{in }\Rnn \times \mathbb{R}^+,\\
    U(x,t,0)=f(x,t) \qquad \qquad \quad (x,t) \in \Rnn. 
    \end{cases}
\end{equation}
When $\alpha=1$, and therefore $a=0$, the problem \eqref{651} was first introduced and solved by Frank Jones in his beautiful but apparently not so well-known 1968 paper \citep{Jon68}. He also proved the following identity
\begin{equation}
    \label{652} 
    -\lim_{y \to 0^+} \frac{\partial U}{\partial y}(x,t,y)= (\partial_t - \Delta)^{\frac{1}{2}}f(x,t).
\end{equation}
More recently, Nystr\"om \& Sande \citep{NS} and Stinga \& Torrea \citep{ST} have independently generalized Jones' results to all fractional powers $0 < s < 1$. Let us notice right-away that since the right-hand side of the PDE in \eqref{651} is zero, we can factor $y^a$ out and write the problem in the equivalent form
\begin{equation}
    \label{653}
    \begin{cases}
    \mathfrak{B}^{(a)}U=\partial_{yy}U + \frac{a}{y}\partial_yU + \Delta_x U - \partial_t U=0 \quad \text{in }\Rnn \times \mathbb{R}^+,\\
    U(x,t,0)=f(x,t) \qquad \qquad \quad (x,t) \in \Rnn. 
    \end{cases}
\end{equation}
We call $\mathfrak{B}^{(a)}$ the \emph{extension operator}. To find its heat kernel we make the following formal considerations, which can be justified a posteriori. Denote by $w$ a point in the Euclidean space $\mathbb{R}^{a+1}$ with fractional dimension $a+1$. Never mind for the time being the fact that this really makes no sense. If we denote by $y=\abs{w}$ (again, this is purely formal), then the Laplacian in the variable $w$, restricted to functions having spherical symmetry, takes the form
\begin{equation*}
    \Delta_w= \partial_{yy}+\frac{a}{y}\partial_y.
\end{equation*}
This suggests that we should look at the following PDE in $\mathbb{R}^{n+a+1}\times \mathbb{R}^+$
\begin{equation}
    \label{654}
    \Delta_w U + \Delta_x U - \partial_t U=0.
\end{equation}
The heat kernel for \eqref{654} is given by by the Gaussian in $\mathbb{R}^{n+a+1}\times \mathbb{R}^+$
\begin{equation}
    \label{655}
    G^{(a)}((x,w),t)=(4 \pi t)^{-\frac{n+a+1}{2}}e^{-\abs{(x,w)}^2}4t=(4\pi t)^{-\frac{n+a+1}{2}}e^{-\frac{\abs{x}^2+y^2}{4t}},
\end{equation}
where we have used the "assumption" that $y=\abs{w}$. Notice that the function \eqref{655} is nothing but the product of the heat operator in a product space such as \eqref{654}. Since the \emph{Bessel operator} on the half-line $\mathbb{R}_y^+$
\begin{equation*}
    \mathscr{B}^{(a)}=\partial_{yy}+\frac{a}{y}\partial_y
\end{equation*}
is self-adjoint with respect to the measure $y^a\,dy$, and since from \eqref{655} we have
\begin{equation*}
    \int_{\Rn_x \times \mathbb{R}_y^+} G^{(a)}((x,w),t)y^a\,dy\,dx=\frac{\Gamma \left( \frac{a+1}{2} \right)}{2\pi^{\frac{a+1}{2}}},
\end{equation*}
we normalise $G^{(a)}((x,w),t)$ as follows
\begin{equation}
    \label{656}
    \mathscr{G}^{(a)}(x-z,y,t)=\frac{2\pi^{\frac{a+1}{2}}}{\Gamma \left( \frac{a+1}{2} \right)} (4 \pi t)^{-\frac{n+a+1}{2}}e^{-\frac{\abs{x-z}^2+y^2}{4t}}.
\end{equation}
In this way we have for every $x \in \Rn$ and $t>0$
\begin{equation}
    \label{657}
    \int_{\Rn_x \times \mathbb{R}_y^+} \mathscr{G}^{(a)}(x-z,y,t)y^a\,dy\,dz=1.
\end{equation}
For reasons that will soon be clear, along with the partial differential operator $\mathfrak{B}^{(a)}$ in \eqref{653} we ought to also consider its \emph{intertwined} operator in $\Rnn \times \mathbb{R}^+$
\begin{equation}
    \label{658}
    \mathfrak{B}^{(2-a)}U=\partial_{yy}U + \frac{2-a}{y}\partial_y U + \Delta_x U - \partial_t U,
\end{equation}
whose heat kernel in $\mathbb{R}^{n+3-a}\times \mathbb{R}^+$ is given by
\begin{equation}
    \label{659}
    G^{(2-a)}((x,w),t)=(4 \pi t)^{-\frac{n+3-a}{2}}e^{-\abs{(x,w)}^2}4t=(4\pi t)^{-\frac{n+3-a}{2}}e^{-\frac{\abs{x}^2+y^2}{4t}}.
\end{equation}
The motivation for introducing $\eqref{658}$ is in the following intertwining equation for the Bessel operators
\begin{equation}
    \label{660}
    \mathscr{B}^{(a)}(y^{1-a}U)=y^{1-a}\mathscr{B}^{(2-a)}U,
\end{equation}
that the reader can easily verify. The equation \eqref{660} shows that $U$ solves $\mathscr{B}^{(2-a)}U=0$ if and only if $\mathscr{B}^{(a)}(y^{1-a}U)=0$. As a consequence, we have the corresponding intertwining relation
\begin{equation}
    \label{661}
    \mathfrak{B}^{(a)}(y^{1-a}U)=y^{1-a}\mathfrak{B}^{(2-a)}U.
\end{equation}
This lead us to introduce the following.
\begin{definition}
\label{d:649}
We define the Poisson kernel of the operator $\mathfrak{B}^{(a)}$ as the function 
\begin{equation*}
    \mathscr{P}^{(a)}(x-z,y,t)= \frac{1}{2^{1-a}\Gamma \left( \frac{1-a}{2}\right)}\frac{y^{1-a}}{t^{\frac{3-a}{2}}}e^{-\frac{y^2}{4t}}G(x-z,t).
\end{equation*}
\end{definition}
We note right-away that, since up to a constant function
\begin{equation*}
    (x,y,t) \to \frac{1}{t^{\frac{3-a}{2}}}G(x-z,t)
\end{equation*}
is for every fixed $x \in\Rn$ the heat kernel \eqref{659}, it is in particular a solution of the equation $\mathfrak{B}^{(2-a)}U=0$ in $\Rn_x \times \mathbb{R}_y^+ \times \mathbb{R}_t^+$. In view of \eqref{661} we deduce that for every $x \in \Rn$ the function
\begin{equation*}
    (z,y,t) \to \mathscr{P}^{(a)}(x-z,y,t)
\end{equation*}
is a solution of the equation
\begin{equation}
    \label{662}
    \mathfrak{B}^{(a)}\mathscr{P}^{(a)}(x-z,y,t)=0
\end{equation}
in $\Rn_x \times \mathbb{R}_y^+ \times \mathbb{R}_t^+$. Furthermore, we have the following.
\begin{proposition}
\label{p:650}
For every $x \in \Rn$ and $y>0$ we have
\begin{equation*}
    \int_{\Rn}\int_0^{\infty}\mathscr{P}^{(a)}(x-z,y,t)\,dz\,dt=1.
\end{equation*}
\end{proposition}
\begin{proof}
By Definition \ref{d:649} and the theorem of Tonelli we have
\begin{equation*}
    \begin{split}
        \int_{\Rn}\int_0^{\infty}& \mathscr{P}^{(a)}(x-z,y,t)\,dz\,dt= \frac{y^{1-a}}{2^{1-a}\Gamma \left(\frac{1-a}{2} \right)}\int_{\Rn}\int_0^{\infty}\frac{1}{t^{\frac{3-a}{2}}}e^{-\frac{y^2}{4t}}G(x-z,t)\,dz\,dt\\
        &= \frac{y^{1-a}}{2^{1-a}\Gamma \left(\frac{1-a}{2} \right)}  \int_0^{\infty}\left( \int_{\Rn} G(x-z,t)\,dz\right)\frac{1}{t^{\frac{3-a}{2}}}e^{-\frac{y^2}{4t}}\,dt\\
        &= \frac{y^{1-a}}{2^{1-a}\Gamma \left(\frac{1-a}{2} \right)}  \int_0^{\infty}\frac{1}{t^{\frac{3-a}{2}}}e^{-\frac{y^2}{4t}}\,dt=1,
    \end{split}
\end{equation*}
where the reader can easily verify the last equality by the change of variable $\sigma=\frac{y^2}{4t}$.
\end{proof}

With Definition \ref{d:649} in hands we can now solve problem \eqref{653}. The following is the main result of this section.
\begin{theorem}
\label{t:651}
Given $f\in \Snn$, consider the function defined by the equation
\begin{equation}
    \label{663}
    U(x,y,t)=\int_0^{\infty}\int_{\Rn}\mathscr{P}^{(a)}(x-z,y,\tau)f(z,t-\tau)\,dz\,d\tau.
\end{equation}
Then, $U\in C^{\infty}(\Rnn \times (0,\infty))$, and for any $1\le p \le \infty$ the function $U$ solves the extension problem in $L^p(\Rnn)$, in sense that we have $\mathfrak{B}^{(a)}U=0$ in $\Rnn \times (0,\infty)$, and moreover 
\begin{equation}
    \label{664}
    \lim_{y \to 0^+}\norma{U(\cdot,y,\cdot)-f}_{L^p(\Rnn)}=0.
\end{equation}
Furthermore, we also have in $L^p(\Rnn)$
\begin{equation}
    \label{665}
    -\frac{2^{-a}\Gamma \left( \frac{1-a}{2} \right)}{\Gamma \left( \frac{1+a}{2} \right)} \lim_{y \to 0^+}y^a\partial_y U(\cdot,y,\cdot)=(\partial_t - \Delta)^{\frac{\alpha}{2}}f.
\end{equation}
\end{theorem}
\begin{proof}
Using the Gaussian character of the function $\mathscr{P}^{(a)}$ in Definition \ref{d:649}, it is not difficult to justify differentiating under the integral sign in \eqref{663}. By iteration one thus proves that $U \in C^{\infty}(\Rnn \times (0,\infty))$. Furthermore, since by \eqref{663}
\begin{equation*}
    \mathfrak{B}^{(a)}U(x,y,t)= \int_0^{\infty}\int_{\Rn} \mathfrak{B}^{(a)}\mathscr{P}^{(a)}(x-z,y,\tau)f(z,t-\tau)\,dz\,d\tau
\end{equation*}
in view of \eqref{662} we see that
\begin{equation*}
    \mathfrak{B}^{(a)}U(x,y,t)=0,
\end{equation*}
in $\Rnn \times (0,\infty)$. Perhaps it is worth noting here that
\begin{equation*}
    \begin{split}
        \partial_t U(x,y,t)&=\int_0^{\infty}\int_{\Rn}\mathscr{P}^{(a)}(x-z,y,\tau)\partial_tf(z,t-\tau)\,dz\,d\tau\\
        &=-\int_0^{\infty}\int_{\Rn}\mathscr{P}^{(a)}(x-z,y,\tau)\partial_{\tau}f(z,t-\tau)\,dz\,d\tau\\
        &=\int_0^{\infty}\int_{\Rn}\partial_{\tau}\mathscr{P}^{(a)}(x-z,y,\tau)f(z,t-\tau)\,dz\,d\tau,
    \end{split}
\end{equation*}
where in the last equality we have integrated by parts. We are thus left with proving \eqref{664} and \eqref{665}. To reach this goal we make the crucial observation that $U$ can be written in the following form using the semigroup $P_{\tau}^H$
\begin{equation}
    \label{666}
    U(x,y,t)=\frac{1}{2^{1-a}\Gamma \left( \frac{1-a}{2} \right)}y^{1-a}\int_0^{\infty}\frac{1}{\tau^{\frac{3-a}{2}}}e^{-\frac{y^2}{4\tau}}P_{\tau}^Hf(x,t)\,d\tau.
\end{equation}
To recognise the validity of \eqref{666} we use \eqref{663} and Definition \ref{d:649} to find
\begin{equation*}
    \begin{split}
        U(x,y,t)&=\int_0^{\infty}\int_{\Rn}\mathscr{P}^{(a)}(x-z,y,\tau)f(z,t-\tau)\,dz\,d\tau\\
        &= \frac{1}{2^{1-a}\Gamma \left( \frac{1-a}{2} \right)}y^{1-a}\int_0^{\infty}\frac{1}{\tau^{\frac{3-a}{2}}}e^{-\frac{y^2}{4\tau}}\left(\int_{\Rn}G(x-z,t)f(z,t-\tau)\,dz \right)\,d\tau
    \end{split}
\end{equation*}
\begin{equation*}
    =\frac{1}{2^{1-a}\Gamma \left( \frac{1-a}{2} \right)}y^{1-a}\int_0^{\infty}\frac{1}{\tau^{\frac{3-a}{2}}}e^{-\frac{y^2}{4\tau}}P_{\tau}^Hf(x,t)\,d\tau,
\end{equation*}
which proves \eqref{666}. In view of Proposition \ref{p:650} we obtain from \eqref{666}
\begin{equation}
    \label{667}
    U(x,y,t)-f(x,t)=\frac{1}{2^{1-a}\Gamma \left( \frac{1-a}{2} \right)}y^{1-a}\int_0^{\infty}\frac{1}{\tau^{\frac{3-a}{2}}}e^{-\frac{y^2}{4\tau}}[P_{\tau}^Hf(x,t)-f(x,t)]\,d\tau.
\end{equation}
Using the representation \eqref{667} we thus find
\begin{equation*}
    \begin{split}
        &\norma{U(\cdot,y,\cdot)-f}_{L^p(\Rnn)} \\
        &\le \frac{1}{2^{1-a}\Gamma \left( \frac{1-a}{2} \right)}y^{1-a}\int_0^{1}\frac{1}{\tau^{\frac{3-a}{2}}}e^{-\frac{y^2}{4\tau}}\norma{P_{\tau}^Hf(x,t)-f(x,t)}_{L^p(\Rnn)}\,d\tau\\
        &+ \frac{1}{2^{1-a}\Gamma \left( \frac{1-a}{2} \right)}y^{1-a}\int_1^{\infty}\frac{1}{\tau^{\frac{3-a}{2}}}e^{-\frac{y^2}{4\tau}}\norma{P_{\tau}^Hf(x,t)-f(x,t)}_{L^p(\Rnn)}\,d\tau.
    \end{split}
\end{equation*}
In the second integral we use the contractivity of $P_\tau^H$ on $L^p(\Rnn)$ to bound
\begin{equation*}
    \frac{1}{\tau^{\frac{3-a}{2}}}e^{-\frac{y^2}{4\tau}}\norma{P_{\tau}^Hf(x,t)-f(x,t)}_{L^p(\Rnn)} \le 2 \norma{f}_{L^p(\Rnn)}\frac{1}{\tau^{\frac{3-a}{2}}}\in L^1(1,\infty),
\end{equation*}
since $\frac{3-a}{2}>1$. In the first integral, instead, we need to crucially use the rate in Lemma \ref{l:639}
\begin{equation*}
    \norma{P_{\tau}^Hf(x,t)-f(x,t)}_{L^p(\Rnn)} =O(\tau),
\end{equation*}
to estimate
\begin{equation*}
    \int_0^{1}\frac{1}{\tau^{\frac{3-a}{2}}}e^{-\frac{y^2}{4\tau}}\norma{P_{\tau}^Hf(x,t)-f(x,t)}_{L^p(\Rnn)}\,d\tau \le C \int_0^{1}\frac{1}{\tau^{\frac{1-a}{2}}}\,d\tau < \infty,
\end{equation*}
since $0 < \frac{1-a}{2} < 1$. In conclusion, the right-hand side in \eqref{667} goes to $0$ in $L^p(\Rnn)$ norm with $y^{1-a}$, and since $1-a>0$, we have proved \eqref{664}.

In order to complete the proof we are left with establishing \eqref{665}. Differentiating with respect to $y$ the representation formula \eqref{667}, we find
\begin{equation}
    \label{668}
    \begin{split}
        -\frac{2^{-a}\Gamma \left( \frac{1-a}{2}\right)}{\Gamma \left( \frac{1+a}{2}\right)}&y^a \partial_y U(x,y,t)\\
        &=-\frac{1-a}{2 \Gamma \left( \frac{1+a}{2} \right)} \int_0^{\infty}\frac{1}{\tau^{\frac{3-a}{2}}}e^{-\frac{y^2}{4\tau}}[P_{\tau}^Hf(x,t)-f(x,t)]\,d\tau\\
        &+ \frac{1}{4 \Gamma \left( \frac{1+a}{2} \right)}y^2 \int_0^{\infty}\frac{1}{\tau^{\frac{3-a}{2}}}e^{-\frac{y^2}{4\tau}}[P_{\tau}^Hf(x,t)-f(x,t)]\frac{d\tau}{\tau}.
    \end{split}
\end{equation}
On the other hand, keeping the equation $a=1-\alpha$ in mind, we can express
\begin{equation}
    \label{669}
    (\partial_t -\Delta)^{\frac{\alpha}{2}}f(x,t)=- \frac{1-a}{2\Gamma \left(\frac{1+a}{2} \right)}\int_0^{\infty}\frac{1}{\tau^{\frac{3-a}{2}}}[P_{\tau}^Hf(x,t)-f(x,t)]\,d\tau.
\end{equation}
Subtracting \eqref{669} from \eqref{668} we thus find
\begin{equation*}
    \begin{split}
        &\norma*{-\frac{2^{-a}\Gamma \left( \frac{1-a}{2}\right)}{\Gamma \left( \frac{1+a}{2}\right)}y^a \partial_y U(\cdot,y,\cdot) -(\partial_t -\Delta)^{\frac{\alpha}{2}}f }_{L^p(\Rnn)}\\
        & \le \frac{1-a}{2 \Gamma \left( \frac{1+a}{2} \right)} \int_0^{\infty}\frac{1}{\tau^{\frac{3-a}{2}}}\abs*{e^{-\frac{y^2}{4\tau}}-1}\norma{P_{\tau}^Hf(x,t)-f(x,t)}_{L^p(\Rnn)}\,d\tau\\
        &+ \frac{1}{4 \Gamma \left( \frac{1+a}{2} \right)}y^2 \int_0^{\infty}\frac{1}{\tau^{\frac{3-a}{2}}}e^{-\frac{y^2}{4\tau}}\norma{P_{\tau}^Hf(x,t)-f(x,t)}_{L^p(\Rnn)}\frac{d\tau}{\tau}\\
        &= \mathbb{I}(y) + \mathbb{II}(y).
    \end{split}
\end{equation*}
To complete the proof of the theorem it suffices to show that both $\mathbb{I}(y),\mathbb{II}(y) \to 0$ as $y \to 0^+$. We handle $\mathbb{II}(y)$ as follows
\begin{equation*}
    \begin{split}
        \mathbb{II}(y) &\equiv y^2 \int_0^1 \frac{1}{\tau^{\frac{1-a}{2}}}e^{-\frac{y^2}{4\tau}}\frac{d\tau}{\tau}+y^2\int_1^{\infty} \frac{1}{\tau^{\frac{3-a}{2}}}\frac{d\tau}{\tau}\\
        &=O(y^{1+a}) \to 0 \quad \text{since } a\in (-1,1).
    \end{split}
\end{equation*}
For $\mathbb{I}(y)$ we consider the integrand
\begin{equation*}
    0\le g_y(\tau) := \frac{1}{\tau^{\frac{3-a}{2}}}\abs*{e^{-\frac{y^2}{4\tau}}-1}\norma{P_{\tau}^Hf(x,t)-f(x,t)}_{L^p(\Rnn)}, \qquad 0<\tau <\infty. 
\end{equation*}
We clearly have $g_y(\tau)\to 0$ as $y \to 0^+$ for every $\tau>0$. On the other hand, there exist an absolute constant $C>0$ and a function $g\in L^1(0,\infty)$ such that $0\le g_y(\tau) \le C g(\tau)$ for every $\tau>0$. In fact, using Lemmas \ref{l:637} and \ref{l:639} it is not difficult to convince oneself that we can take
\begin{equation*}
    g(\tau)=
    \begin{cases}
    \frac{1}{\tau^{\frac{1-a}{2}}} \qquad 0 < \tau <1,\\
    \frac{1}{\tau^{\frac{3-a}{2}}} \qquad 1 < \tau <\infty.
    \end{cases}
\end{equation*}
By Lebesgue dominated convergence we conclude that $\mathbb{I}(y)\to 0$ as $y \to 0^+$.
\end{proof}

%%%%%%%%%%%%%%%%%
%%%%%CHAPTER 4%%%%%
%%%%%%%%%%%%%%%%%

%!TEX root = ../dissertation.tex

\chapter{Higher order}
\label{chp:4}
\section{Fractional operators of higher order}\label{s:foho}
In this section we want to introduce the fractional operators of higher order considered in the previous chapters, that is, the fractional Laplacean $(-\Delta)^s$, see also \citep{Ab}, and the fractional heat operator $(-H)^s$ for $s\in \mathbb{R}\setminus \mathbb{Z}$. In Chapter \ref{chp:2} we introduced in Definition \ref{d:fraclap} the fractional Laplacean with $s\in(0,1)$,
\begin{equation*}
        (-\Delta)^su(x)=\frac{\gamma(n,s)}{2}\int_{\mathbb{R}^n}\frac{2u(x)-u(x+y)-u(x-y)}{\abs{y}^{n+2s}}\,dy.
\end{equation*}
The limitation $s<1$ is due to the singularity of the kernel $\abs{y}^{n+2s}$ and, therefore, the same formula does not carry over to $s>1$. 

To introduce the Laplacean of higher order we start from its formulation according to Balakrishnan in Definition \ref{d:615} in which we take $s=\frac{\alpha}{2}$,
\begin{equation}\label{41a}
    (-\Delta)^{s}u(x)=-\frac{s}{\Gamma \left( 1- s\right)} \int_0^{\infty} \frac{1}{t^{1+s}}(P_tu(x)-u(x))\,dt,
\end{equation}
which is also written as 
\begin{equation}\label{41}
    (-\Delta)^{s}u(x)=-\frac{1}{\Gamma \left( 1- s\right)} \int_0^{\infty} \frac{1}{t^{s}}(-\partial_t)P_tu(x)\,dt.
\end{equation}
An advantage of the formulations \eqref{41a} and \eqref{41} is that they admit the following natural generalization to higher powers
\begin{definition}
\label{d:41}
For every $s > 0$ non-integer we write
\begin{equation*}
    s=k+\sigma, \qquad \text{where } k=[s]:=\max \{ d \in \mathbb{Z}:d <s\} \quad \text{ and } \quad \sigma \in [0,1)
\end{equation*}
and we define 
\begin{equation}
    \label{42}
    \begin{split}
    (-\Delta)^{s}u(x)&= (-\Delta)^{\sigma}\left((-\Delta)^ku\right)(x)\\
    &=-\frac{\sigma}{\Gamma \left( 1- \sigma \right)} \int_0^{\infty} \frac{1}{t^{1+\sigma}}(P_t(-\Delta)^ku(x)-(-\Delta)^ku(x))\,dt\\
    &=-\frac{1}{\Gamma \left( 1- \sigma\right)} \int_0^{\infty} \frac{1}{t^{\sigma}}(-\partial_t)^{k+1}P_tu(x)\,dt.
    \end{split}
\end{equation}
\end{definition}
This formula easily follows from one of the main properties of the heat semigroup; indeed, for every $k \in \mathbb{N}$ we have $(-\partial_t)^k P_t u(x)=(-\Delta)^k(P_t u)(x)=P_t(-\Delta)^ku(x)$, and from the Definition \ref{d:41} it follows that
\begin{equation*}
    (-\Delta)^su(x)=(-\Delta)^{\sigma}\left((-\Delta)^ku \right)(x)
\end{equation*}
where $(-\Delta)^{\sigma}$ is defined through \eqref{41}.\\
\\
Analogously, we can naturally define the fractional heat operator as
\begin{definition}
\label{d:43}
For every $s\ge 0$ we write
\begin{equation*}
    s=k+\sigma, \qquad \text{where } k=[s]:=\max \{ d \in \mathbb{Z}:d <s\} \quad \text{ and } \quad \sigma \in [0,1)
\end{equation*}
and we define 
\begin{equation}
    \label{43}
    \begin{split}
    (-H)^{s}u(x,t)&=(\partial_t-\Delta)^{s}u(x,t)= (\partial_t-\Delta)^{\sigma}\left((\partial_t-\Delta)^ku\right)(x,t) =(-1)^k(\partial_t-\Delta)^{\sigma}\left((H)^ku\right)(x,t)\\
    &=-(-1)^k\frac{\sigma}{\Gamma \left( 1- \sigma \right)} \int_0^{\infty} \frac{1}{\tau^{1+\sigma}}(P_{\tau}^HH^ku(x,t)-H^ku(x,t))\,d\tau\\
    %&=-(-1)^k\frac{1}{\Gamma \left( 1- \sigma\right)} \int_0^{\infty} \frac{1}{t^{\sigma}}(-\partial_t)^{k+1}P_tu(x)\,dt.
    \end{split}
\end{equation}
\end{definition}

%\begin{proposition}
%\label{p:41}
%If $s=k+\sigma \not \in \mathbb{N}$ the Definition \ref{d:41} agree with the following Balakrishnan type definition
%\begin{equation*}
%    (-\Delta)^s u(x)= \frac{(-1)^{k+1}\Gamma(\sigma + k +1)}{\Gamma(\sigma)\Gamma(1-\sigma)}\int_0^{\infty}t^{-\sigma -1-k}\left( P_tu(x)- \sum_{j=0}^k %\frac{t^j}{j!}\Delta^ju(x)\right)\,dt.
%\end{equation*}
%\end{proposition}
%\begin{proof}
%{\color{red} COMPLETARE LA DIMOSTRAZIONE} The fact that the two formulas agree can be checked by integration by parts and using the identity
%\begin{equation*}
%    \frac{\Gamma(\sigma + k +1)}{\Gamma(\sigma)}=\Pi_{j=0}^k(\sigma + j).
%\end{equation*}
%\end{proof}

\section{Extension problem of higher order}\label{s:epho}

The goal of this chapter is to generalize the Caffarelli-Silvestre extension problem to higher powers of the heat operator. In the time independent setting, Successful attempts in such direction appeared in\citep{CG}  (using conformal geometry techniques), see also \citep{Ya}, \citep{CY} and \citep{CM}.

In accordance with our work in Section \ref{s:teb}, and differently from the previously mentioned articles, we are going to focus on the heat counterpart of the extension problem \eqref{101}.

In preparation for the main result of this section, we recall that we write $H=\Delta-\partial_t$ for the heat operator, and define
\begin{equation*}
    \mathscr{H}_{a}:=\partial_{yy}+\frac{a}{y}\partial_y +H=\partial_{yy}+\frac{a}{y}\partial_y +\Delta - \partial_t.
\end{equation*}
Furthermore, the integer part of a real number $s$ is expressed as follows
\begin{equation*}
[s]:=\max \{d \in \mathbb{Z}: d < s\}.
\end{equation*}

In order to prove the Theorem \ref{t:our}, we first need to formalize the Poisson kernel for the equation \eqref{t1s}. Thereby, following the blueprint of section \ref{s:teb}, we define the following function:
\begin{definition}\label{d:pks}
for $s>0$ non-integer
\begin{equation}\label{pks}
\mathscr{P}^{(s)}(x-z,y,t):=\frac{1}{2^{2s}\Gamma(s)}\frac{y^{2s}}{t^{1+s}}e^{-\frac{y^2}{4t}}G(x-z,t)
\end{equation}
\end{definition}

The following is the main result of this section.

\begin{theorem}\label{t:our}
Let $s>0$ be some non-integer and $a=1-2(s-[s])$. Given $f\in \mathscr{S}(\Rnn)$, consider the function defined by the equation
\begin{equation*}
    U(x,y,t):=\int_0^{\infty}\int_{\Rn}\mathscr{P}^{(s)}(x-z,y,\tau)f(z,t-\tau)\,dz\,d\tau.
\end{equation*}
Then the function $U$ in $\Rnn\times (0,\infty)$ solves
\begin{equation}\label{t1s}
\mathscr{H}_{(a)}^{[s]+1}U(x,y,t) =\left( \partial_{yy} +\frac{a}{y}\partial_y +H \right)^{[s]+1}U(x,y,t)=0,
\end{equation}
and for any $1\le p \le \infty$ we have
\begin{equation}\label{bd1}
\lim_{y\to 0^+}\norma{U(\cdot,y,\cdot)-f}_{L^p(\mathbb{R}^{n+1})}=0.
\end{equation}
Moreover, for every odd integer $k \in \mathbb{N}$ such that $k \le [s]$, we have
\begin{equation}\label{bd2}
\lim_{y\to0}y^a\frac{\partial^{k}}{\partial y^{k}}U(x,y,t) =0.
\end{equation}
Furthermore, we also have in $L^p(\Rnn)$
\begin{equation}\label{tls}
(-H)^{\,s}f(x,t)=(\partial_t-\Delta)^{s}f(x,t)=K(s)\lim_{y \to0}y^a \partial_y \mathscr{H}_{(a)}^{[s]} U(x,y,t),
\end{equation}
where
\begin{equation*}
K(s):=-(-1)^{[s]}\frac{\Gamma(s)}{\Gamma(1+[s]-s)}\frac{2^{2(s-[s])-1}}{[s]!}.
\end{equation*}
\end{theorem}

\subsection{Proof of the Theorem \ref{t:our}}\label{su:1}

We want to prove that $\mathscr{H}_{(a)}^{[s]+1}\mathscr{P}^{(s)}=0$. Let us start with the computation of $\mathscr{H}_{(a)}\mathscr{P}^{(s)}$, and therefore:
\begin{equation}\label{mio1}
\frac{a}{y}\partial_y\mathscr{P}^{(s)}=\frac{2s(1-2(s-[s]))}{y^2}\left(\mathscr{P}^{(s)}-\mathscr{P}^{(s+1)} \right),
\end{equation}

\begin{equation}\label{mio2}
\partial_{yy}\mathscr{P}^{(s)}=\frac{4s^2-2s}{y^2}\mathscr{P}^{(s)}-\frac{8s^2+2s}{y^2}\mathscr{P}^{(s+1)} +\frac{4s^2+4s}{y^2}\mathscr{P}^{(s+2)}.
\end{equation}
Summing the \eqref{mio1} with \eqref{mio2} we have
\begin{equation*}
\left( \partial_{yy} +\frac{a}{y}\partial_y\right)\mathscr{P}^{(s)}=\frac{4s[s]}{y^2}\mathscr{P}^{(s)}-\frac{4s(1+s+[s])}{y^2}\mathscr{P}^{(s+1)} +\frac{4s(s+1)}{y^2}\mathscr{P}^{(s+2)}.
\end{equation*}
Now we compute the heat operator of $\mathscr{P}^{(s)}$:
\begin{equation}\label{belheat}
\begin{split}
H\mathscr{P}^{(s)}&=\frac{4s(1+s)}{y^2}\left(\mathscr{P}^{(s+1)}-\mathscr{P}^{(s+2)}\right).
\end{split}
\end{equation}
Thus, we easily obtain:
\begin{equation*}
\mathscr{H}_{(a)}\mathscr{P}^{(s)}=\left( \partial_{yy} +\frac{a}{y}\partial_y +H \right)\mathscr{P}^{(s)}=\frac{4s[s]}{y^2}\left( \mathscr{P}^{(s)}-\mathscr{P}^{(s+1)}\right).
\end{equation*}
Trivially then
\begin{equation}\label{formutile2}
\mathscr{H}_{(a)}\mathscr{P}^{(s-[s])}=\frac{4(s-[s])[s-[s]]}{y^2}\left( \mathscr{P}^{(s-[s])}-\mathscr{P}^{(s-[s]+1)}\right)=0.
\end{equation}
Our aim is then to compute $\mathscr{H}_{(a)}^{[s]}\mathscr{P}^{(s)}$, starting from the observation that:
\begin{equation}\label{connessione}
\mathscr{H}_{(a)}\mathscr{P}^{(s)}=\frac{[s]}{s-1}H\mathscr{P}^{(s-1)},
\end{equation}
and iterating
\begin{equation*}
\mathscr{H}_{(a)}^2\mathscr{P}^{(s)}=\frac{[s]([s]-1)}{(s-1)(s-2)}H^2\mathscr{P}^{(s-2)},
\end{equation*}

\begin{equation*}
\begin{split}
\mathscr{H}_{(a)}^3\mathscr{P}^{(s)}&=\frac{[s]([s]-1)([s]-2)}{(s-1)(s-2)(s-3)}H^3\mathscr{P}^{(s-3)}.
\end{split}
\end{equation*}

Thereby, for $k \in \{1,2,\dots,[s]\}$, we have:
\begin{equation*}
\mathscr{H}_{(a)}^k\mathscr{P}^{(s)}=\frac{[s]!}{([s]-k)!}\frac{\Gamma(s-k)}{\Gamma(s)}H^k\mathscr{P}^{(s-k)}.
\end{equation*}
Consequently, for $k=[s]$ we get:
\begin{equation}\label{formutile3}
\mathscr{H}_{(a)}^{[s]}\mathscr{P}^{(s)}=[s]!\frac{\Gamma(s-[s])}{\Gamma(s)}H^{[s]}\mathscr{P}^{(s-[s])}.
\end{equation}

Finally we can directly show that $\mathscr{H}_{(a)}^{[s]+1}\mathscr{P}^{(s)}=0$. Starting from \eqref{formutile3} and exploiting \eqref{formutile2} we have:
\begin{equation*}
\begin{split}
\mathscr{H}_{(a)}^{[s]+1}\mathscr{P}^{(s)}&=\mathscr{H}_{(a)}\left([s]!\frac{\Gamma(s-[s])}{\Gamma(s)}H^{[s]}\mathscr{P}^{(s-[s])}\right)\\
&=[s]!\frac{\Gamma(s-[s])}{\Gamma(s)}H^{[s]}\mathscr{H}_{(a)}\mathscr{P}^{(s-[s])}\\
&=0.
\end{split}
\end{equation*}
Thus we have shown that
\begin{equation}\label{primres}
\mathscr{H}_{(a)}^{[s]+1}\mathscr{P}^{(s)}=0.
\end{equation}
By virtue of the previous passages, we are ready to prove that:
\begin{equation*}
U(x,y,t)= \int_0^{\infty}\int_{\Rn}\mathscr{P}^{(s)}(x-z,y,\tau)f(z,t-\tau)\,dz\,d\tau
\end{equation*}
is solution of 
\begin{equation*}
\begin{cases} 
\mathscr{H}_{(a)}^{[s]+1}U=0 & \text{in $\mathbb{R}_+^{n+1}\times \mathbb{R}^+$} \\ 
U(x,0,t)=f(x,t) & \text{per $(x,t)\in \mathbb{R}^{n+1}$}.
\end{cases}
\end{equation*}
Indeed, we have the following relation:
\begin{equation*}
\mathscr{H}_{(a)}U(x,y,t)=\int_0^{\infty}\int_{\Rn}\mathscr{H}_{(a)}\mathscr{P}^{(s)}(x-z,y,\tau)f(z,t-\tau)\,dz\,d\tau,
\end{equation*}
which follows from:
\begin{equation*}
\begin{split}
\partial_t U(x,y,t)&=\int_0^{\infty}\int_{\Rn}\mathscr{P}^{(s)}(x-z,y,\tau)\partial_tf(z,t-\tau)\,dz\,d\tau\\
&=-\int_0^{\infty}\int_{\Rn}\mathscr{P}^{(s)}(x-z,y,\tau)\partial_{\tau}f(z,t-\tau)\,dz\,d\tau\\
&=\int_0^{\infty}\int_{\Rn}\partial_{\tau}\mathscr{P}^{(s)}(x-z,y,\tau)f(z,t-\tau)\,dz\,d\tau,
\end{split}
\end{equation*}
and consequently, hinging on \eqref{primres} we easily obtain:
\begin{equation*}
\mathscr{H}_{(a)}^{[s]+1}U(x,y,t)=\int_0^{\infty}\int_{\Rn}\mathscr{H}_{(a)}^{[s]+1}\mathscr{P}^{(s)}(x-z,y,\tau)f(z,t-\tau)\,dz\,d\tau=0.
\end{equation*}
Now we show the following Lemma:
\begin{lemma}\label{l:npk}
For every $x\in\Rn$, $y>0$ and $s>0$ we have
\begin{equation*}
\int_{\Rn}\int_0^{\infty}\mathscr{P}^{(s)}(x-z,y,t)\,dz\,dt=1.
\end{equation*}
\end{lemma}
\begin{proof}
It can be directly shown through:
\begin{equation*}
\begin{split}
\int_{\Rn}&\int_0^{\infty}\mathscr{P}^{(s)}(x-z,y,t)\,dz\,dt=\frac{y^{2s}}{2^{2s}\Gamma(s)}\int_0^{\infty}\int_{\Rn} \frac{1}{\tau^{1+s}}e^{-\frac{y^2}{4\tau}}G(x-z,t)\,dz\,d\tau\\
&=\frac{y^{2s}}{2^{2s}\Gamma(s)}\int_0^{\infty}\left( \int_{\Rn}G(x-z,t)\,dz\right)\frac{1}{\tau^{1+s}}e^{-\frac{y^2}{4\tau}}\,d\tau\\
&=\frac{y^{2s}}{2^{2s}\Gamma(s)}\int_0^{\infty}\frac{1}{\tau^{1+s}}e^{-\frac{y^2}{4\tau}}\,d\tau=1.
\end{split}
\end{equation*}
\end{proof}

Now we want to reformulate the function $U$ as:
\begin{equation}\label{newform}
\begin{split}
U(x,y,t)&=\int_0^{\infty}\int_{\Rn}\mathscr{P}^{(s)}(x-z,y,\tau)f(z,t-\tau)\,dz\,d\tau\\
&=\frac{y^{2s}}{2^{2s}\Gamma(s)}\int_0^{\infty} \frac{1}{\tau^{1+s}}e^{-\frac{y^2}{4\tau}} \left( \int_{\Rn} G(x-z,t)f(z,t-\tau)\,dz\right)\,d\tau\\
&=\frac{y^{2s}}{2^{2s}\Gamma(s)}\int_0^{\infty} \frac{1}{\tau^{1+s}}e^{-\frac{y^2}{4\tau}}P_{\tau}^Hf(x,t)\,d\tau.
\end{split}
\end{equation}
Thus, thanks to Lemma \eqref{l:npk}, we get:
\begin{equation}\label{pernorma}
U(x,y,t)-f(x,t)=\frac{y^{2s}}{2^{2s}\Gamma(s)}\int_0^{\infty} \frac{1}{\tau^{1+s}}e^{-\frac{y^2}{4\tau}}\left(P_{\tau}^Hf(x,t)-f(x,t)\right)\,d\tau.
\end{equation}
Through the use of \eqref{pernorma} we manage to prove \eqref{bd1}; indeed
\begin{equation*}
\begin{split}
\norma{U(\cdot,y,\cdot)&-f}_{L^p(\mathbb{R}^{n+1})} \le\\
&	\le \frac{y^{2s}}{2^{2s}\Gamma(s)}\int_0^{1} \frac{1}{\tau^{1+s}}e^{-\frac{y^2}{4\tau}}\norma{P_{\tau}^Hf-f}_{L^p(\mathbb{R}^{n+1})}\,d\tau\\ 
&+ \frac{y^{2s}}{2^{2s}\Gamma(s)}\int_1^{\infty} \frac{1}{\tau^{1+s}}e^{-\frac{y^2}{4\tau}}\norma{P_{\tau}^Hf-f}_{L^p(\mathbb{R}^{n+1})}\,d\tau.
\end{split}
\end{equation*}
For the second integral we exploit the contractivity of $P_{\tau}^H$
\begin{equation*}
\frac{1}{\tau^{1+s}}e^{-\frac{y^2}{4\tau}}\norma{P_{\tau}^Hf-f}_{L^p(\mathbb{R}^{n+1})}\le \frac{2}{\tau^{1+s}}e^{-\frac{y^2}{4\tau}}\norma{f}_{L^p(\mathbb{R}^{n+1})}\in L^1(1,\infty).
\end{equation*}
For the first integral we avail ourselves of the fact that $\tau \in (0,1)$
\begin{equation}\label{cruciale}
\norma{P_{\tau}^Hf-f}_{L^p(\mathbb{R}^{n+1})} = O(\tau).
\end{equation}
Hence, for the first integral we have
\begin{equation*}
\int_0^{1} \frac{1}{\tau^{1+s}}e^{-\frac{y^2}{4\tau}}\norma{P_{\tau}^Hf-f}_{L^p(\mathbb{R}^{n+1})}\,d\tau\le C\int_0^1 \frac{e^{-\frac{y^2}{4\tau}}}{\tau^s}\,d\tau, 
\end{equation*}
where the last integral with the change of variables $\omega= \frac{1}{\tau}$ we obtain
\begin{equation*}
C\int_0^1 \frac{e^{-\frac{y^2}{4\tau}}}{\tau^s}\,d\tau=C\int_1^{\infty}\omega^{s-2} e^{-\frac{y^2}{4}\omega}\,d\omega.
\end{equation*}
Now we call on formula 3.381 of \citep{GR80}, which claims that
\begin{equation*}
\int_u^{\infty}x^{\nu-1}e^{-\mu x}\,dx=\mu^{-\nu}\int_{\mu u}^{\infty}e^{-t}t^{\nu-1}\,dt \quad \text{per } u>0, \,\, \mathfrak{Re}\mu>0.
\end{equation*}
In our case, we have $\nu=s-1$, $u=1$ and $\mu=\frac{y^2}{4}$, and therefore
\begin{equation*}
C\int_1^{\infty}\omega^{s-2} e^{-\frac{y^2}{4}\omega}\,d\omega=C \frac{y^{2-2s}}{4^{1-s}}\int_{\frac{y^2}{4}}^{\infty}e^{-\omega}\omega^{s-1-1}\,d\omega \le C \frac{y^{2-2s}}{4^{1-s}}\int_{0}^{\infty}e^{-\omega}\omega^{s-1-1}\,d\omega= C \frac{\Gamma(s-1)}{4^{1-s}} y^{2-2s}.
\end{equation*}
Thus, for the first integral we get
\begin{equation*}
\begin{split}
\frac{y^{2s}}{2^{2s}\Gamma(s)}&\int_0^{1} \frac{1}{\tau^{1+s}}e^{-\frac{y^2}{4\tau}}\norma{P_{\tau}^Hf-f}_{L^p(\mathbb{R}^{n+1})}\,d\tau \le \frac{y^{2s}}{2^{2s}\Gamma(s)} C \frac{\Gamma(s-1)}{4^{1-s}} y^{2-2s}\\
&=\frac{C}{4}\frac{\Gamma(s-1)}{\Gamma(s)}y^2 \to 0 \quad \text{as } y \to 0^+,
\end{split}
\end{equation*}
and we have shown that
\begin{equation*}
\lim_{y\to0}\norma{U(\cdot,y,\cdot)-f}_{L^p(\mathbb{R}^{n+1})}=0.
\end{equation*}
Now we want to prove \eqref{bd2}, starting from the observation that
\begin{equation}\label{derivatay}
\partial_y \mathscr{P}^{(s)}=\frac{y}{2(s-1)}H\mathscr{P}^{(s-1)},
\end{equation}
and iterating
\begin{equation*}
\partial_{y}^3 \mathscr{P}^{(s)}=3\frac{y}{2^2}\frac{1}{(s-1)(s-2)}H^{2}\mathscr{P}^{(s-2)} + \left(\frac{y}{2}\right)^3\frac{1}{(s-1)(s-2)(s-3)}H^3\mathscr{P}^{(s-3)},
\end{equation*}
\begin{equation*}
\begin{split}
\partial_{y}^5 \mathscr{P}^{(s)}&=15\frac{y}{2^3}\frac{1}{(s-1)(s-2)(s-3)}H^{3}\mathscr{P}^{(s-3)}+ 10\frac{y^3}{2^4}\frac{1}{(s-1)(s-2)(s-3)(s-4)}H^{4}\mathscr{P}^{(s-4)} \\
&\qquad  + \left(\frac{y}{2}\right)^5\frac{1}{(s-1)(s-2)(s-3)(s-4)(s-5)}H^5\mathscr{P}^{(s-5)}.
\end{split}
\end{equation*}
Thereby, we obtain the following iterative formula for suitable constants $c_k(i)>0$ and with $k\in \{ 1, 2, \dots, [s]\}$ such that $k$ is odd
\begin{equation}\label{itedevy}
\partial_y^k\mathscr{P}^{(s)}=\sum_{i=0}^{\frac{k-1}{2}}c_k(i)\frac{y^{k-2i}}{2^{k-i}}\frac{\Gamma(s-k)}{\Gamma(s)}H^{k-i}\mathscr{P}^{(s-k+i)}.
\end{equation}

By means of \eqref{itedevy} we can write
\begin{equation*}
\begin{split}
\partial_y^k U(x,y,t)&=\partial_y^k\left(\int_0^{\infty}\int_{\Rn} \mathscr{P}^{(s)}(x-z,y,\tau) f(z,t-\tau)\,dz\,d\tau \right)\\
&=\int_0^{\infty}\int_{\Rn}\partial_y^k \mathscr{P}^{(s)}(x-z,y,\tau) f(z,t-\tau)\,dz\,d\tau \\
&=\sum_{i=0}^{\frac{k-1}{2}}c_k(i)\frac{y^{k-2i}}{2^{k-i}}\frac{\Gamma(s-k)}{\Gamma(s)}\int_0^{\infty}\int_{\Rn}H^{k-i} \mathscr{P}^{(s-k+i)}(x-z,y,\tau) f(z,t-\tau)\,dz\,d\tau\\
&=\sum_{i=0}^{\frac{k-1}{2}}c_k(i)\frac{y^{k-2i}}{2^{k-i}}\frac{\Gamma(s-k)}{\Gamma(s)}H^{k-i}\left(\int_0^{\infty}\int_{\Rn} \mathscr{P}^{(s-k+i)}(x-z,y,\tau) f(z,t-\tau)\,dz\,d\tau\right)\\
&=\sum_{i=0}^{\frac{k-1}{2}}c_k(i)\frac{y^{k-2i}}{2^{k-i}}\frac{\Gamma(s-k)}{\Gamma(s)}H^{k-i}\left(\frac{y^{2(s-k+i)}}{2^{2(s-k+i)}\Gamma(s-k+i)}\int_0^{\infty} \frac{1}{\tau^{1+s-k+i}}e^{-\frac{y^2}{4\tau}}P_{\tau}^Hf(x,t)\,d\tau \right)\\
&=\sum_{i=0}^{\frac{k-1}{2}}c_k(i)\frac{y^{k-2i}}{2^{k-i}}\frac{\Gamma(s-k)}{\Gamma(s)}\frac{y^{2(s-k+i)}}{2^{2(s-k+i)}\Gamma(s-k+i)}\int_0^{\infty} \frac{1}{\tau^{1+s-k+i}}e^{-\frac{y^2}{4\tau}}P_{\tau}^HH^{k-i}f(x,t)\,d\tau.
\end{split}
\end{equation*}
Then for 
\begin{equation*}
A_k(i,s)=\frac{c_k(i)}{2^{3k-3i-2s}}\frac{\Gamma(s-k)}{\Gamma(s)\Gamma(s-k+i)}>0,
\end{equation*}
we have
\begin{equation*}
\partial_y^k U(x,y,t)=\sum_{i=0}^{\frac{k-1}{2}}A_k(i,s)y^{2s-k}\int_0^{\infty} \frac{1}{\tau^{1+s-k+i}}e^{-\frac{y^2}{4\tau}}P_{\tau}^HH^{k-i}f(x,t)\,d\tau, 
\end{equation*}
and clearly
\begin{equation*}
y^a\partial_y^k U(x,y,t)=\sum_{i=0}^{\frac{k-1}{2}}A_k(i,s)y^{1+2[s]-k}\int_0^{\infty} \frac{1}{\tau^{1+s-k+i}}e^{-\frac{y^2}{4\tau}}P_{\tau}^HH^{k-i}f(x,t)\,d\tau .
\end{equation*}
Finally, we can compute
\begin{equation*}
\begin{split}
\norma*{y^a\partial_y^k U(\cdot,y,\cdot)}_{L^p(\mathbb{R}^{n+1})} &\le \sum_{i=0}^{\frac{k-1}{2}}A_k(i,s)y^{1+2[s]-k}\int_0^{\infty} \frac{1}{\tau^{1+s-k+i}}e^{-\frac{y^2}{4\tau}}\norma*{P_{\tau}^HH^{k-i}f}_{L^p(\mathbb{R}^{n+1})}\,d\tau \\
&\le \sum_{i=0}^{\frac{k-1}{2}}B_k(i,s)y^{1+2[s]-k}\int_0^{\infty} \frac{1}{\tau^{1+s-k+i}}e^{-\frac{y^2}{4\tau}}\,d\tau\\
&\le C_k(i,s)y^{1+2[s]+k-2s-2i} \to 0 \quad \text{ as } y\to 0.
\end{split}
\end{equation*}
As a matter of fact $k+1-2(i+s-[s])>0$ for every $i \in \{0,1,\dots,\frac{k-1}{2} \}$.\\
\\
Thus we have proved that \eqref{bd2} for $k$ odd such that $k\le[s]$.\\
\\
\\
Now we want to prove \eqref{tls}. Exploiting \eqref{newform} we obtain
\begin{equation*}
\begin{split}
H^{[s]}&\left(  \int_0^{\infty}\int_{\Rn} \mathscr{P}^{(s-[s])}(x-z,y,\tau) f(z,t-\tau)\,dz\,d\tau \right)\\
&\quad =H^{[s]}\left( \frac{y^{2(s-[s])}}{2^{2(s-[s])}\Gamma(s-[s])}\int_0^{\infty} \frac{1}{\tau^{1+s-[s]}}e^{-\frac{y^2}{4\tau}}P_{\tau}^Hf(x,t)\,d\tau \right)\\
&\quad =\frac{y^{2(s-[s])}}{2^{2(s-[s])}\Gamma(s-[s])}\int_0^{\infty} \frac{1}{\tau^{1+s-[s]}}e^{-\frac{y^2}{4\tau}}P_{\tau}^HH^{[s]}f(x,t)\,d\tau,
\end{split}
\end{equation*}
and consequently
\begin{equation}\label{calore4}
\begin{split}
\mathscr{H}_{(a)}^{[s]}U(x,y,t)&=\mathscr{H}_{(a)}^{[s]}\int_0^{\infty}\int_{\Rn}\mathscr{P}^{(s)}(x-z,y,\tau)f(z,t-\tau)\,dz\,d\tau\\
&=\int_0^{\infty}\int_{\Rn}\mathscr{H}_{(a)}^{[s]}\mathscr{P}^{(s)}(x-z,y,\tau)f(z,t-\tau)\,dz\,d\tau\\
&=[s]!\frac{\Gamma(s-[s])}{\Gamma(s)}\int_0^{\infty}\int_{\Rn}H^{[s]}\mathscr{P}^{(s-[s])}(x-z,y,\tau) f(z,t-\tau)\,dz\,d\tau\\
&=[s]!\frac{\Gamma(s-[s])}{\Gamma(s)}H^{[s]}\left(  \int_0^{\infty}\int_{\Rn} \mathscr{P}^{(s-[s])}(x-z,y,\tau) f(z,t-\tau)\,dz\,d\tau \right)\\
&=[s]!\frac{\Gamma(s-[s])}{\Gamma(s)}\frac{y^{2(s-[s])}}{2^{2(s-[s])}\Gamma(s-[s])}\int_0^{\infty} \frac{1}{\tau^{1+s-[s]}}e^{-\frac{y^2}{4\tau}}P_{\tau}^HH^{[s]}f(x,t)\,d\tau\\
&=[s]!\frac{y^{2(s-[s])}}{2^{2(s-[s])}\Gamma(s)}\int_0^{\infty} \frac{1}{\tau^{1+s-[s]}}e^{-\frac{y^2}{4\tau}}P_{\tau}^HH^{[s]}f(x,t)\,d\tau.
\end{split}
\end{equation}
Hence
\begin{equation}\label{calore5}
\begin{split}
&y^a\partial_y \mathscr{H}_{(a)}^{s}U(x,y,t)=\\
&\quad =y^a \partial_y \left([s]!\frac{\Gamma(s-[s])}{\Gamma(s)}H^{[s]}\left(  \int_0^{\infty}\int_{\Rn} \mathscr{P}^{(s-[s])}(x-z,y,\tau) f(z,t-\tau)\,dz\,d\tau \right)\right)\\
&\quad =y^a \partial_y \left([s]!\frac{y^{2(s-[s])}}{2^{2(s-[s])}\Gamma(s)}\int_0^{\infty} \frac{1}{\tau^{s+1-[s]}}e^{-\frac{y^2}{4\tau}}P_{\tau}^HH^{[s]}f(x,t)\,d\tau\right)\\
&\quad =\frac{(s-[s])[s]!}{2^{2(s-[s])-1}\Gamma(s)}\int_0^{\infty} \frac{1}{\tau^{1+s-[s]}}e^{-\frac{y^2}{4\tau}}P_{\tau}^HH^{[s]}f(x,t)\,d\tau\\
&\qquad - \frac{[s]!y^{2}}{2^{2(s-[s])+1}\Gamma(s)}\int_0^{\infty} \frac{1}{\tau^{2+s-[s]}}e^{-\frac{y^2}{4\tau}}P_{\tau}^HH^{[s]}f(x,t)\,d\tau.
\end{split}
\end{equation}

Now we observe that:
\begin{equation*}
\begin{split}
\partial_y \left( \mathscr{H}_{(a)}^{[s]}U(x,y,t)\right)&=\partial_y \left( \mathscr{H}_{(a)}^{[s]}U(x,y,t) - [s]!\frac{\Gamma(s-[s])}{\Gamma(s)}H^{[s]}f(x,t)\right)\\
&=\partial_y \left([s]!\frac{y^{2(s-[s])}}{2^{2(s-[s])}\Gamma(s)}\int_0^{\infty} \frac{1}{\tau^{1+s-[s]}}e^{-\frac{y^2}{4\tau}}\left(P_{\tau}^HH^{[s]}f(x,t)-H^{[s]}f(x,t)\right)\,d\tau\right),
\end{split}
\end{equation*}
and therefore
\begin{equation}\label{speranza2}
\begin{split}
y^a &\partial_y \mathscr{H}_{(a)}^{[s]}U(x,y,t)=y^a\partial_y\left( \mathscr{H}_{(a)}^{[s]}U(x,y,t) - [s]!\frac{\Gamma(s-[s])}{\Gamma(s)}H^{[s]}f(x,t)\right)\\
& =\frac{(s-[s])[s]!}{2^{2(s-[s])-1}\Gamma(s)}\int_0^{\infty} \frac{1}{\tau^{1+s-[s]}}e^{-\frac{y^2}{4\tau}}\left(P_{\tau}^HH^{[s]}f(x,t)-H^{[s]}f(x,t)\right)\,d\tau\\
&\qquad - \frac{[s]!y^{2}}{2^{2(s-[s])+1}\Gamma(s)}\int_0^{\infty} \frac{1}{\tau^{2+s-[s]}}e^{-\frac{y^2}{4\tau}}\left(P_{\tau}^HH^{[s]}f(x,t)-H^{[s]}f(x,t)\right)\,d\tau.
\end{split}
\end{equation}

Multiplying \eqref{speranza2} for the following constant
\begin{equation*}
K(s):=-(-1)^{[s]}\frac{\Gamma(s)}{\Gamma(1+[s]-s)}\frac{2^{2(s-[s])-1}}{[s]!},
\end{equation*}

 \eqref{speranza2} becomes:
\begin{equation}\label{speranza3}
\begin{split}
K(s)&y^a\partial_y \left( \mathscr{H}_{(a)}^{[s]}U(x,y,t)\right)\\
& =-(-1)^{[s]}\frac{s-[s]}{\Gamma(1+[s]-s)}\int_0^{\infty} \frac{1}{\tau^{1+s-[s]}}e^{-\frac{y^2}{4\tau}}\left(P_{\tau}^HH^{[s]}f(x,t)-H^{[s]}f(x,t)\right)\,d\tau\\
&\qquad + (-1)^{[s]}\frac{y^{2}}{4\Gamma(1+[s]-s)}\int_0^{\infty} \frac{1}{\tau^{2+s-[s]}}e^{-\frac{y^2}{4\tau}}\left(P_{\tau}^HH^{[s]}f(x,t)-H^{[s]}f(x,t)\right)\,d\tau.
\end{split}
\end{equation}

From the Definition \ref{d:43} of the fractional heat operator of higher order, for $k=[s]$ and $\sigma=s-[s]$ we have
\begin{equation*}
\begin{split}
(-H)^{s}f(x,t)&=(-1)^{[s]}(-H)^{s-[s]}(H^{[s]}f)(x,t)=\\
&=-(-1)^{[s]}\frac{s-[s]}{\Gamma \left( 1- s +[s] \right)} \int_0^{\infty} \frac{1}{\tau^{1+\sigma}}(P_{\tau}^HH^{[s]}f(x,t)-H^{[s]}f(x,t))\,d\tau,
\end{split}
\end{equation*}

and finally we can compute:
\begin{equation*}
\begin{split}
&\norma*{K(s)y^a\partial_y \left( \mathscr{H}_{(a)}^{[s]}U(\cdot,y,\cdot)\right)- (-H)^{s}f}_{L^p(\mathbb{R}^{n+1})}\le\\
&\quad \le \frac{s-[s]}{\Gamma(1-s+[s])}\int_0^{\infty} \frac{1}{\tau^{1+s-[s]}}\abs*{e^{-\frac{y^2}{4\tau}}-1}\norma*{P_{\tau}^HH^{[s]}f-H^{[s]}f}_{L^p(\mathbb{R}^{n+1})}\,d\tau\\
&\qquad +\frac{y^{2}}{4\Gamma(1+[s]-s)}\int_0^{\infty} \frac{1}{\tau^{2+s-[s]}}e^{-\frac{y^2}{4\tau}}\norma*{P_{\tau}^HH^{[s]}f-H^{[s]}f}_{L^p(\mathbb{R}^{n+1})}\,d\tau\\
&\quad= \mathbb{I}(y) + \mathbb{II}(y) .
\end{split}
\end{equation*}

Thus we are left with showing that $\mathbb{I}(y),\mathbb{II}(y)\to 0$ for $y\to 0$.

Indeed, for $\mathbb{II}(y)$ we have:
\begin{equation*}
\begin{split}
\mathbb{II}(y)&\sim y^2\int_0^{1} \frac{1}{\tau^{2+s-[s]}}e^{-\frac{y^2}{4\tau}}\norma*{P_{\tau}^HH^{[s]}f-H^{[s]}f}_{L^p(\mathbb{R}^{n+1})}\,d\tau \\
&\qquad+ y^2\int_1^{\infty} \frac{1}{\tau^{2+s-[s]}}e^{-\frac{y^2}{4\tau}}\norma*{P_{\tau}^HH^{[s]}f-H^{[s]}f}_{L^p(\mathbb{R}^{n+1})}\,d\tau,
\end{split}
\end{equation*}
where for the first integral we use the formula \eqref{cruciale} to obtain the following estimate
\begin{equation*}
 y^2\int_0^{1} \frac{1}{\tau^{2+s-[s]}}e^{-\frac{y^2}{4\tau}}\norma*{P_{\tau}^HH^{[s]}f-H^{[s]}f}_{L^p(\mathbb{R}^{n+1})}\,d\tau \le C_2 y^2\int_0^{1} \frac{1}{\tau^{1+s-[s]}}e^{-\frac{y^2}{4\tau}}\,d\tau.
\end{equation*}
Now, after the change of variables $\omega=\frac{1}{\tau}$ and using the formula 3.381 of \citep{GR80} we obtain
\begin{equation*}
C_2 y^2\int_0^{1} \frac{1}{\tau^{1+s-[s]}}e^{-\frac{y^2}{4\tau}}\,d\tau \le \frac{C_2\Gamma(s-[s])}{4^{[s]-s}} y^{2([s]-s+1)} \to 0 \quad \text{as }\, t \to 0,
\end{equation*}
since $[s]-s+1>0$. Instead, the second integral is finite for $2+s-[s]>1$ and therefore
\begin{equation*}
\mathbb{II}(y) \to 0 \quad \text{ as } y\to 0.
\end{equation*}

While for $\mathbb{I}(y)$ we consider:
\begin{equation*}
0\le g_y(\tau):=\frac{1}{\tau^{1+s-[s]}}\abs*{e^{-\frac{y^2}{4\tau}}-1}\norma*{P_{\tau}^HH^{[s]}f-H^{[s]}f}_{L^p(\mathbb{R}^{n+1})}.
\end{equation*}
Clearly, we have $g_y(\tau)\to0$ as $y\to0^+$ for every $\tau>0$. 
On the other hand, it exists a constant $C>0$ and a function $g\in L^1(0,\infty)$ such that $0\le g_y(\tau)\le Cg(\tau)$ for every $\tau >0$. Recalling that $1+s-[s]\in(1,2)$, it suffices to take
\begin{equation*}
g(\tau)=
\begin{cases}
\frac{1}{\tau^{s-[s]}} \qquad \quad \text{ per } 0<\tau\le 1\\
\frac{1}{\tau^{s-[s]+1}} \qquad \, \text{ per }  1<\tau<\infty,
\end{cases}
\end{equation*}
and from Lebesgue dominated convergence we conclude that $\mathbb{I}(y)\to0$ for $y\to 0^+$.

%%%%%%%%%%%%%%%%%
%%%%%CONCLUSION%%%%%
%%%%%%%%%%%%%%%%%

\chapter{Conclusion}
\label{chp:c}

Considering the proof of the Theorem \ref{t:our}, we would like to generalize our result to a broader family of hypoelliptic operators, namely the one introduced by H\"{o}rmander \citep{Ho67}
\begin{equation}
    \label{gt1}
    \mathscr{K}u:= \text{tr}\,(Q\nabla^2u)+ \braket{BX,\nabla u} -\partial_t u.
\end{equation}
It was proven by H\"{o}rmander that $\mathscr{K}$ is hypoelliptic if and only if the covariance matrix
\begin{equation}
    \label{gt2}
    K(t)=\frac{1}{t}\int_0^t e^{sB}Qe^{sB^*}\,ds
\end{equation}
is invertible, i.e., $\det K(t)>0$ for every $t>0$. In \eqref{gt1} $Q$ and $B$ are $N\times N$ matrices with real, constant coefficients, with $Q\ge 0, Q=Q^*$. We have denoted by $X$ the variable in $\mathbb{R}^N$, and by $A^*$ the transpose of a matrix $A$.

The class of operators \eqref{gt1} includes several examples of interest in analysis, physics and the applied sciences. The simplest one is of course the ubiquitous heat equation, corresponding to the nondegenerate case when $Q=\mathbb{I}_N$, $B=\mathbb{O}_N$. When $Q=\mathbb{I}_N$, $B=-\mathbb{I}_N$ one has the Ornstein-Uhlenbeck operator, which is of great interest in the probability literature. Another example is the degenerate Kolmogorov operator in $\mathbb{R}^{N+1}$ with $N=2n$, with the choices $Q= \begin{bmatrix}\mathbb{I}_n & \mathbb{O}_n \\ \mathbb{O}_n & \mathbb{O}_n \end{bmatrix}$, and $B= \begin{bmatrix}\mathbb{O}_n & \mathbb{O}_n \\ \mathbb{I}_n & \mathbb{O}_n \end{bmatrix}$, which arose in the seminal paper \citep{DPZ} on Brownian motion and the theory of gases.

Then, we would like to establish results analogous (at least on the formal level) to Chapter \ref{chp:4} for the nonlocal operators $(-\mathscr{K})^s$ for $s>0$. In particular, the case $0<s<1$ has already been proven by Garofalo and Tralli, as can be seen in \citep{GT21}.

%%%%%%%%%%%%%%%%%
%%%%%BIBLIOGRAPHY%%%%%
%%%%%%%%%%%%%%%%%

\phantomsection
\addcontentsline{toc}{chapter}{Bibliography}
\bibliographystyle{natbib}

%%%%%%%%%%%%%%%%%%%%
%%%%%ACKNOWLEDGMENTS%%%%%
%%%%%%%%%%%%%%%%%%%%

\clearpage
\phantomsection
\noindent
\vspace*{\fill}
\chapter*{Acknowledgments}
Innanzitutto, vorrei ringraziare il prof. Nicola Garofalo non solo per la sua gentilezza e disponibilità nell’accompagnarmi per tutta la stesura della tesi, ma anche per la sua capacità di trasmettere tutta la passione che ha per la matematica. Per me è stato di grande aiuto ed esempio.\\
\\
Ringrazio il prof. Giulio Tralli per il suo prezioso confronto.  \\
\\
Ora vorrei ringraziare tutti i miei amici, che mi sono stati accanto in ogni momento.\\ 
Inizio con il ringraziare Davide, che da 24 anni è più scarso di me a Risiko. Forse un giorno riuscirà a vincere anche lui.\\
Ringrazio la Calle e tutti i film della Marvel visti insieme, spero non finiscano mai. \\
Ringrazio Ida per il supporto emotivo ed in particolare per le birre bevute insieme a casa sua che ci hanno\\ accompagnato in mille discorsi senza fine.\\
Ringrazio Simone e la sua laurea in psicologia che mi hanno permesso di raccontare un sacco di nuove storie. Battute a parte, mi ritengo fortunato ad essermi seduto accanto a lui il primo giorno di magistrale.\\
Ringrazio Franz di aver condiviso con me la passione per gli aquiloni e il Sarrismo.\\
Ringrazio Soato, alla prossima festa a casa sua prometto che porterò gli occhiali. \\
\\
Ringrazio Elisa, che ha dato una svolta alla mia vita. Prima di incontrarla non pensavo che una persona potesse essere così centrale in ogni mia pianificazione del futuro, e mi sono reso conto che non ha senso preoccuparmi di ciò che sarà la vita, se alla fine lei sarà lì, accanto a me.  \\
\\
Un pensiero speciale non può che andare alla mia famiglia che ha costruito la strada per permettermi di raggiungere questo primo traguardo.\\                             Ringrazio mia madre, Alice, che mi ha trasmesso il suo modo di vivere a colori, l'allegria e il lato artistico. Devo a lei la mia parte stravagante, alla quale tengo molto.\\
Ringrazio mio padre, Andrea, perché tutte le volte in cui mi chiedo che uomo voglio diventare, lui rimane sempre la mia risposta. \\
\\
Ringrazio i miei nonni per essere stati il braccio destro dei miei genitori nel prendersi cura di me. Mi hanno dato tanto,  e questo lavoro è anche merito loro.

\end{document}